\documentclass{amsart}
\usepackage{amssymb,
enumitem,
mathrsfs,
hyperref,
amsmath,
upgreek,
subcaption,
comment,
multirow, 
bigstrut,
float
}
\usepackage{tikz}
\usetikzlibrary{arrows,backgrounds,patterns,shapes,calc,fit}

\usepackage{etoolbox}
\usepackage{ifthen}
\newcommand{\adddetails}[1]{\ifthenelse{\equal{#1}{0}}{\toggletrue{notes}}{\togglefalse{notes}}}
\newtoggle{notes}
\adddetails{1} 
\newcommand{\details}[1]{\iftoggle{notes}{\begin{quotation}\begin{small}{\color{blue}\textbf{Details.} #1}\end{small}\end{quotation}}{}}

\usepackage{graphicx}
\hypersetup{colorlinks=true}
\numberwithin{equation}{section}
\usepackage[T1]{fontenc}
\usepackage[shadow,loadshadowlibrary]{todonotes}

\newcommand{\NN}{\omega}
\newcommand{\RR}{\mathbb{R}}
\newcommand{\QQ}{\mathbb{Q}}
\newcommand{\pre}[2]{{}^{#1} #2}

\newcommand{\Can}{{}^{\omega}2}
\newcommand{\can}{{}^{<\omega}2}
\newcommand{\bpi}{\boldsymbol{\pi}}
\newcommand{\AAA}{\mathscr{A}}
\newcommand{\Sloc}{\mathbf{S}}

\newcommand{\pow}{\mathscr{P}}

\newcommand{\lh}{\operatorname{lh}}

\newcommand{\dom}{\operatorname{dom}}

\newcommand{\lev}{\operatorname{\lambda}}
\newcommand{\Lev}{\operatorname{Lev}}
\newcommand{\Aut}{\operatorname{Aut}}
\newcommand{\AutL}{\operatorname{Aut}^*}
\newcommand{\Iso}{\operatorname{Iso}}
\newcommand{\coi}{\operatorname{coi}}

\newcommand{\weight}{\operatorname{weight}}
\newcommand{\Wr}{\operatorname{Wr}}

\newcommand{\supp}{\operatorname{supp}}
\newcommand{\Fin}{\operatorname{Fin}}
\newcommand{\LF}{\operatorname{LF}}
\newcommand{\Max}{\operatorname{Max}}
\newcommand{\UM}{\operatorname{Wsp}}

\newcommand{\Sym}{\operatorname{Sym}}
\newcommand{\ran}{\operatorname{ran}}
\newcommand{\spl}[1]{\mathrm{split}(#1)}
\newcommand{\tr}{\mathcal{T}}

\newcommand{\cat}[1]{\mathbb{#1}}
\newcommand{\markdef}[1]{\textbf{#1}}
\newcommand{\func}[1]{\mathscr{#1}}
\renewcommand{\mid}{\, : \, }

\newenvironment{enumerate-(a)}{\begin{enumerate}[label={\upshape (\alph*)}, leftmargin=2pc]}{\end{enumerate}}
\newenvironment{enumerate-(a)-r}{\begin{enumerate}[label={\upshape (\alph*)}, leftmargin=2pc,resume]}{\end{enumerate}}
\newenvironment{enumerate-(a)-5}{\begin{enumerate}[label={\upshape (\alph*)}, leftmargin=2pc,start=5]}{\end{enumerate}}
\newenvironment{enumerate-(A)}{\begin{enumerate}[label={\upshape (\Alph*)}, leftmargin=2pc]}{\end{enumerate}}
\newenvironment{enumerate-(A)-r}{\begin{enumerate}[label={\upshape (\Alph*)}, leftmargin=2pc,resume]}{\end{enumerate}}
\newenvironment{enumerate-(i)}{\begin{enumerate}[label={\upshape (\roman*)}, leftmargin=2pc]}{\end{enumerate}}
\newenvironment{enumerate-(i)-r}{\begin{enumerate}[label={\upshape (\roman*)}, leftmargin=2pc,resume]}{\end{enumerate}}
\newenvironment{enumerate-(I)}{\begin{enumerate}[label={\upshape (\Roman*)}, leftmargin=2pc]}{\end{enumerate}}
\newenvironment{enumerate-(I)-r}{\begin{enumerate}[label={\upshape (\Roman*)}, leftmargin=2pc,resume]}{\end{enumerate}}
\newenvironment{enumerate-(1)}{\begin{enumerate}[label={\upshape (\arabic*)}, leftmargin=2pc]}{\end{enumerate}}
\newenvironment{enumerate-(1)-r}{\begin{enumerate}[label={\upshape (\arabic*)}, leftmargin=2pc,resume]}{\end{enumerate}}
\newenvironment{itemizenew}{\begin{itemize}[leftmargin=2pc]}{\end{itemize}}
\newenvironment{enumerate-(H1)}{\begin{enumerate}[label={\upshape (H\arabic*)}, leftmargin=2pc]}{\end{enumerate}}
\newenvironment{enumerate-(G1tilde)}{\begin{enumerate}[label={\upshape (\(\widetilde{\mathrm{G}}_\arabic*\))}, leftmargin=2pc]}{\end{enumerate}}
\newenvironment{enumerate-(G1)}{\begin{enumerate}[label={\upshape (G\arabic*)}, leftmargin=2pc]}{\end{enumerate}}
\newenvironment{enumerate-(L1)}{\begin{enumerate}[label={\upshape (L\arabic*)}, leftmargin=2pc]}{\end{enumerate}}
\newenvironment{enumerate-(L'1)}{\begin{enumerate}[label={\upshape (\(\mathrm{L}\arabic*'\))}, leftmargin=2pc]}{\end{enumerate}}
\newenvironment{enumerate-(H'1)}{\begin{enumerate}[label={\upshape (\(\mathrm{H}'\arabic*\))}, leftmargin=2pc]}{\end{enumerate}}
\newenvironment{enumerate-(U1)}{\begin{enumerate}[label={\upshape (U\arabic*)}, leftmargin=2pc]}{\end{enumerate}}
\newenvironment{enumerate-(U1tilde)}{\begin{enumerate}[label={\upshape (\(\widetilde{\mathrm{U}}_\arabic*\))}, leftmargin=2pc]}{\end{enumerate}}
\newenvironment{enumerate-(P1)}{\begin{enumerate}[label={\upshape (P\arabic*)}, leftmargin=2pc]}{\end{enumerate}}
\newenvironment{enumerate-(P1tilde)}{\begin{enumerate}[label={\upshape (\(\widetilde{\mathrm{P}}_\arabic*\))}, leftmargin=2pc]}{\end{enumerate}}

\newtheorem{theorem}{Theorem}[section]
\newtheorem{lemma}[theorem]{Lemma}
\newtheorem{corollary}[theorem]{Corollary}
\newtheorem{proposition}[theorem]{Proposition}

\newtheorem{claim}{Claim}[theorem]
\theoremstyle{definition}
\newtheorem{defin}[theorem]{Definition}

\newtheorem{problem}{Problem}
\theoremstyle{remark}
\newtheorem{remark}[theorem]{Remark}
\newtheorem{ex}[theorem]{Example}

\newtheoremstyle{colon}%
{}
{}
{}
{}
{\bfseries}
{:}
{ }
{}
\theoremstyle{colon}
\newtheorem{condition}{Condition}

\begin{document}

\title{Isometry groups of Polish ultrametric spaces}
\date{\today}
\author[R.~Camerlo]{Riccardo Camerlo}
\address{Dipartimento di matematica, Università di Genova, Via Dodecaneso 35, 16146 Genova --- Italy}
\email{riccardo.camerlo@unige.it}
\author[A.~Marcone]{Alberto Marcone}
\address{Dipartimento di scienze matematiche, informatiche e fisiche, Universit\`a di Udine, Via delle Scienze 208, 33100 Udine --- Italy}
\email{alberto.marcone@uniud.it}
\author[L.~Motto Ros]{Luca Motto Ros}
\address{Dipartimento di matematica \guillemotleft{Giuseppe Peano}\guillemotright, Universit\`a di Torino, Via Carlo Alberto 10, 10121 Torino --- Italy}
\email{luca.mottoros@unito.it}

\thanks{We warmly thank S.~Gao, W.~Kubi\'s, M.~Malicki, and S.~Solecki for useful discussions related to the content of this paper.
Camerlo was partially supported by the MUR excellence department project awarded to the Department of Mathematics of the University of Genoa, CUP D33C23001110001.
Marcone and Motto Ros were supported by the Italian PRIN 2022 ``Models, sets and classifications'', prot.\ 2022TECZJA.
The authors are members of INdAM-GNSAGA}

\begin{abstract}
We solve a long-standing open problem, formulated by Krasner in the 1950's, in the context of Polish (i.e.\ separable complete) ultrametric spaces by providing a characterization of their isometry groups using suitable forms of generalized wreath products of full permutation groups. Since our solution is developed in the finer context of topological (Polish) groups, it also solves a problem of Gao and Kechris from 2003.
Furthermore, we provide an exact correspondence between the isometry groups of Polish ultrametric spaces belonging to some natural subclasses and various kinds of generalized wreath products proposed in the literature by Hall, Holland, and Malicki.
\end{abstract}

\maketitle
\details{
\tableofcontents
}

\section{Introduction}

An \markdef{ultrametric} (or \markdef{non-Archimedean metric}) \( d \) on a set \( X \) is a metric satisfying the following strong form of the triangle inequality, for all \( x,y,z \in X\):
\[
d(x,z) \leq \max \{ d(x,y),d(y,z) \} .
\]
Ultrametric spaces have grown in importance in mathematics, where, besides descriptive set theory and general topology, they are used e.g.\ in number theory (\( p \)-adic analysis) and combinatorics (consider e.g.\ the distance given by the weight of the minimax path between vertices in an edge-weighted undirected graph).
Ultrametrics are widely considered in computer science too, where distances between words (such as the \( p \)-closed distance) are often ultrametrics, but also in other fields such as 
condensed matter physics~\cite{Parisi}, 
the theory of aperiodic solids~\cite{Ultrametricforphysicists}, the study of certain discrete models of turbulence of fluids~\cite{Physicalreview} or of disordered systems in equilibrium statistical mechanics~\cite{Parisi2,ParisiNobel}, data science~\cite{Murtagh}, taxonomy~\cite{NumericalEcology3}, and so on.

In the 1950's, the algebraic number theorist Marc Krasner posed the following problem~\cite{Kra,LemSmi}:%
\footnote{Surprisingly, a related problem posed by Pestov, asking for a characterization of all \emph{subgroups} of the isometry groups of ultrametric spaces, turned out to be much simpler than Problem~\ref{problem:Krasner}, and was solved by Lemin and Smirnov~\cite{LemSmi} already in 1986.}
\begin{problem}[Krasner, 1956] \label{problem:Krasner}
Characterize the isometry groups of ultrametric spaces.
\end{problem}
This is one of the oldest problems concerning ultrametric spaces, as it was proposed shortly after the concept of ultrametric was introduced by Krasner himself as an abstraction of the special properties of the \( p \)-adic metrics~\cite{Kra2}.
Although the problem has remained wide open for about 70 years, several partial results were obtained, each one involving some form of generalized wreath product of transitive permutation groups.
For example, Fe\u{\i}nberg~\cite{Fei,Fei2} showed that if \( X \) is a compact ultrametric space, then its isometry group is isomorphic to a (unrestricted) generalized wreath product, denoted by \( \Wr^{\Max}_{\delta \in \Delta} H_\delta \) in this work, as defined by Holland~\cite{holland1969}.
This result was later extended by Nosova and Fe\u{\i}nberg~\cite{FeiNos} to all \( T \)-complete (also called spherically complete) ultrametric spaces.
More recently, Malicki considered what he calls \( W \)-spaces, i.e.\ Polish ultrametric spaces which are locally non-rigid and exact,%
\footnote{See Section~\ref{sec:exactdistances} for the definition.} and showed that their isometry groups can again be characterized in terms of generalized wreath products, this time using a suitable variant of Holland's construction that in this paper is denoted by \( \Wr^{\UM}_{\delta \in \Delta} \Sym(N_\delta) \)~\cite{malick2014}.
Thus, Krasner's Problem~\ref{problem:Krasner}, which has a somewhat vague formulation, has always been intended as the request of isolating the right notion of generalized wreath product giving rise to the isometry groups of (Polish) ultrametric spaces.
We also remark that Krasner's problem and all the partial results mentioned so far are purely algebraic, as no attempt to define a group topology on generalized wreath products appears in the literature.

Problem~\ref{problem:Krasner} reappeared in 2003 in a slightly different context, namely, that of topological groups.
The isometry group \( \Iso(X) \) of any Polish (i.e.\ separable complete) metric space is a Polish group when equipped with the pointwise convergence topology. In their seminal paper~\cite{GaoKec}, Gao and Kechris proved that, indeed, every Polish group is of this form, up to topological group isomorphism; therefore, the class of isometry groups of Polish metric spaces and the class of Polish groups are the same.
This naturally started the search for similar characterizations for interesting subclasses of Polish metric spaces. 
To mention a few examples:
compact Polish groups are the isometry groups of compact metric spaces~\cite{GaoKec,Mel};
closed subgroups of \( \Sym(\omega) \) are the isometry groups of zero-dimensional locally compact Polish metric spaces~\cite{GaoKec}; 
locally compact Polish groups are the isometry groups of proper Polish metric spaces%
~\cite{GaoKec,MalSol}.
Some characterizations are expressed in terms of a specific group operator applied to a certain class of Polish groups: for example, the closed subgroups of groups of the form  \( \prod_{n \in \omega} (\Sym(\omega) \ltimes  G_n^\omega) \), for \( (G_n)_{n \in \omega} \) a sequence of locally compact Polish groups, characterize the isometry groups of locally compact (or, equivalently, \( \sigma \)-compact) Polish metric spaces~\cite{GaoKec}.

Along the same lines, Gao and Kechris posed the following problem~\cite[Problem 10.10]{GaoKec}:
\begin{problem}[Gao and Kechris, 2003] \label{problem:GaoKechris}
Characterize the isometry groups of Polish (or locally compact Polish) ultrametric spaces.
\end{problem}
This time, the problem must be intended in the framework of topological groups.
Very little was known in this direction.
It is easy to see that if \( X \) is a Polish ultrametric space, then \( \Iso(X) \) is isomorphic to a closed subgroup of the infinite symmetric group \( \Sym(\omega) \); thus \( \Iso(X) \) is isomorphic to the automorphism group \( \Aut(T) \) of a first-order countable structure \( T \), and Gao and Shao~\cite{GaoShao2011} showed that \( T \) can always be taken to be an \( R \)-tree.
However, not all closed subgroups of \( \Sym(\omega) \) are isomorphic to the isometry group of a Polish ultrametric space: for example, if \( \Iso( X )\) is nontrivial, then it must contain a nontrivial involution~\cite[Proposition 4.7]{GaoKec}; if moreover it is simple, then it must be isomorphic to either the whole \( \Sym(\omega) \) or to \( \mathbb{Z}_2 \)~\cite[Proposition 4.1]{MalSol}.

The main goal of this paper is to solve Krasner's Problem~\ref{problem:Krasner} for Polish ultrametric spaces (in the algebraic sense, as originally formulated, but also in the finer context of topological groups) and, simultaneously, to fully solve Gao-Kechris' Problem~\ref{problem:GaoKechris} for both arbitrary Polish ultrametric spaces and for the locally compact ones --- indeed, we will see that there is no difference between these two variants of the problem.
The main ingredients we use are the following.

First, in Section~\ref{sec:L-trees} we introduce the notion of \( L \)-tree, a simple combinatorial object that naturally generalizes the \( R \)-trees considered in~\cite{GaoShao2011}.

Then, in Section~\ref{sec:functors} we define various categorical full embeddings among the categories associated to the objects we are interested in: \( L \)-trees, Polish ultrametric spaces, and some natural subclasses of the latter, like uniformly discrete spaces and perfect locally compact spaces.%
\footnote{The connection between ultrametric spaces and (some kind of) trees is quite natural, and is at the base e.g.\ of \cite{Hughes2004}, where a few categorical equivalences between categories of ultrametric spaces (equipped with different kinds of morphisms) and categories of \( \RR \)-trees (not to be confused with \( R \)-trees) are presented.}
This already yields that all such classes give rise to the same collection of automorphism or isometry groups, up to (topological) group isomorphism.

In Section~\ref{sec:wreathproducts} we move instead to (unrestricted) generalized wreath products. 
Although generalized wreath products have been involved in each of the partial results towards the solution of Problem~\ref{problem:Krasner} previously mentioned, the variants of this algebraic construction that can be found in the literature fail to provide a complete answer. To overcome this, we introduce a very general approach to wreath products that we believe could pique the interest also of a group theory oriented reader. More in detail, we first present the known variants (namely, Hall's wreath product \( \Wr^{\Fin}_{\delta \in \Delta} H_\delta \) from~\cite{hall1962}, Holland's wreath product \( \Wr^{\Max}_{\delta \in \Delta} H_\delta \) from~\cite{holland1969}, and Malicki's wreath product \( \Wr^{\UM}_{\delta \in \Delta} H_\delta \)  from~\cite{malick2014}), 
in the unified framework of generalized wreath products over a global domain \( \Wr^S_{\delta \in \Delta} H_\delta\) (Section~\ref{subsec:classicalwreathproduct}). 
Then we consider a fourth option based on locally finite supports  \( \Wr^{\LF}_{\delta \in \Delta} H_\delta \), and show that, among the mentioned four variants, this is the only class giving rise to Polish groups when we equip the generalized wreath products with a very natural topology (see Proposition~\ref{prop:Polishtopologyfrommetric} and the discussion following it). Finally, we introduce two powerful generalizations of the above mentioned generalized wreath products, that is, generalized wreath products over local domains \( \Wr^{\Sloc}_{\delta \in \Delta} H_\delta \) (Section~\ref{subsec:wreathproductlocal}) and projective wreath products \( \Wr^{\Sloc,\bpi}_{\delta \in \Delta} H_\delta \) (Section~\ref{subsec:projectivewreathproduct}).

Combining these tools all together, in Section~\ref{sec:main} we obtain the desired characterization of isometry groups of (locally compact) Polish ultrametric spaces in terms of generalized wreath products of full permutation groups. 
We remark that the following result can be read in purely algebraic terms (i.e.\ omitting any topological consideration on the groups and on their isomorphisms), or in the finer setup of Polish groups.

\begin{theorem}\label{thm:intromain1}
Up to (topological) isomorphism, the following classes of groups are the same:
\begin{enumerate-(1)}
\item \label{thm:intromain1-1}
isometry groups of arbitrary Polish ultrametric spaces,
\item \label{thm:intromain1-2}
isometry groups of locally compact Polish ultrametric spaces,
\item \label{thm:intromain1-4}
generalized wreath products over local domains \( \Wr^{\Sloc}_{\delta \in \Delta} \Sym(N_\delta) \),
\item \label{thm:intromain1-5}
projective wreath products \( \Wr^{\Sloc,\bpi}_{\delta \in \Delta} \Sym(N_\delta) \),
\item \label{thm:intromain1-6}
projective wreath products of the form \( \Wr^{\LF,\bpi}_{\delta \in \Delta} \Sym(N_\delta) \),
\end{enumerate-(1)}
 where the underlying order \( \Delta \) of each wreath product is a countable \( L \)-tree, and \( \Sloc \) is a countable collection of local domains satisfying Holland's maximum condition.%
\end{theorem}

Our characterization is naturally tied to quite old and well-established group theoretic constructions, like generalized wreath products applied to full permutation groups, and it aligns with all previous partial results on the matter in the literature; therefore, we believe that Theorem~\ref{thm:intromain1} provides an adequate and satisfactory solution to both long-standing Problems~\ref{problem:Krasner} and~\ref{problem:GaoKechris} at once.
Moreover, our methods are flexible enough to allow for restricted forms of such
characterization, drawing exact correspondences between certain natural subclasses of Polish ultrametric spaces (like homogeneous or homogeneous discrete spaces, Urysohn spaces, and exact spaces), and various natural variants (both old and new) of generalized wreath products. 
The main results are summarized in the table from Appendix~\ref{sec:table}.
For example, all characterizations can be phrased using only the variant introduced in this paper, namely, the one based on locally finite supports. 
Or we can determine to what extent the generalized wreath products previously appeared in the literature can be used in this context:
\begin{itemizenew}
\item 
Hall's generalized wreath products \( \Wr^{\Fin}_{\delta \in \Delta} \Sym(N_\delta) \) characterize the isometry groups of \emph{homogeneous discrete} Polish ultrametric spaces.
\item 
Malicki's generalized wreath products \( \Wr^{\UM}_{\delta \in \Delta} \Sym(N_\delta) \) characterize the isometry groups of \emph{homogeneous} Polish ultrametric spaces if \( \Delta \) is linear, and of \emph{exact} Polish ultrametric spaces  if \( \Delta \) is an \( L \)-tree. The latter result strengthens~\cite[Theorem 4.13]{malick2014} because it dispenses from the local non-rigidity assumption, and provides its converse.
\item 
Holland's maximum condition, a distinguished feature of his version of generalized wreath products \( \Wr^{\Max}_{\delta \in \Delta} H_\delta \), turns out to be a crucial limiting condition for most of our arguments to work.
\end{itemizenew}

Along the way, we also provide results that are of independent interest in the context of group theory. For example, we show that, under mild conditions, projective wreath products and generalized wreath products over local domains give rise to the same class of groups (see Theorems~\ref{thm:proj->unchained} and~\ref{thm:generalcase}), and we identify natural conditions under which generalized wreath products over local domains can be realized as the more classical generalized wreath products over global domains (Proposition~\ref{prop:closedvsfull}).

Our neat presentation of isometry groups of Polish ultrametric spaces as generalized wreath products unveil their rich combinatorial structure and, together with the concepts and methods developed to study them, may unlock the solution to various key problems in the area concerning the properties of such groups and their actions. We collect some of these problems, together with a few additional results, in Section~\ref{sec:openproblems}.

\section{Preliminaries}

We assume some familiarity with classical descriptive set theory~\cite{Kechris1995}, and with the basics of category theory~\cite{MacLane}. 
In this section, we present only a few concepts and notations which are maybe less standard.

We let \( \mathrm{Card} \) denote the class of all cardinals. The first infinite cardinal \( \aleph_0 \) coincides with the ordinal \( \omega \), which is the order type of \( (\mathbb{N}, {\leq}) \) and is identified with it. 
Given a linear order \( (L,{\leq_L}) \), its coinitiality \( \coi(L) \) is the smallest \( \kappa \in \mathrm{Card} \) such that there is a sequence \( (\ell_\alpha)_{\alpha < \kappa}\) in \( L \) which is coinitial, i.e.\ such that for every \( \ell \in L \) there is \( \alpha < \kappa \) such that \( \ell_\alpha \leq_L \ell \). 
A subset \( C \subseteq L \) is convex if for every \( \ell \leq_L \ell' \leq_L \ell'' \), if \( \ell,\ell'' \in C \) then \( \ell' \in C \) as well. Among the convex subsets of \( L \) we find the intervals, for which we use standard notations like \( [\ell;+\infty)_L \), \( (\ell;\ell')_L \), and alike.
The reverse of a linear order \( L \) is denoted by \( L^* \); in particular, \( \omega^* \) is the linear order \( (\mathbb{N}, {\geq} ) \).
The product of two linear orders \( L \) and \( L' \) is denoted by \( L \cdot L' \), and consists of the cartesian product of \( L \) and \( L' \) ordered antilexicographically. Similarly, products of partial orders are always endowed with the antilexicographic order. 

Suppose that \( (Q,{\leq_Q}) \) is a partial order, and \( (A_q)_{q \in Q} \) is a family of nonempty sets. If \( x \in \prod_{q \in Q} A_q \), for every \( p \in Q \) we let \( x|_p = x \restriction \{ q \in Q \mid q \geq_Q p \} \in \prod_{q \geq_Q p} A_q  \) be the restriction of \( x \) to the cone above \( p \). If \( A_q \in \mathrm{Card} \) for every \( q \in Q \), we also let 
\( \supp(x) = \{ q \in Q \mid x(q) \neq 0 \} \)
be the support of \( x \).

Let \( \RR^+ = \{ r \in \RR \mid r> 0 \} \).
A Polish metric space is a metric space \( (X,d)\) such that \( d \) is complete and induces a second-countable topology. 
Its distance set is \( D = \{ d(x,x') \mid x,x' \in X \} \setminus \{ 0 \} \subseteq \RR^+ \). 
A metric space \( (X,d) \) is uniformly discrete if there is \( r \in \RR^+ \) such that \( B_{d}(x,r) = \{ x \} \) for every \( x \in X \), where \( B_d(x,r) \) is the sphere centered in \( x \) with radius \( r \). 
Equivalently, \( (X,d) \) is uniformly discrete if its distance set \( D \) is bounded away from \( 0 \), that is, \( \inf D > 0\).
Throughout the paper, we will tacitly use the known special properties of ultrametric spaces, such as the fact that every point in a sphere is a center of it, or the fact that if two spheres intersect, then one of them is contained into the other one. Some of the constructions concerning ultrametric spaces that will be used in this work are inspired by those appearing in~\cite{MotSurvey,Camerlo:2015ar}, so the interested reader might consult those papers for more information on the matter.

In this paper, several groups that were introduced as purely algebraic objects, such as the generalized wreath products, will be turned into topological groups by introducing natural topologies on them. To distinguish the two setups, we use \( \simeq \) to denote algebraic isomorphism, while \( G \cong H \) means that \( G \) and \( H \) are isomorphic as topological groups;
when speaking of topological isomorphism, we usually drop the adjective from the terminology. 
Moreover, unless otherwise stated (and with the notable exception of generalized wreath products), symmetric groups \( \Sym(N) \) over a nonempty set \( N \) will be endowed with the pointwise convergence topology with respect to the discrete topology on \( N \), and their subgroups \( H \leq \Sym(N) \) will be equipped with the relative topology.

Finally, recall that a functor \( \func{F} \) between categories \( \cat{A} \) and \( \cat{B} \) is faithful (respectively, full) if for every \( A,A' \in \cat{A} \), the restriction of \( \func{F} \) to the set of arrows between \( A \) and \( A' \) is injective (respectively, surjective onto the set of arrows between \( \func{F}(A) \) and \( \func{F}(A') \)). A (categorical) full embedding is a fully faithful functor which, moreover, is also injective on objects.

\section{\( L \)-trees} \label{sec:L-trees}

Let \( (L, {\leq_L}) \) be a linear order.

\begin{defin} \label{def:L-tree}
An \markdef{\( L \)-tree} is a triple \( \tr = (T,{\leq_T}, \lev_T) \) such that \( (T, {\leq_T} ) \) is a partial order, \( \lev_T \colon T \to L \) is a surjective map, and the following conditions are satisfied, for every \( t,t' \in T\):
\begin{enumerate-(1)}
\item \label{def:L-tree-1}
if \( t <_T t' \) then \( \lev_T(t) <_L \lev_T(t') \); 
\item \label{def:L-tree-2}
for every \( \ell \geq_L \lev_T(t) \) there exists a unique \( t'' \in T \), denoted by \( t|_\ell \), such that \( t'' \geq_T t \) and \( \lev_T(t'') = \ell \) (in particular \( t|_{\lev_T(t)} = t \));
\item \label{def:L-tree-3}
there exists \( \ell \geq_L \max \{ \lev_T(t), \lev_T(t') \} \) such that \( t|_\ell = t'|_\ell\);
\item \label{def:L-tree-4}
if \( t , t' \in T \) are \( \leq_T \)-incomparable, 
then the nonempty set 
\[ 
\{ \ell \geq_L \max \{ \lev_T(t), \lev_T(t') \} \mid t|_\ell \neq t'|_\ell \}
\] 
has a maximum in \( L \), denoted by \( \spl{t,t'} \).
\end{enumerate-(1)}
If condition~\ref{def:L-tree-4} is dropped, then \( \tr \) is called a \markdef{weak \( L \)-tree}.
The cardinality of a weak \( L \)-tree \( \tr \), denoted by \( |\tr| \), is the cardinality of its domain \( T \). 
\end{defin}

Notice that if the reverse of \( L \) is well-founded (or even just \( \omega+1 \) does not embed into \( L \)),
then every weak \( L \)-tree is an \( L \)-tree. Picture~\ref{fig:L-tree} depicts an \( L \)-tree, together with the main ingredients involved in its definition.

\begin{figure}
\centering
\begin{tikzpicture}
\draw [thick, <-] (-4,3) to (-4,-3);
\node at (-4,-3.5) {\( L \)};
\node at (0.5,-3.5) {\( T \)};
\draw [thick] (0,3) to [out=300,in=110] (0,1.5) to [out=290,in=35] (-0.1,1.2) to [out=210,in=90] (-1,-0.5) to [out=270,in=125] (-0.3,-2) to [out=310,in=90] (0.3,-3);
\draw [thick] (0,1.5) to [out=290,in=165] (0.2,1.2) to [out=345,in=90] (2,-1.3) to [out=270,in=90] (2.5,-3);
\draw [thick] (2,-1.3) to [out=270,in=90] (1.7,-3);
\draw [thick] (2,-1.3) to [out=270,in=90] (3.3,-3);
\draw [thick] (-1,-0.5) to [out=270,in=30] (-1.8,-1.5) to [out=210,in=90] (-2.5,-3);
\draw [thick] (-0.3,-2) to [out=310, in=90] (-1,-3);
\draw [red,dashed, thick] (3,1.2) to  (-4,1.2);
\node [red] at (-5,1.2) {\begin{small}\( \spl{t,t'} \)\end{small}};
\draw [fill,red] (-0.66,-1.5) circle [radius=0.05];
\node [red] at (-0.36,-1.5) {\begin{small}\( t \)\end{small}};
\draw [red,dashed, thick] (-0.66,-1.5) to  (-4,-1.5);
\node [red] at (-4.8,-1.5) {\begin{small}\( \lev_T(t) \)\end{small}};
\draw [fill,red] (-0.95,0) circle [radius=0.05];
\node [red] at (-0.65,0) {\begin{small}\( t|_\ell \)\end{small}};
\draw [red,dashed, thick] (-0.95,0) to  (-4,0);
\node [red] at (-4.45,0) {\begin{small}\( \ell \)\end{small}};
\draw [fill,red] (1.52,0.3) circle [radius=0.05];
\node [red] at (1.88,0.3) {\begin{small}\( t' \)\end{small}};
\end{tikzpicture}
\caption{A graphical representation of an arbitrary \( L \)-tree, highlighting in particular the splitting level \( \spl{t,t'} \) of the two \( \leq_T \)-incomparable nodes \( t \) and \( t' \).}
\label{fig:L-tree}
\end{figure}

Examples of \( L \)-trees are the usual descriptive set-theoretic trees \( T \subseteq {}^{< \omega} A \) for any set \(A\), once they are turned upside down. 
Indeed, it is enough to let \( L = n^* \), where \( n \leq \omega \) is the height of the tree, \( \leq_T \) be the order of \emph{reverse} inclusion (i.e.\ \( t \leq_T s \iff t \supseteq s \), for \( s,t \in T\)), and \( \lev_T \colon T \to n^* \) be the function assigning to each \( t \in T \) its length as a sequence. 
Other examples of \( L \)-trees are  set-theoretic trees of transfinite height and linear orders.

\begin{remark}
Our notion of \( L \)-tree can be regarded as a generalization of the \( R \)-trees introduced by Gao and Shao in~\cite[Definition 6.3]{GaoShao2011}.
Indeed, if \( R \subseteq \RR^+ \) has no maximum, then any \( R \)-tree is just a special instance of an \( L \)-tree for \( L = R \). If instead \( T \) is an \( R \)-tree for some \( R\) with a maximum, then \( T \) looks more like a forest than a tree, as it lacks condition~\ref{def:L-tree-3} from Definition~\ref{def:L-tree}.
Nevertheless, letting \( L = R \cup \{ \max R +1 \} \), the \( R \)-tree \( T \) can easily be turned, without changing its main features, into an \( L \)-tree by adding a root to it: in particular, this procedure preserves the automorphism group of the tree (up to isomorphism).
\end{remark}

Given a weak \( L \)-tree \( \tr \) and \( \ell \in L \), let
\[
\Lev_\ell(\tr) = \{ t \in T \mid \lev_T(t) = \ell \},
\]
which is always nonempty by surjectivity of \( \lev_T\).
Moreover, let
\[  
[\tr] = \{ b \in {}^L T \mid \lev_T(b(\ell)) = \ell  \text{ for every } \ell \text{ and } b(\ell) \leq_T b(\ell') \text{ for every } \ell \leq_L \ell' \}
\]
be the \markdef{body} of \( \tr \), and call its elements \markdef{branches} of \( \tr \). 
We say that \( \tr \) is \markdef{pruned} if
for every \( t \in T \) there is \( b \in [\tr] \) such that \( b(\lev_T(t)) = t\).

Given a (weak) \( L \)-tree \( \tr \) and \( \ell \in L \), we denote by \( \tr|_\ell \) the restriction of \( \tr \) to its elements \( t \in T \) with \( \lev_T(t) \geq_L \ell \). More precisely, \( \tr|_\ell \) is the (weak) \( L' \)-tree \( \tr' = (T', {\leq_{T'}}, \lev_{T'}) \) where \( L' = [\ell;+\infty)_L \), \( T' = \{ t \in T \mid \lev_T(t) \geq_L \ell \}\), and, for every \( t,t' \in T' \),  \( t \leq_{T'} t'\iff t \leq_T t' \) and \( \lev_{T'}(t) = \lev_T(t) \).

An \markdef{embedding} between two weak \( L \)-trees \( \tr_0 \) and \( \tr_1 \) is an injection \( f \colon T_0 \to T_1 \) such that \( \lev_{T_1}(f(t)) = \lev_{T_0}(t) \) and \( t \leq_{T_0} t' \iff f(t) \leq_{T_1} f(t') \), for all \( t,t' \in T_0 \). 
An \markdef{isomorphism} is a surjective embedding, and is an \markdef{automorphism} when \( \tr_0 = \tr_1 \). 
The group of automorphisms of a weak \( L \)-tree \( \tr \) is denoted by \( \Aut(\tr) \), and is equipped with the pointwise convergence topology with respect to the discrete topology on \( T \). 
Therefore, when \( T \) is countable the topological group \( \Aut(\tr) \) is a closed subgroup of \( \Sym(\omega) \), and therefore it is a Polish group.

Given a weak \( L \)-tree \( \tr \), let \( \sim \) be the equivalence relation on \( T \) given by
\[  
t \sim t' \iff f(t) = t' \text{ for some } f \in \Aut(\tr).
\]
We denote by \( [t]\) the \( \sim\)-equivalence class of \( t \in T \).
Notice that if \( t \sim t' \) then \( \lev_T(t) = \lev_T(t') \). 
Moreover, if \( f \in \Aut(\tr) \) witnesses \( t \sim t' \), then for every \( \ell \geq_L \lev_T(t) = \lev_T(t') \) we have \( f( t|_\ell) = t'|_\ell \), hence \( t|_\ell\sim t'|_\ell \).

The \markdef{condensed tree} \( \Delta(\tr) \) of \( \tr \) is the weak \( L \)-tree \( \Delta(\tr) = (\Delta(T), {\leq_{\Delta(T)}}, \lev_{\Delta(T)}) \) where \( \Delta(T)  = \{ [t] \mid t \in T \} \) is the \( \sim \)-quotient of \( T \), and for \( \delta , \delta' \in \Delta(T) \) we set \( \lev_{\Delta(T)}(\delta) = \lev(t) \) for some (equivalently, any) \( t \in \delta \) and \( \delta \leq_{\Delta(T)} \delta' \iff t \leq_T t' \) for some \( t \in \delta \) and \( t' \in \delta' \). Notice that \( \delta \leq_{\Delta(T)} \delta' \) if and only if \( t |_{\lev_{\Delta(T)}(\delta')} \in \delta' \) for some (equivalently, any) \( t \in \delta \).

If \( \tr \) is pruned, then so is its condensed tree \( \Delta(\tr) \). If \( L \) has countable coinitiality, then the converse holds as well by the following technical fact.

\begin{lemma} \label{lem:inhabitated}
Suppose that \(\coi(L) \leq \aleph_0 \), and let \( \tr \) be a weak \( L \)-tree.
Let \( b\in [\Delta (\tr)] \), \( \bar \ell \in L \) and fix \( t\in b( \bar \ell ) \).
Then there is \( b' \in [\tr] \) such that \( b'( \bar \ell )=t\) and \( b'(\ell )\in b(\ell ) \) for every \( \ell\in L \).
%
\end{lemma}

\begin{proof}
Assume first that \( L \) has no minimum, so let \( (\ell_n)_{n\in\omega } \) be a strictly \( \leq_L \)-decreasing sequence coinitial in \( L \) with \( \ell_0= \bar \ell \).
We define recursively a sequence \( (t'_n)_{n \in \omega} \) such that \( t'_{n+1} \leq_T 
 t'_n \) and \( t'_n \in b(\ell_n)  \) for every \( n \in \omega \), starting with \( t'_0=t \): setting \( b'(\ell) = t'_n |_\ell \) for some (equivalently, any) \( t'_n \) such that \( \ell \geq_L \lev_T(t'_n) \), we get that \( b' \in [\tr] \) and \( b' \) is as desired.
Suppose we have defined \( t'_n \) so that \( t'_n \in b(\ell_n)\).
Since \( b(\ell_{n+1}) \leq_{\Delta(T)} b(\ell_n) \), let \( t' \in b(\ell_{n+1})\) and \( t'' \in b(\ell_n) \) be such that \( t' \leq_T t'' \).
Let \( f \in \Aut(\tr) \) be such that \( f(t'') = t'_n \): then setting \( t'_{n+1} = f(t') \) we get that \( t'_{n+1} \in b(\ell_{n+1}) \) and \( t'_{n+1} \leq_T t'_n \), as desired.

Suppose \( L \) has a minimum \( \ell^* \).
If \( \bar \ell =\ell^* \) the result is obvious, so assume \( \bar \ell >_L\ell^* \).
Since \( b(\ell^*) \leq_{\Delta(T)} b( \bar \ell ) \), there are \(t^*\in b(\ell^*) \) and \( \bar t \in b( \bar \ell ) \) such that \( t^*\leq_T \bar t \): then it is enough to set \( b'(\ell) = f(t^*)|_\ell \) for every \( \ell \in L \), where \( f \in \Aut(\tr) \) is such that \( f( \bar t ) = t \).
\end{proof}

\details{
As the next example shows, the condensed tree \( \Delta(\tr) \) might fail to be an \( L \)-tree, even when \( \tr \) is an \( L \)-tree itself. 

\begin{ex}\label{ex:counterexampleDelta}
Let
\( L = \omega + \omega^* = \{ 0, 1, 2, \ldots, 2^*, 1^*, 0^* \}\) and consider the following \( L \)-tree \( \tr = (T,{\leq_T},\lev_T) \):
\begin{enumerate-(a)}
 \item    
\begin{multline*}
    T = \can \cup \{ t_\alpha^n \mid \alpha \in \Can, \alpha \text{ is eventually constant}, n \in \NN \}\\
    \cup \{ t_\alpha^\diamond \mid \alpha \in \Can, \alpha \text{ is eventually } 1 \};
\end{multline*}
\item  \( \lev_T(s) = \lh(s)^*\) when \( s \in \can \), \( \lev_T(t_\alpha^n) = n\), and \( \lev_T(t_\alpha^\diamond) = 0\);
\item 
\( \leq_T \) is defined by the following clauses:
\begin{itemizenew}
\item in \( \can \) the order \(\leq_T \) is reverse inclusion;
\item \( t_\alpha^n \leq_T t_{\alpha'}^{n'} \) if and only if \( \alpha = \alpha' \) and \( n \leq n'\);
\item \( t_\alpha^\diamond \leq_T t_{\alpha'}^{n} \) if and only if \( \alpha = \alpha' \) and \( n>0\);
\item \( t_\alpha^n \leq_T s \) if and only if \( s \subset \alpha \), and similarly \( t_\alpha^\diamond \leq_T s \) if and only if \( s \subset \alpha \).
\end{itemizenew}
\end{enumerate-(a)}
The condensed tree \( \Delta(\tr) \) of \( \tr \) consists of:
\begin{itemize}
\item a single \( \sim \)-class \( [0^n] \) at each level \( n^*\in\omega^* \);
\item two distinct \( \sim \)-classes \( [t_{0^{ \NN }}^n] \) and \( [t_{1^{ \NN }}^n] \), at each level \( n\in \omega \).
\end{itemize}
Any pair consisting of \( [t_{0^\NN}^n] \) and \( [t_{1^\NN}^{n'}] \) witnesses that \( \Delta(\tr) \) fails to satisfy condition \ref{def:L-tree-4} in Definition \ref{def:L-tree}.  
\end{ex}
}

On the other hand, the condensed tree \( \Delta(\tr) \) might fail to be an \( L \)-tree, even when \( \tr \) is an \( L \)-tree itself. 
This justifies the following definition.

\begin{defin}
An \( L \)-tree \( \tr\) is \markdef{special} if its condensed tree \( \Delta(\tr) \) is still an \( L \)-tree.
\end{defin}

If \( L \) is finite, or even just a reverse well-order, then every \( L \)-tree is special. 
It can be further shown that every \( L \)-tree \( \tr \) can be turned into a special \( L' \)-tree \( \tr' \) such that \( |\tr'|=|\tr| \) and \( \Aut(\tr') \cong \Aut(\tr) \)
by possibly adding suitable ``intermediate levels'' to \( L \).\details{Vogliamo appuntarci i dettagli?}
If \( L \) is a countable linear order, this can also be obtained by composing the functors \( \func{G} \) and \( \func{F}\) from the next section.





\section{Some useful functors}\label{sec:functors}

In this section we consider several embeddings among various categories of \( L \)-trees and complete ultrametric spaces.

Given a linear order \( L \), let \( \cat{T}_L\) be the category of all \( L \)-trees with their embeddings, and \( \cat{T}_L^{\mathrm{pr}}\) be the induced subcategory consisting of all pruned \( L \)-trees. 
These categories are clearly determined by the isomorphism type of \( L \).
Indeed, if \( f \colon L \to L' \) is an isomorphism of linear orders, then each \( L \)-tree \( \tr \) can be transformed into a corresponding \( L' \)-tree \( \tr'\) by letting \( T = T' \), \( t \leq_{T'} t' \iff t \leq_T t' \), and \( \lev_{T'} = f \circ \lev_T \). 
This transformation yields a categorical isomorphism between \( \cat{T}_L \) and \( \cat{T}_{L'} \), which moreover preserves all relevant properties such as being pruned, being special, and so on. 
Notice also that if \( \tr' \in \cat{T}_{L'} \) is obtained from \( \tr \in \cat{T}_{L} \) using the transformation above, then \( \Aut(\tr') \cong \Aut(\tr) \).

For \( D \subseteq \RR^+\), let \( \cat{U}_D \) be the category whose objects are the nonempty complete ultrametric spaces with distance set contained in \( D \), and whose arrows are the distance-preserving functions (i.e.\ the isometric embeddings).
Let also \( \cat{U}_D^0 \) be the induced subcategory of \( \cat{U}_D \) of discrete spaces.



\subsection{The functor \(\func{U}\)} \label{subsec:U}
Throughout this subsection, we assume
\smallskip
\begin{center}
\fbox{\begin{minipage}{0.85\textwidth}
\begin{condition} \label{condition1}
\( D \subseteq D' \subseteq \RR^+ \) are such that  \( \inf D  > 0 \) and \( \inf D'  = 0 \).
\end{condition}
\end{minipage}}\end{center}
\smallskip

\noindent
We define the functor \( \func{U} \colon \cat{U}_D \to \cat{U}_{D'} \) as follows.

Let \( (r_n)_{n \in \omega} \) be a strictly decreasing coinitial sequence in \( D' \) such that \( r_0 < \inf D \).
Let \( \bar U \) be a \emph{strongly rigid} perfect compact Polish ultrametric space with distance set \( \{ r_n \mid n \in \omega \} \), where strongly rigid means that the only isometric embedding from \( \bar U \) into itself is the identity (and hence, in particular, \( \Iso(\bar U) \) is trivial).
For example, such a \( \bar U \) can be obtained by equipping the Cantor space \( {}^\omega 2 \) with the compatible complete ultrametric \( \bar{d} \) defined as follows. Let \( \{ s_n \mid n \in \omega \} \) be an enumeration without repetitions of \( {}^{< \omega } 2 \) such that if \( \lh(s_n) < \lh(s_m) \) then \( n < m \). 
Given distinct \( y,y' \in {}^\omega 2 \), let \( s_n \) be the longest common initial segment of \(y\) and \(y'\), and set \( \bar d (y,y')=r_n \). 
The Polish ultrametric space \( \bar U = ({}^\omega 2,\bar d) \) is perfect and compact because, topologically, it is the usual Cantor space. 
Moreover, it is also strongly rigid. 
Indeed, given distinct \( y,y' \in {}^\omega 2\), let \( s_n \) be the longest common initial segment of \(y\) and \(y'\). 
Then there is \( y'' \in {}^\omega 2 \) such that \( \bar d (y,y'') = r_n \), while this fails for \( y' \), which means that no isometric embedding of \( \bar U \) into itself can map \( y \) to \( y' \). 

Let \( \func{U}(U) \in \cat{U}_{D'} \) be the space obtained from \( U \in \cat{U}_D \) by replacing each of its points with a distinct copy of \( \bar U \). Formally, if \( d \) is the ultrametric on \( U \), then \( \func{U}(U) \) is obtained by endowing \( U \times \bar U\) with the complete ultrametric \( \hat d \) defined by 
\[  
\hat d((x,y),(x',y')) = 
\begin{cases}
d(x,x') & \text{if } x \neq x' \\
\bar{d}(y,y') & \text{if } x = x'.
\end{cases}
\]
Moreover, given an isometric embedding \( \psi \colon U \to U' \) with \( U,U' \in \cat{U}_D \), define \( \func{U}(\psi) \colon \func{U}(U) \to \func{U}(U') \) by setting \( \func{U}(\psi) (x,y) = (\psi(x),y) \), for all \( x \in U \) and \( y \in \bar U \).

\begin{theorem} \label{thm:toperfectlocallycompact}
The functor \( \func{U} \colon \cat{U}_D \to \cat{U}_{D'} \) is a categorical full embedding such that \( \func{U}(U) \) is a perfect locally compact space with \( \weight(\func{U}(U)) = \max \{ |U|, \aleph_0 \} \).
\end{theorem}

\begin{proof}
It is easy to see that, by construction, \( \func{U} (U)\in \cat{U}_{D'} \) is a nonempty perfect locally compact complete ultrametric space with weight 
\[ 
|U|\cdot \weight(\bar U)  = \max \{ |U|,\aleph_0 \} .
\]
Moreover, \( \func{U} \) is clearly a faithful functor which is injective on objects, so we only need to show that it is full. 

Let \( \varphi \colon \func{U}(U) \to \func{U}(U') \) be an arbitrary isometric embedding. Let \( d \), \( d' \), \( \hat d \), and \( \hat d' \) be the ultrametrics on \( U \), \( U' \), \( \func{U}(U) \), and \( \func{U}(U') \), respectively. 
\begin{claim} \label{claim:toperfectlocallycompact}
Let \( x \in U \) and \( y \in \bar U \). Then \(\varphi(x,y) = (x',y) \), for some \( x' \in U' \).
\end{claim}

\begin{proof}[Proof of the claim]
Let \( x ' \in U' \) and \( y' \in \bar U \) be such that \( \varphi(x, y) = (x',y') \). Since 
\( B_{\hat d}((x,y),\inf D ) = \{ x \} \times \bar U \) and \( B_{\hat d'}((x',y'),\inf D) = \{ x' \} \times \bar U\), it follows from the definition of \( \hat d \) and the strong rigidy of \( \bar U \) that the map sending \( z \in \bar U \) to the unique \( z' \in \bar U \) such that \( \varphi(x,z) = (x',z') \) is the identity. Therefore, \( y' = y \). 
\end{proof}

Now fix any \( \bar y \in \bar U \). Let \( \psi \colon U \to U' \) be defined by setting \( \psi(x) = x' \) for the unique \( x' \in U' \) such that \( \varphi(x,\bar y) = (x', \bar y) \).
Notice that the definition of \( \psi \) does not depend on the choice of \( \bar y \).
It is  easy to verify that \( \psi \) is an isometric embedding, and that \( \func{U}(\psi) = \varphi \).
\end{proof}

In particular, for any \( U \in \cat{U}_D \) the restriction of the fully faithful functor \( \func{U} \) to \( \Iso(U) \) witnesses \( \Iso(U) \simeq \Iso(\func{U}(U)) \); we now check that this isomorphism is also \emph{topological} when, as usual, both groups are equipped with the pointwise convergence topology.

\begin{corollary} \label{cor:toperfectlocallycompact}
For every \( U \in \cat{U}_D \), 
\[
\Iso(U) \cong \Iso(\func{U}(U)).
\]
\end{corollary}

\begin{proof}
Let \( d \) be the ultrametric on \( U \) and \( \hat d \) be the ultrametric on \( \func{U}(U) \).
A fundamental system of neighborhoods of the identity of \( \Iso(U) \) is given by the sets of the form 
\( V_{x,r} = \{ \psi \in \Iso(U) \mid d(\psi(x),x) < r \} \), for \( x \in U \) and \( 0 < r < \inf D \). By choice of \( r \), we indeed get that \( V_{x,r} = \{ \psi \in \Iso(U) \mid \psi(x) = x \} \), and moreover \( V_{x,r} = V_{x,r'} \) for all \( r,r' < \inf D \). On the other hand, a fundamental system of neighborhoods of the identity of \( \Iso(\func{U}(U)) \) is given by the sets of the form 
\( W_{x,y,r} = \{ \varphi \in \Iso(\func{U}(U)) \mid \hat d(\varphi(x,y),(x,y)) < r  \}  \), for \( x \in U \), \( y \in \bar U \), and \( 0 < r < \inf D \). 
By choice of \( r \), definition of \( \hat d\), and Claim~\ref{claim:toperfectlocallycompact}, we get \( W_{x,y,r} = \{ \varphi \in \Iso(\func{U}(U)) \mid \varphi(x,y) = (x,y)  \}  \), and moreover \( W_{x,y,r} = W_{x,y',r'} \) for all \( y,y' \in \bar U \) and \( r ,r' < \inf(D) \). Finally, by definition of \( \func{U}(\psi) \) we have that for every \( x \in U \) and \( r < \inf D \), \( \func{U}(V_{x,r}) = W_{x,y,r} \) for some (equivalently, any) \( y \in \bar U \). This concludes the proof.
\end{proof}

\subsection{The functor \(\func{F}\)} \label{subsec:F}
Throughout this subsection, we assume
\smallskip

\begin{center}
\fbox{\begin{minipage}{0.85\textwidth}
\begin{condition} \label{condition2}
\( D \subseteq L\subseteq \RR^+ \) are such that
\begin{itemizenew}
\item 
either \(D\) has no maximum or there exists \( \ell \in L \) with \( \ell > \max D \);
\item 
either \( \inf D = 0 \) or there exists \( \ell \in L\) with \( \ell \leq \inf D \).
\end{itemizenew}
\end{condition}
\end{minipage}}\end{center}
\smallskip

\noindent
We define the functor \( \func{F} \colon \cat{U}_D \to \cat{T}_{L}^{\mathrm{pr}} \) as follows.

Given \( U \in \cat{U}_D \), let \( d \) be its ultrametric.
Let \( \func{F}(U) \) be the triple \( \tr = (T, {\leq_T}, \lev_T) \) where:
\begin{itemizenew}
\item
\( T = \{ (B,\ell) \mid \ell \in L \text{ and } B=B_{d}(x,\ell) \text{ for some } x \in U \} \),
\item 
\( (B,\ell) \leq_T (B',\ell') \iff \ell \leq \ell' \text{ and } B \subseteq B' \), and
\item 
\( \lev_T \colon T \to L \) is the map sending each \( (B,\ell) \) to \( \ell \). 
\end{itemizenew}

\begin{remark} \label{12072025}
Notice that \( \func{F}(U)=\{ (B_d(x,\ell ),\ell )\mid x\in A_U \text{ and } \ell\in L \} \) for any dense \( A_U\subseteq U \).
This is useful to control the cardinality of \( \func{F}(U) \), as in Theorem~\ref{thm:spacestotrees}.
\end{remark}

\begin{lemma} \label{lem:spacestotrees}
For every \( U \in \cat{U}_D \), \( \func{F}(U) \) is a pruned \( L \)-tree.
\end{lemma}

\begin{proof}
We have to check that conditions~\ref{def:L-tree-1}--\ref{def:L-tree-4} of Definition~\ref{def:L-tree} are satisfied by the triple \( \tr = (T, {\leq_T}, \lev_T) \) defined as above.
Condition \ref{def:L-tree-1} holds by definition. As for \ref{def:L-tree-2}, given \( (B,\ell) \) and \( \ell' \geq \ell \), it is enough to set \( (B,\ell)|_{\ell'} = (B_d(x,\ell'),\ell')\) for some (equivalently, any) \( x \in B \).

Given now \( (B,\ell),(B',\ell') \in T \), let \( x,x' \) be such that \( B = B_{d}(x,\ell) \) and \( B' = B_{d}(x',\ell') \).
Without loss of generality we can assume \(\ell \leq \ell'\).
If \( d(x,x') < \ell' \), then \( (B,\ell ) \le_T (B',\ell') \).
If \( \ell' \leq d(x,x') \), then by the first assumption of Condition~\ref{condition2} we can find \( \ell''\in L \) such that \( \ell''>d(x,x') \): setting \( B'' = B_{d}(x,\ell'') \) we get that \( x,x' \in B'' \), and hence both \( (B'',\ell'') \geq_T (B,\ell) \) and \( (B'',\ell'') \geq_T (B',\ell') \) because \( \ell'' \geq \ell' \geq \ell \).
This shows that~\ref{def:L-tree-3} is satisfied as well.

Finally, assume that \( (B,\ell) \) and \( (B',\ell') \) are \( \leq_T \)-incomparable, so that \( B \cap B' = \emptyset \). 
Pick any \( x \in B \) and \( x' \in B' \), and set \( \bar \ell = d(x,x')  \in D \subseteq L \).
For every \( \ell'' \in L \), if \( \max \{ \ell,\ell' \} \leq \ell'' \leq \bar \ell \) then \( (B,\ell)|_{\ell''} = (B_{d}(x,\ell''),\ell'') \neq (B_{d}(x',\ell''),\ell'') = (B',\ell')|_{\ell''} \) because \( B_{d}(x,\ell'') \cap B_{d}(x',\ell'') = \emptyset \); on the other hand, if \( \ell'' > \bar \ell \), then \( (B,\ell)|_{\ell''} = (B_{d}(x,\ell''),\ell'') = (B_{d}(x',\ell''),\ell'') = (B',\ell')|_{\ell''} \). Thus~\ref{def:L-tree-4} is satisfied and \( \bar \ell = \spl{(B,\ell),(B',\ell')}\). 

Moreover, the \( L \)-tree \( \tr \) is pruned because if \( (B,\ell) \in T \), then for every \( x \in B \) the sequence \( b \in {}^{L} T \) defined by \( b(\ell') = (B_{d}(x,\ell'),\ell') \) for every \( \ell' \in L \) is an element of \( [\tr] \) such that \( b(\ell) = (B,\ell) \), as desired. 
\end{proof}

Let \( U,U' \in \cat{U}_D \) with ultrametrics \( d \) and \( d' \), respectively, and let \( \tr = \func{F}(U) \) and \( \tr'= \func{F}(U') \). For every isometric embedding \( \psi \colon U \to U' \), let \( \func{F}(\psi) \) be the map \( f \colon T \to T' \) defined by \( f(B,\ell) = (B_{d'}(\psi(x),\ell),\ell ) \) for some (equivalently, any) \( x \in B \). It is easy to check that \( f \) is a well-defined embedding between the pruned \( L \)-trees \( \tr \) and \( \tr' \).

\begin{theorem} \label{thm:spacestotrees} 
The functor \( \func{F} \colon \cat{U}_D \to \cat{T}_{L}^{\mathrm{pr}} \) is a categorical full embedding  such that \( |\func{F}(U)| \leq  \weight(U) \cdot |L| \) for every \( U \in \cat{U}_D \).
\end{theorem}

\begin{proof}
Let \( \tr = \func{F}(U) \in \cat{T}_{L}^{\mathrm{pr}} \).
Then \( |\tr| \leq \weight(U) \cdot |L| \) because by Remark \ref{12072025} the map \( (x,\ell) \mapsto (B_{d}(x,\ell),\ell) \) is a surjection from \( A_U \times L \) onto \( T \), for any dense set $A_U\subseteq U$. 

The functor \( \func{F}\) is clearly injective on objects.
We now check that it is also faithful. Fix \( U , U' \in \cat{U}_D \) with ultrametrics \( d \) and \( d' \), respectively, and let \( \psi_0,\psi_1 \colon U \to U' \) be distinct isometric embeddings. 
Since \( \psi_0 \neq \psi_1\), there is \( x \in U \) such that \( \psi_0(x) \neq \psi_1(x) \): let \( \ell = d'(\psi_0(x),\psi_1(x)) \in D \subseteq L \), and let \( B = B_{d}(x,\ell) \). Then \( \func{F}(\psi_0)(B,\ell) = (B_{d'}(\psi_0(x),\ell),\ell) \) and \( \func{F}(\psi_1)(B,\ell) = (B_{d'}(\psi_1(x),\ell),\ell) \), which implies that \( \func{F}(\psi_0)(B,\ell) \neq \func{F}(\psi_1)(B,\ell) \) because \( B_{d'}(\psi_0(x),\ell) \cap B_{d'}(\psi_1(x),\ell) = \emptyset\) by choice of \( \ell \). This shows that \( \func{F}(\psi_0) \neq \func{F}(\psi_1) \).

It remains to show that \( \func{F} \) is full. Let \( \tr = \func{F}(U) \) and \( \tr' = \func{F}(U') \) for some \( U,U' \in \cat{U}_D\), and let \( f \colon \tr \to \tr' \) be an embedding. Let \( d \) and \( d' \) be the ultrametrics of \( U \) and \( U' \), respectively. 
For every \( x \in U \), let \( b_x \in [\tr] \) be defined by \( b_x(\ell) = (B_{d}(x,\ell),\ell) \) for every \( \ell \in L \). 
Consider the branch \( b' \in [\tr'] \) defined by \( b'(\ell) =  f(b_x(\ell)) \), and let \( B'_\ell  \) be such that \( f(b_x(\ell)) = (B'_\ell,\ell) \). 
The intersection \( \bigcap_{\ell \in L} B'_\ell \) contains at most one point.
Indeed, suppose towards a contradiction that there are two distinct points \( y,y' \in \bigcap_{\ell \in L} B'_\ell \), and let \( \bar \ell = d'(y,y') \in D \subseteq L \): since \( B'_{\bar \ell} \) is a ball with radius \( \bar \ell \), it cannot contain both \( y \) and \( y' \), a contradiction. 
We claim that \( \bigcap_{\ell \in L } B'_\ell \) is also nonempty. We distinguish two cases. If \( \inf D = 0 \) we use the fact that the balls \( B'_\ell \) are clopen and \( d' \) is complete: in this case, \( \bigcap_{\ell \in L} B'_\ell \) contains the limit of the \( d' \)-Cauchy sequence given by the centers of the balls \( B'_{\ell_n} \), for any sequence \( (\ell_n)_{n \in \omega} \) coinitial in \( L \).
Suppose now that \( \inf D > 0 \), in which case there is some \( \ell \in L \) with \( \ell \leq \inf D \) by the second assumption of Condition~\ref{condition2}: then \( B'_\ell = \{ y \} \) is a singleton, and necessarily \( y \in \bigcap_{\ell \in L} B'_\ell \) because for each \( \ell' \leq \ell \) the ball \( B'_{\ell'} \) must be nonempty and contained in \( B'_\ell \).
We thus showed that in all cases \( \bigcap_{\ell \in L} B'_\ell = \{ y \} \) for some \( y \in U '\): set \( \psi(x) = y \). 

We claim that the map \( \psi \colon U \to U' \) just described is an isometric embedding.
Let \( x,x' \in U \) be distinct, and let \( \bar \ell = d(x,x') \in D \subseteq L \). Let \( b_x,b_{x'} \in [\tr] \) be the branches determined by \( x\) and \( x' \) as above. 
Notice that \( b_x(\ell) \neq  b_{x'}(\ell) \) if and only if  \( \ell \leq \bar \ell \).
Since \( f \) is an embedding, \( f (b_x(\ell)) = (B_{d'}(\psi(x),\ell),\ell) \), and \( f (b_{x'}(\ell)) = (B_{d'}(\psi(x'),\ell),\ell) \), we have that
\( (B_{d'}(\psi(x),\ell),\ell) \neq (B_{d'}(\psi(x'),\ell),\ell) \) if and only if  \( \ell \leq \bar \ell \).
Therefore, \( B_{d'}(\psi(x), \bar \ell) \cap B_{d'}(\psi(x'),\bar \ell) = \emptyset \), while \( B_{d'}(\psi(x), \ell) = B_{d'}(\psi(x'),\ell) \) for every \( \ell > \bar \ell \).
It follows that \( d'(\psi (x),\psi (x')) < \ell \) for every \( \ell \in L \) with \( \ell > \bar \ell \), but \( d'(\psi (x),\psi (x')) \geq \bar \ell \). But since \( U' \in \cat{U}_D \), the distance \( d'(\psi (x),\psi (x')) \) must belong to \( D \), and hence to \( L \): thus \( d'(\psi (x),\psi (x')) = \bar \ell = d(x,x') \), as desired.

Finally, we show that \( \func{F}(\psi) = f \). Fix any \( (B,\ell) \in T\) and pick \( x \in B \), so that \( (B,\ell) = b_x(\ell) \). By the way \( \psi \) is defined, \( \psi(x) \in B' \) where \( B' \) is such that \( f(B,\ell) = (B',\ell) \). But this implies that \( B' = B_{d'}(\psi(x),\ell) \), and thus \( \func{F}(\psi)(B,\ell) = (B',\ell) = f(B,\ell) \), as desired. 
\end{proof}


Analogously to what happened with the functor \( \func{U} \) from the previous subsection, also in this case for every \( U \in \cat{U}_D \) the restriction of \( \func{F} \) to \( \Iso(U) \) is a \emph{topological} group isomorphism between \( \Iso(U) \) and \( \Aut(\func{F}(U)) \).

\begin{corollary} \label{cor:spacestotrees}
For every \( U \in \cat{U}_D \),
\[  
\Iso(U) \cong \Aut(\func{F}(U)) .
\]
\end{corollary}

\begin{proof}
Let \( d \) be the ultrametric of \( U \). 
By choice of \( L \), the sets of the form \( V_{x,\ell} = \{ \psi \in \Iso(U) \mid d(\psi(x),x) < \ell \} \) for \( x \in U \) and \( \ell \in L \) form a fundamental system of neighborhoods of the identity of \( \Iso(U) \). 
On the other hand, since \( \tr =  \func{F}(U) \) is given the discrete topology, the sets of the form \( W_{B,\ell} = \{ f \in \Aut(\tr) \mid f (B,\ell) = (B,\ell) \} \) for \( (B,\ell) \in T \) form a fundamental system of neighborhoods of the identity of \( \Aut(\tr) \). 
The result thus follows from the fact that for \( B = B_d(x,\ell) \),
\[  
d(\psi(x),x) < \ell \iff \func{F}(\psi)(B,\ell) = (B,\ell),
\]
and hence \( \func{F}(V_{x,\ell}) = W_{B_d(x,\ell),\ell} \) for every \( x \in U \) and \( \ell \in L \).
\end{proof}

We conclude this subsection by discussing to what extent one can control extra properties of \( \func{F}(U) \) and \( L \).

In general, the \( L \)-trees produced by \( \func{F} \) are not necessarily special.\details{Un controesempio naturale è definire \( U \) tale che \( \func{F}(U) \) sia il \( \tr \) non special dell'Example~\ref{ex:counterexampleDelta}; si può fare con una variante dell'\( U \) dell'Example \ref{ex:notexact}}
However, for every \( U \in \cat{U}_D \) we can arrange the construction so that \( \func{F}(U) \) is special for the chosen \( U \).

\begin{lemma} \label{lem:spacestotreesspecial}
For every \( U \in \cat{U}_D\) there is an \( L \) satisfying  Condition~\ref{condition2} such that 
\( \func{F}(U) \) is special. Moreover:
\begin{itemizenew}
\item 
if \( \inf D > 0 \), then we can assume that \( L \) has a minimum, \( \min L = \inf D \), and \( |L| \leq  \weight(U)^2 + |D| + 2 \);
\item 
alternatively, if \( D \) is countable and \( U \) is separable we can assume that \( L \) has order type \( \eta \), the order type of \( \QQ \).
\end{itemizenew}
\end{lemma}

\begin{proof}
Let \( d \) be the ultrametric of \( U \) and fix a dense set \( A_U\subseteq U \) of cardinality \( \weight (U) \).
Let \( L' \) be obtained from \( D\) by adding to it \( \inf D \) if the latter is greater than \( 0 \), and \( 2 \max D\) if \( D \) has a maximum, so that \( |L'| \leq |D|+2 \). 
For each \( x,x' \in U \), let \( L'_{x,x'} \) be the set of those \( \ell \in L' \) such that there is no \( \psi \in \Iso(U) \) for which \( d(\psi(x),x') < \ell \), and observe that
 \( L'_{x,x'} \) is  downward closed in \( L' \). Set
\begin{equation} \label{13072025}
L = L' \cup \{ \sup L'_{x,x'} \mid x,x' \in U \text{ and } L'_{x,x'} \neq \emptyset \}.
\end{equation}
Then \( L \) has size \( |L| \leq  \weight(U)^2 + |D| + 2 \).
Indeed, if \( \sup L'_{x,x'}= \bar \ell \), \( d(x,y)< \bar \ell \), and \( d(x',y')< \bar \ell \), then \( L'_{y,y'}=L'_{x,x'} \); so in \eqref{13072025} the points \( x,x' \) can be taken ranging in a dense subset of \( U \).
Also, \( L \) satisfies Condition~\ref{condition2} because \( L \supseteq L' \). Moreover, by construction \( \inf L = \inf L' = \inf D \), and \( L \) has minimum \( \min L = \min L' = \inf D \) if \( \inf D > 0 \).

We need to show that if \( \func{F}\) is defined starting from such \( L \), then \( \tr = \func{F}(U) \) is special, i.e.\ \( \Delta(\tr) \) is an \( L \)-tree. 
The key point is that by the way the \emph{full} functor \( \func{F} \) is defined, 
for all \( \ell \in L \), \( (B,\ell) , (B',\ell ) \in T \), \( x \in B \), and \( x' \in B' \) we have
\begin{equation} \label{eq:simforultra}
(B,\ell) \sim (B',\ell)  \iff d(\psi(x),x') < \ell \text{ for some } \psi \in \Iso(U) ,
\end{equation}
and hence if \( \ell \in L' \) also
\begin{equation} \label{eq:simforultra-2}
(B,\ell) \sim (B',\ell) \iff \ell \notin L'_{x,x'}.
\end{equation}
Pick two \( \leq_{\Delta(T)} \)-incomparable \( [(B,\ell)], [(B',\ell')] \in \Delta(T) \) and two points \( x \in B \) and \( x' \in B' \). 
Notice that for every \( \ell'' \leq \min \{ \ell, \ell' \} \), we have that \( \spl{[(B,\ell)],[(B',\ell')]} \) exists if and only if \( \spl{[(B_d(x,\ell''),\ell'')],[(B_d(x',\ell''),\ell'')]} \) exists, and in that case they coincide.
Therefore, without loss of generality, we may assume that \( \ell = \ell' \), and that \( \ell \in L' \) (since \( L' \) is coinitial in \( L \)).
We need to show that there is a maximal \( \bar \ell \geq \ell \) such that \( (B,\ell)|_{\bar \ell} \not \sim (B',\ell )|_{\bar \ell} \), as in this case \( \bar \ell = \spl{[(B,\ell)],[(B',\ell)]} \). 
Since \( (B,\ell) \not\sim (B',\ell) \), we have \( \ell \in L'_{x,x'} \) by~\eqref{eq:simforultra-2}, and in particular \( L'_{x,x'} \neq \emptyset \). 
Let \( \bar \ell = \sup L'_{x,x'} \in L \), so that \( \bar \ell \geq \ell \).
Suppose first that \( \ell^*\in L \) is such that \( \ell \leq \ell^*< \bar \ell \). Then there is \( \ell'' \in L'_{x,x'} \) such that \( \ell^*\leq \ell'' \).
Therefore \( (B,\ell)|_{\ell''} \not\sim (B',\ell)|_{\ell''} \) by~\eqref{eq:simforultra-2}, and thus \( (B,\ell)|_{\ell^*} \not\sim (B',\ell)|_{\ell^*} \).
Now consider \( \bar \ell \). If \( \bar \ell \in L'_{x,x'}  \), then \( (B,\ell)|_{\bar \ell} \not\sim (B',\ell)|_{\bar \ell} \) by~\eqref{eq:simforultra-2}. If instead \( \bar \ell \notin L'_{x,x'} \), then there are \( \ell^*\in L'_{x,x'} \) which are arbitrarily close to \( \bar \ell \). Therefore, using~\eqref{eq:simforultra} we get that  \( (B,\ell)|_{\bar \ell} \sim (B',\ell)|_{\bar \ell} \) would imply \( (B,\ell)|_{\ell^*} \sim (B',\ell)|_{\ell^*} \) for some \( \ell^*\in L'_{x,x'}  \), contradicting~\eqref{eq:simforultra-2}. Therefore in all cases \( (B,\ell)|_{\bar \ell} \not\sim (B',\ell)|_{\bar \ell} \). Finally, assume that \( \ell^*\in L \) is such that \( \bar \ell < \ell^*\): we claim that \( (B,\ell)|_{\ell^*} \sim (B',\ell)|_{\ell^*} \). If \( \ell^*\in L' \), then \( \ell^*\notin L'_{x,x'} \) by choice of \( \bar \ell \), and therefore the claim follows again from~\eqref{eq:simforultra-2}. If instead \( \ell^*\notin L' \), then by definition of \( L \) it is the supremum of a set of elements of \( L' \). It follows that there is \( \ell'' \in L' \) such that \( \bar \ell < \ell'' < \ell^*\). As already shown, this implies that \( (B,\ell)|_{\ell''} \sim (B',\ell)|_{\ell''} \), and hence also \( (B,\ell)|_{\ell^*} \sim (B',\ell)|_{\ell^*} \). 
This proves that \( \bar \ell \) is as required.

The first item in the additional part is already achieved by the above construction. 
For the second item, instead, one can easily check that it is enough to start the construction with \( L' = D \cup \QQ \): since under our assumptions the resulting \( L \) will be a countable dense linear order with neither a minimum nor a maximum, its order type will be \( \eta \) by Cantor's theorem.
\end{proof}

As observed, the category \( \cat{T}_L^{\mathrm{pr}}\) is determined up to isomorphism by the isomorphism type of \( L \), therefore we get:

\begin{corollary} \label{cor:spacestotrees2}
If \( D \) is countable, then for every separable \( U \in \cat{U}_D \) there is a categorical full embedding \( \func{F}_{\QQ} \colon \cat{U}_D \to \cat{T}^{\mathrm{pr}}_{\QQ} \) such that \( \func{F}_{\QQ}(U) \) is countable and special, and moreover \( \Iso(U) \cong \Aut(\func{F}_{\QQ}(U)) \).
\end{corollary}

In particular, for every countable \( D \subseteq \RR^+ \) there is a categorical full embedding from the category of Polish ultrametric spaces with distance set contained in \( D \) into the category of countable pruned \( \QQ \)-trees, and for every Polish ultrametric space \( U \in \cat{U}_D \) fixed in advance we can further ensure that its associated \( \QQ \)-tree is special.



\subsection{The functor \(\func{G}\)} \label{subsec:G}

Throughout this subsection, we assume
\smallskip

\begin{center}
\fbox{\begin{minipage}{0.85\textwidth}
\begin{condition} \label{condition3}
\( L \) is a countable linear order such that either \( L \) has no minimum or else \( L \setminus \{ \min L \} \) has no minimum, and \( D \subseteq \RR^+ \) is such that \( \inf D>0  \) and \( 2 \cdot L\) embeds into \( D \).
\end{condition}
\end{minipage}}\end{center}

\smallskip

\noindent
Notice that since \( \inf D > 0 \), then \( \cat{U}_D = \cat{U}^0_D \), and \( \cat{U}^0_D \) actually consists of \emph{uniformly} discrete spaces.
We define the functor \( \func{G} \colon \cat{T}_L^{\mathrm{pr}} \to \cat{U}^0_{D}\) as follows.

Fix an embedding of \( 2 \cdot L \) into \( D \), and for every \( \ell \in L \) let \( \ell^- \) be the image of \( (0,\ell) \) and \( \ell^+ \) be the image of \( (1,\ell) \), so that \( \ell^- < \ell^+ \). 
Given \( \tr \in \cat{T}_L^{\mathrm{pr}}\), let \( \func{G}(\tr) \in \cat{U}_D \) be defined by equipping \( T \) with the distance function
\[  
d_T(t_0,t_1) = 
\begin{cases}
0 & \text{if } t_0 = t_1 \\
\lev_T(t_i)^- & \text{if } t_{1-i} <_T t_i \text{ for some } i \in \{0,1\} \\
\spl{t_0,t_1}^+ & \text{if } t_0 \text{ and } t_1 \text{ are \( \leq_T \)-incomparable}.
\end{cases}
\]
Obviously \( |\func{G}(\tr)| = |\tr| \) because we did not change the underlying set.
It is straightforward to verify that \( d_T \) is an ultrametric, and that \( d_T(t,t') \geq \lev_T(t)^- \) for all distinct \( t,t' \in T \), hence \( B_{d_T}(t,\lev_T(t)^-) = \{ t \} \) for every \( t \in T \). Moreover, \( (T,d_T) \) is uniformly discrete, and thus complete, because \( \inf D > 0 \). 

\begin{lemma} \label{lem:treestospaces}
For every \( \tr,\tr' \in \cat{T}_L^{\mathrm{pr}} \) and every \( \varphi \colon T \to T' \), the map \( \varphi \) is an embedding between the \( L \)-trees \( \tr \) and \( \tr' \) if and only if it is an isometric embedding between the complete ultrametric spaces \(\func{G}(\tr)\) and \(\func{G}(\tr')\).
\end{lemma}

\begin{proof} 
One direction is easy. If \( \varphi \) is an embedding of \( L \)-trees, then, in particular, \( \lev_{T'}(\varphi(t)) = \lev_T(t) \), and \( t_0,t_1 \in T \)  are \( \leq_T \)-incomparable if and only if \( \varphi(t_0), \varphi(t_1) \) are \( \leq_{T'}\)-incomparable too, in which case \( \spl{\varphi(t_0),\varphi(t_1)} = \spl{t_0,t_1} \). Thus \( \varphi \) is also an isometric embedding. 

Vice versa, suppose that \( d_{T'}(\varphi(t_0),\varphi(t_1)) = d_T(t_0,t_1) \) for all \( t_0,t_1 \in T \), so that, in particular, \( \varphi \) is injective.
We first show that \( \varphi \) preserves levels. 
Pick any \( t_0 \in T \). 
If \( \lev_T(t_0) = \min L \), then it realizes all distances in \( \{ \ell^- \mid \ell \in  L \setminus \{ \min L \} \} \), that is: 
for every \( \ell \in L \setminus \{ \min L \} \) there is \( t_1 \in T \) such that \( d_T(t_0,t_1) = \ell^- \) (take \( t_1 = t_0|_{\ell} \)). 
By Condition~\ref{condition3}, we have that \( L \setminus \{ \min L \} \) does not have a minimum (since \( L \) does by case assumption), hence the only nodes in \( T' \) realizing all distances in \( \{ \ell^- \mid \ell \in  L \setminus \{ \min L \} \} \) are those in \( \Lev_{\min L}(T') \): it follows that \( \lev_{T'}(\varphi(t_0)) = \min L = \lev_T(t_0) \).
Suppose now that \( \lev_T(t_0) \) is not the minimum of \( L \). Using again that \(L\) satisfies the first part of Condition~\ref{condition3}, there are infinitely many \( \ell \in L \) below \( \lev_T(t_0) \). 
Since \( \tr \) is pruned, this means that we can find \( t_1,t_2 \in T\)
such that \( t_2 <_T t_1 <_T t_0\). Since \( d_T(t_1,t_0) = d_T(t_2,t_0) = \lev_T(t_0)^- > \lev_T(t_1)^- = d_T(t_1,t_2) \),  
the points
\( \varphi(t_0) \), \( \varphi(t_1) \), and \( \varphi(t_2) \) are pairwise \( \leq_{T'} \)-comparable (otherwise two of them would be at distance \( \ell^+ \) for some \( \ell \in L \), which is impossible since \( \varphi \) is distance preserving), and \( \varphi(t_0) >_{T'} \varphi(t_1),\varphi(t_2) \)
because the distance of \( \varphi(t_0) \) from both \( \varphi(t_1) \) and \( \varphi(t_2) \) is greater than \( d_{T'}(\varphi(t_1),\varphi(t_2)) \).
Hence \( \lev_{T'}(\varphi(t_0))^- = d_{T'}(\varphi(t_0),\varphi(t_1)) = \lev_T(t_0)^-\), and therefore \( \lev_{T'}(\varphi(t_0)) = \lev_T(t_0) \).

It remains to show that \( t_0 <_T t_1 \iff \varphi(t_0) <_{T'} \varphi(t_1) \), for all \( t_0,t_1 \in T \). If \( t_0 <_T t_1 \), then \( \lev_T(t_0) < \lev_T(t_1) \) and \( d_T(t_0,t_1) = \lev_T(t_1)^- \): it follows that \( \varphi(t_0) \) and \( \varphi(t_1) \) are \( \leq_{T'} \)-comparable (because their distance is of the form \( \ell^- \), for some \( \ell \in L \)), and indeed \( \varphi(t_0) <_{T'} \varphi(t_1) \) because \( \lev_{T'}(\varphi(t_0)) = \lev_T(t_0) < \lev_T(t_1) = \lev_{T'}(\varphi(t_1)) \). The argument for the reverse implication is similar.
\end{proof}

In view of Lemma~\ref{lem:treestospaces}, it is natural to let \( \func{G} \) be the identity on arrows. This means that \( \func{G} \) is a fully faithful functor, and since it is obviously injective on objects we get:

\begin{theorem} \label{thm:treestospaces}
The functor \( \func{G} \colon \cat{T}_L^{\mathrm{pr}} \to \cat{U}^0_{D}\) is a categorical full embedding such that \( \func{G} ( \tr ) \) is uniformly discrete and \( |\func{G}(\tr) | = |\tr| \) (hence also \( \weight(\func{G}(\tr)) = |\tr| \)).
\end{theorem}

In particular, for every countable linear order \( L \) satisfying the first part of Condition~\ref{condition3} there is a categorical full embedding from the category of pruned \( L \)-trees into the category of uniformly discrete Polish ultrametric spaces, which combined with the functor \( \func{U} \) from Section~\ref{subsec:U} yields also a categorical full embedding from the category of pruned \( L \)-trees into the category of perfect locally compact Polish ultrametric spaces.
This can be extended to an arbitrary countable linear order \( L \) because, by the construction in the proof of~\ref{thm:exactdistancesnew-2} \( \Rightarrow \) \ref{thm:exactdistancesnew-1''} of Theorem~\ref{thm:exactdistancesnew}, there is a categorical full embedding from \( \cat{T}_L^{\mathrm{pr}} \) to \( \cat{T}_{L'}^{\mathrm{pr}} \), where \( L' = \omega^* + L \) is still countable and has no minimum.

Since the restriction of \( \func{G} \) to \( \Aut(\tr) \) is the identity, 
the underlying sets of \( \tr \) and \( \func{G}(\tr) \) are the same set \( T \), and \( d_T \) induces the discrete topology on it, we also have:

\begin{corollary} \label{cor:treestospaces}
For every \( \tr \in \cat{T}_L^{\mathrm{pr}} \),
\[  
\Aut(\tr) = \Iso(\func{G}(\tr)) .
\]
\end{corollary}

\section{Generalized wreath products} \label{sec:wreathproducts}

\subsection{The classical definition} \label{subsec:classicalwreathproduct}

Generalized wreath products were introduced by Hall~\cite{hall1962} (linear case) and Holland~\cite{holland1969} (general case) as powerful generalizations of the usual wreath product of two groups.
A further variant was introduced by Malicki~\cite{malick2014} to study isometry groups of Polish ultrametric spaces with certain special features, called \( W \)-spaces.
Here we present their definitions in a unified and quite general framework. 
Although our applications to isometry groups of Polish ultrametric spaces involve wreath products of full permutation groups, up to Section~\ref{subsec:projectivewreathproduct} we follow the literature and consider wreath products \( \Wr^{S}_{\delta \in \Delta} H_\delta \) of arbitrary transitive permutation groups \( H_\delta \).

Each variant of the generalized wreath product is determined by two ingredients:%
\footnote{In~\cite{malick2014}, the partial order \( \Delta \) is called the ``underlying plenary family'' of \( \Wr^S_{\delta \in \Delta} H_\delta \), while \( S \) is called the ``underlying set'' of \( \Wr^S_{\delta \in \Delta} H_\delta \).}
\begin{enumerate-(a)}
\item
A nonempty partially ordered set \( \Delta \), together with a labeling function \( N \colon \Delta \to \mathrm{Card} \colon \delta \mapsto N_\delta \) such that \( N_\delta \neq 0 \) for all \( \delta \in \Delta \).
The order of \( \Delta \) will be denoted by \( \leq_\Delta \), or by \( \leq \) when \( \Delta \) is clear from the context.
The labeled partial order \( \langle \Delta, N \rangle \) is called the \markdef{skeleton} of the wreath product. 
We say that the skeleton is \markdef{countable} if \( |\Delta| \leq \aleph_0 \) and \( N_\delta \leq \omega \) for all \( \delta \in \Delta \). It is \markdef{linear} when \( \Delta \) is a linear order.
\item \label{skeleton-c}
A set \( S \subseteq \prod_{\delta \in \Delta} N_\delta \), called (\markdef{global}) \markdef{domain}
of the wreath product, closed under pointwise perturbations, i.e.\ such that for every \( x \in S \), \( \delta \in \Delta \), and \( i \in N_\delta \), there is some \( x^\delta_i \in S\) satisfying \( x^\delta_i(\delta) = i \) and \( x^\delta_i(\gamma) = x(\gamma) \) for all \( \gamma > \delta \).
\end{enumerate-(a)}

Despite the chosen notation, in general the element \( x^\delta_i \in S  \) in condition~\ref{skeleton-c} might not be unique, although in many cases there are natural choices for it.
For example, Holland~\cite{holland1969} and Malicki~\cite{malick2014} set \( x^\delta_i(\gamma) = 0 \) and \( x^\delta_i (\gamma) = x(\gamma) \) for every \( \gamma \not\geq \delta \), respectively. 
As we will see (Definition~\ref{def:wreathproduct}), these variants are inessential, as the actual choice of \( x^\delta_i \) does not matter at all.

Given a skeleton \( \langle \Delta, N \rangle \) and a domain \( S \), the associated generalized wreath product \( \Wr^{S}_{\delta \in \Delta} \) is an operator that takes as input a family \( (H_\delta)_{\delta \in \Delta} \), where each \( H_\delta \subseteq \Sym(N_\delta) \) is a transitive permutation group on \( N_\delta \) that will be used on the \( \delta \)-th coordinate, and produces a group of permutations of the domain \( S \).
The following definition comes from~\cite[p.~160]{holland1969}, except that we allow arbitrary domains instead of the specific one considered by Holland.
See below for details.


\begin{defin} \label{def:wreathproduct}
Let \( \langle \Delta,N \rangle \) be a skeleton, \( S \subseteq \prod_{\delta \in \Delta} N_\delta \) be a domain, and let \( (H_\delta)_{\delta \in \Delta} \) be a family of transitive permutation groups over the corresponding sets \( N_\delta \).
The \markdef{generalized wreath product} 
\[ 
\Wr^{S}_{\delta \in \Delta} H_\delta
\] 
is the group of all permutations \( g \in \mathrm{Sym}(S) \) satisfying the following two conditions,%
\footnote{The letter ``H'' in the enumeration refers to Holland, who first introduced and studied wreath products of infinite families of groups over arbitrary partial orders.}
for all \( x,y \in S \) and \( \delta \in \Delta \):
\begin{enumerate-(H1)}
\item \label{Wr1}
\( x|_\delta = y|_\delta\) if and only if \( g(x)|_\delta = g(y)|_\delta \);
\item \label{Wr2}
the map \( i \mapsto g(x^\delta_i)(\delta)\) is a permutation of \( N_\delta\) belonging to \( H_\delta \).
\end{enumerate-(H1)}
Condition~\ref{Wr1} implies that~\ref{Wr2} is independent of the choice of the elements \( x^\delta_i \).
\end{defin}

\begin{remark} \label{rmk:permutationisautomatic}
Condition~\ref{Wr1} already yields that the map \( i\mapsto g(x_i^{\delta })(\delta ) \) considered in condition~\ref{Wr2} is a permutation of \( N_\delta \).
Therefore, if each \( H_\delta \) is the full permutation group \( \Sym(N_\delta) \), as it will be often the case in this paper, condition~\ref{Wr2} can be dropped. 
\end{remark}

In most cases, the domain \( S \) of the generalized wreath product is determined through a family \( \AAA \) of admissible supports.
More in detail, 
let \( \AAA \subseteq \pow(\Delta) \) be a nonempty family of sets such that if \( A \in \AAA \) and \( A' \) has finite symmetric difference from \( A \), then \( A'\in \AAA \) as well. 
Then
\[  
S^\AAA = \left\{ x \in \prod\nolimits_{\delta \in \Delta} N_\delta \mid \supp(x) \in \AAA \right\}
\]
is a domain; 
indeed, the condition on \( \AAA \) ensures that for every \( x \in S^\AAA \), \( \delta \in \Delta \), and \( i \in N_\delta \), there is a canonical choice for \( x^\delta_i \), namely, the one obtained by setting \( x^\delta_i(\gamma) = x(\gamma) \) for all \( \gamma \neq \delta \) (as in Malicki's~\cite{malick2014}).
Thus we can form the wreath product \( \Wr^{S^\AAA}_{\delta \in \Delta} H_\delta \), that for the sake of simplicity will also be denoted by \( \Wr^{\AAA}_{\delta \in \Delta} H_\delta \).

Hall's generalized wreath product from~\cite{hall1962} corresponds to the case where \( \Delta \) is linearly ordered and \( S = S^{\Fin} \), where \(\Fin\) is the collection of all finite subsets of \(\Delta\).
Notice that for an arbitrary \( \AAA \) as above, if $\emptyset \in \AAA$ then \(\Fin \subseteq \AAA \), so that Hall's choice for \( \AAA \) is somewhat minimal.

Holland~\cite{holland1969} considers instead the case of an arbitrary partial order \( \Delta \) with \( N_\delta > 1 \) for all \( \delta \in \Delta \), and \( S \) the set of all \( x \in \prod_{\delta \in \Delta} N_\delta \) satisfying the \emph{maximum condition}, that is, \( S = S^{\Max}\) for \( \Max \) the collection of all \( A \subseteq \Delta \) such that every nonempty subset of \( A \) has a maximal element or, equivalently, such that all strictly increasing chains in \( A \) are finite.

Finally, in~\cite{malick2014} Malicki introduced a further variant of Holland's construction by restricting the attention to supports \( A \in \Max \) all of whose infinite descending chains have no lower bound in \( \Delta \). 
Since this variant was specifically introduced to study the isometry groups of \( W \)-spaces, we will denote by \( \UM \) the collection of all supports \( A \subseteq \Delta \) with such property, and say that \( x \in \prod_{\delta \in \Delta} N_\delta \) satisfies the \emph{Malicki's condition} if \( \supp(x) \in \UM \).

\begin{remark} \label{rmk:wreathproduct}
In~\cite{malick2014}, condition~\ref{Wr1} is weakened by considering only the forward implication. This is justified by the fact that, by definition, all elements \( x \in S^{\UM} \) in the domain chosen by Malicki satisfy Holland's maximum condition. Together with~\ref{Wr2}, this entails that the forward implication in condition~\ref{Wr1} can be automatically reversed. 
Indeed, suppose that \( x|_{\delta} \neq y|_{\delta} \), and consider the set \( A = \{ \gamma \geq \delta \mid x(\gamma) \neq y(\gamma) \} \). Then \( A \) has a maximal element \( \bar \gamma \): if not, from any infinite strictly increasing sequence in \( A \) one could extract an infinite subsequence contained in either \( \supp(x) \) or \( \supp(y) \). By choice of \( \bar \gamma \), \( y \) is of the form \( x^{\bar \gamma}_i \) for \( i = y(\bar \gamma) \neq x(\bar \gamma) \). Therefore by condition~\ref{Wr2} we have \( g(y)(\bar \gamma) \neq g(x)(\bar \gamma)  \). Since \( \bar \gamma \geq \delta \), this means that \( g(x)|_\delta \neq g(y)|_\delta \), as desired.
Although we preferred to follow Holland's original formulation, in all definitions below we could instead follow Malicki, as the two approaches remain equivalent whenever all elements in the chosen domain satisfy the maximum condition, which will always be the case in this paper.
\end{remark}

Yet another variant is obtained by considering the collection \( \LF \) of locally finite supports, that is: \( A \in \LF \) if and only if \( A \cap \{ \gamma \in \Delta \mid \gamma \geq \delta \} \in \Fin \) for all \( \delta \in \Delta \).
Clearly, \( \Fin \subseteq \LF \subseteq \UM \subseteq \Max \).
Moreover, $\UM = \LF$ whenever \(\Delta\) has no infinite antichain with a lower bound: this happens in particular when \(\Delta\) is an $L$-tree, which is the case we consider in our main results. 



Notice that if \( \Delta = \{ \delta_0,\delta_1 \}\) with \( \delta_0 < \delta_1 \), so that necessarily \( \AAA = \Fin  = \pow(\Delta) \), then \( \Wr^{\AAA}_{\delta \in \Delta} H_\delta \) is the usual wreath product 
of the two groups \( H_{\delta_0} \) and \( H_{\delta_1} \).
On the other hand, if \( \Delta \) is an antichain and \( \AAA = \LF = \pow(\Delta) \), 
then \( \Wr^{\AAA}_{\delta \in \Delta} H_\delta \) is the direct product \( \prod_{\delta \in \Delta} H_\delta \).

Since the groups \( H_\delta \) are transitive, whenever \( \AAA \) is an ideal (i.e.\ it is closed under subsets and finite unions) \( \Wr^{\AAA}_{\delta \in \Delta} H_\delta \) is transitive too;  in particular, this happens when \( \AAA \in \{ \Fin, \LF, \UM, \Max \} \). 
However, \( \Wr^S_{\delta \in \Delta} H_\delta \) might fail to be transitive in the more general context of generalized wreath products over arbitrary domains.
\details{Se possibile con dominio indotto da supporti ammissibili. Esempio di \( S \) (chiuso e separabile, se possibile) per cui \( \Wr^S_{\delta \in \Delta} H_\delta \) non sia transitivo, ed esempio di \( S \) che sia approximately homogeneous ma \( \Wr^S_{\delta \in \Delta} H_\delta \) non transitivo. }

For our purposes, the following weaker form of homogeneity will suffice. 

\begin{defin} \label{def:approximatehomogeneity}
A domain \( S \subseteq \prod_{\delta \in \Delta} N_\delta \) is \markdef{approximately homogeneous} if for every \( x,y \in S \) and \( \delta \in \Delta \) there is \( g \in \Wr^S_{\delta \in \Delta} \Sym(N_\delta) \) such that \( g(x)|_\delta = y|_\delta \).    
\end{defin}

Since \( \Wr^{\LF}_{\delta \in \Delta} H_\delta \) is always transitive, all domains of the form \( S^{\LF} \) are approximately homogeneous.

There is a natural way to turn each generalized wreath product \( \Wr^{S}_{\delta \in \Delta} H_\delta \) into a topological group. First of all, we equip \( \prod_{\delta \in \Delta} N_\delta \) with the topology \( \tau_\Delta \) generated by
the sets of the form
\[
V_{x,\gamma} = \left\{ y \in \prod\nolimits_{\delta \in \Delta} N_\delta \mid y|_\gamma = x|_\gamma \right\},
\]
for \( x \in \prod_{\delta \in \Delta} N_\delta \) and \( \gamma \in \Delta \). 
Alternatively, \( \tau_\Delta \) can be presented as the pullback of the product topology on \( \prod_{\delta \in \Delta} \left( \prod_{\gamma \geq \delta} N_\gamma \right) \), where each \( \prod_{\gamma \geq \delta} N_\gamma \) is given the discrete topology, along the embedding \( x \mapsto (x|_\delta)_{\delta \in \Delta} \).
Topology \( \tau_{\Delta } \) is Hausdorff, it is usually strictly finer than the product topology, and it is designed to match the combinatorics behind generalized wreath products: indeed, it is the coarsest topology which makes open all congruences \( \equiv_\gamma \) (for \( \gamma \in \Delta \)) on \( \prod_{\delta \in \Delta} N_\delta \) introduced by Holland~\cite{holland1969} as a crucial ingredient in his analysis.

Unless otherwise stated, we endow \( \Wr^{S}_{\delta \in \Delta} H_\delta \) with the pointwise convergence topology with respect to the relativization of \( \tau_\Delta \) to \( S \), which will simply be denoted by \( \tau \). 
Notice that when \( \Delta \) is an antichain and \( \AAA = \LF = \pow (\Delta ) \), we recover the usual product topology on \( \prod_{\delta \in \Delta} H_\delta = \Wr^{\AAA}_{\delta \in \Delta} H_\delta \).

Suppose that the skeleton \( \langle \Delta, N \rangle \) is countable. Then the topological space \( \prod_{\delta \in \Delta} N_\delta \) can be equipped with a natural compatible complete ultrametric \( d_\Delta \) as follows: enumerate \( \Delta \) as \( (\delta_m)_{m < M} \) for the appropriate \( M \leq \omega \), and given distinct \(x,y \in \prod_{\delta \in \Delta} N_\delta \), let \( d_\Delta (x,y)= 2^{-m} \) with \( m < M \) smallest such that \( x|_{\delta_m} \neq y|_{\delta_m} \).
It follows that each domain \( S \subseteq \prod_{\delta \in \Delta} N_\delta \) is metrizable too.
We say that \( S \) is \markdef{separable} if it is such when endowed with the relative topology induced by \( \tau_\Delta \); by countability of \( \langle \Delta, N \rangle \), this happens precisely when \( |S_\delta| \leq \aleph_0 \) for all \( \delta \in \Delta \), where 
\begin{equation} \label{eq:Sdelta} 
S_\delta = \{ x|_\delta \mid x \in S \} .
\end{equation}
If \( S \) is separable, then \( \Wr^S_{\delta \in \Delta} H_\delta \) is second-countable,%
\footnote{One can verify that if \( S \) is approximately homogeneous, then the reverse implication holds as well.%
\details{Let \( \delta \in \Delta \) be such that \( |S_\delta| = \kappa > \aleph_0 \). Let \( (s_\alpha)_{\alpha < \kappa} \) be an enumeration of \( S_\delta \), and \( x_\alpha \in S \) be such that \( x_\alpha |_{\delta} = s_\alpha \). 
Since \( \Wr^S_{\delta \in \Delta} H_\delta \) is transitive (Ci vuole ``approximately homogeneous''!!!),
for every \( \alpha < \kappa \) there is \( g_\alpha \in \Wr^S_{\delta \in \Delta} H_\delta \) such that \( g_\alpha (x_0) = x_\alpha \). 
The basic open neighborhoods of \( g_\alpha \) determined by \( x_0 \) and \( \delta \) (i.e. \( \{ g \in \Wr^S_{\delta \in \Delta} H_\delta \mid g(x_0)|_{\delta} = g_\alpha(x_0)|_\delta \} \)) form a \( \kappa \)-sized family of pairwise disjoint nonempty open sets.}
}
and hence also metrizable by the Birkhoff-Kakutani theorem. 

Furthermore, we say that \( S \) is \markdef{closed} if it is \( \tau_\Delta \)-closed in \( \prod_{\delta \in \Delta} N_\delta \);
this obviously implies that when equipped with the restriction of the ultrametric \( d_\Delta \), the metric space \( S \) is complete.
For example, both Holland's \( S^{\Max} \) and Malicki's \( S^{\UM} \) are closed, as is our \( S^{\LF}\). 
In contrast, if \( \Delta \) is not well-founded, then Hall's \( S^{\Fin} \) might fail to be closed: its \( \tau_\Delta \)-closure is precisely \( S^{\LF}\), and this is one reason why generalized wreath products over domains induced by locally finite supports play a special role in our framework. 

By the above discussion, every closed separable domain \( S \) over a countable skeleton \( \langle \Delta, N \rangle \) is a Polish ultrametric space
when equipped with the (restriction of the) distance \( d_\Delta \).
Another way to see this is the following. Suppose that \( S \) is separable. Then \( \prod_{m<M} S_{\delta_m} \), when equipped with the product of the discrete topologies on the sets \( S_{\delta_m} \), is homeomorphic to a closed subset of the Baire space \( \pre{\omega}{ \omega }\). 
Under this identification, \( d_{\Delta} \) is just the pullback of the usual metric on \( \pre{\omega}{ \omega } \) along the embedding \( x \mapsto (x|_{\delta_m})_{m<M} \).
If \( S \) is also closed, then the range of such embedding is closed in \( \pre{\omega}{\omega}\), and the result easily follows.

Equip the group \( \Iso(S,d_\Delta) \) with the pointwise convergence topology, so that it is a Polish group (\cite[Example 9 of \S9.B]{Kechris1995}). 
By condition~\ref{Wr1}, the generalized wreath product \( \Wr^S_{\delta \in \Delta} H_\delta \) is a topological subgroup of \( \Iso(S,d_\Delta) \): we show that it is indeed Polish when all groups \( H_\delta \) are closed.

\begin{proposition} \label{prop:Polishtopologyfrommetric}
Let \( \langle \Delta,N \rangle \) be a countable skeleton, and \( S \subseteq \prod_{\delta \in \Delta} N_\delta \) be a closed separable domain. If each \( H_\delta \subseteq \Sym(N_\delta) \) is a closed group, then \( \Wr^S_{\delta \in \Delta} H_\delta \) is a Polish group.
\end{proposition}

\begin{proof}
Since under our assumptions \( (S,d_\Delta) \) is a Polish metric space, \( \Iso(S,d_\Delta) \) is a Polish group.
It is clear that~\ref{Wr1} is a closed condition, hence the subgroup \( G \subseteq \Iso(S,d_\Delta) \) of all \( g \) satisfying it is Polish.\details{Direi che $G$ è l'intersezione di tutti i $\Iso(S,d_{\Delta })$ al variare dell'enumerazione $\delta_n$. Ci insegna qualcosa?}
For any \( \delta \in \Delta \) and \( z \in S_\delta \), consider the map \( f_{\delta,z} \colon G \to \Sym(N_\delta) \) sending \( g \) to the permutation \( i \mapsto g(x^\delta_i)(\delta) \) for some (equivalently, any) \( x \in S \) with \( x|_\delta = z \): by~\ref{Wr1}, the choice of \( x \) is irrelevant, and \( f_{\delta,z}(g) \in \Sym(N_\delta) \) by Remark~\ref{rmk:permutationisautomatic}. Since \( f_{\delta,z} \) is continuous and 
\[ 
\Wr^S_{\delta \in \Delta} H_\delta = \bigcap_{\delta \in \Delta} \bigcap_{z \in S_\delta} f_{\delta ,z}^{-1}(H_\delta) ,
\]
the wreath product \( \Wr^S_{\delta \in \Delta} H_\delta \) is a closed subgroup of \( G \), hence we are done.
\end{proof}

%

A situation where Proposition~\ref{prop:Polishtopologyfrommetric} applies, and that will be crucial in this paper, is when \( S = S^{\LF} \). For every skeleton \( \langle \Delta, N \rangle \), it holds \( S^{\Fin} \subseteq S^{\LF} \subseteq S^{\UM} \subseteq S^{\Max} \), and moreover: 
\begin{itemizenew}
\item
\( S^{\LF} \subsetneq S^{\UM} \) if and only if \( \Delta \) contains an infinite antichain \( \{ \delta_n \mid n \in \omega \}\) bounded from below and such that \( N_{\delta_n} > 1 \) for all \( n \in \omega \): in this case, \( \Wr^{\UM}_{\delta \in \Delta} H_\delta \) is not second-countable, and hence neither Polish.
\item
\( S^{\UM} \subsetneq S^{\Max} \) if and only if \( \Delta \) contains an infinite decreasing chain \( \{ \delta_n \mid n \in \omega \}\) bounded from below and such that \( N_{\delta_n} > 1 \) for all \( n \in \omega \): in this case, \( \Wr^{\Max}_{\delta \in \Delta} H_\delta \) is not Polish again because it is not second-countable.
\item 
when \( \Delta \) is a countable linear order,
\( S^{\Fin}\ \subsetneq S^{\LF} \) if and only if \( \Delta \) has no minimum and \( \{ \delta \in \Delta \mid N_\delta > 1 \} \) is coinitial in \( \Delta \): in this case, \( \Wr^{\Fin}_{\delta \in \Delta} H_\delta \) is not completely metrizable, and hence it is not a Polish group.
\details{Let \( (\delta_n)_{n \in \omega} \) be coinitial in \( \Delta \) with \( N_{\delta_n} > 1 \). Consider the groups \( H = \Wr^{\Fin}_{\delta \in \Delta } H_\delta \) and \( G = \Wr^{\LF}_{\delta \in \Delta } H_\delta \). Then \( H \) embeds into \( G \) as follows. Given \( g \in H \), let \( g' \in G \) be defined by setting \( g'(x) = y \iff \) for each \( \delta \) it holds \( y|_\delta = g(x|_\delta^+)|_\delta\), where \( x|_\delta^+ \in S^{\Fin}\) is the only element which follows \( x \) up to \( \delta \) and is \( 0 \) below it. [Alternatively, one can identify \( g \) with the maps \( g_\delta \colon S^{\Fin}_\delta \to S^{\Fin}_\delta \): every such map canonically induces an element \( g' \in G \) because \( S^{\Fin}_\delta = S^{\LF}_\delta \).] Now, \( G \) is a Polish group and \( H \), up to the identification \( g \mapsto g' \) above, is a subgroup of it. Indeed, \( H = \{ g \in G \mid g(S^{\Fin}) = S^{\Fin } \} \). The claim is that \( H \) is \( F_\sigma \) hard as a subset of \( G \). To this aim, we define a continuous map from the Cantor space to \( G \) reducing the usual \( F_\sigma \)-complete set (=``finitely many \( 1\)s'') to \( H \). For each \( n \), fix \( h_n \in H_{\delta_n} \) sending \( 0 \) to \( 1 \). The idea is that each time a new bit \( w(n) \) of the input element \( w \) of the Cantor space is revealed, we better approximate its image \( g_w\in  G \) by describing \( (g_w)_{\delta_n} \). This is done by letting \( (g_w)_{\delta_n} \) extend \( (g_w)_{\delta_{n-1}}\) trivially (=identity on the new coordinates) if \( w(n) = 0 \), while if \( w(n) = 1\) then on the last coordinate (=\( \delta_n \)) we use \( h_n \). [In other word, \( g_w\) acts identically execept for the coordinates \( \delta_n \) such that \( w(n) = 1 \).] In this way, if \( w \) has infinitely many \( 1\)s, then \( g_w\) will send the element \( x \in S^{\Fin }\) with empty support into an element \( g_w(x) \notin S^{\Fin}\), and hence \( g_w \notin H \). If \( w \) instead has finitely many \( 1 \)s, then \( g_w \) is the identity except for the finitely many coordinates where the \( h_n\) have been used, hence \( g_w(S^{\Fin}) = S^{\Fin}\) and \( g_w \in H\).}
\end{itemizenew}
This implies that even when the skeleton \( \langle \Delta, N \rangle \) is countable and the groups \( H_\delta \) are closed, Hall's \( \Wr^{\Fin}_{\delta \in \Delta} H_\delta \) with \( \Delta \) linear, Holland's \( \Wr^{\Max}_{\delta \in \Delta} H_\delta \), and Malicki's \( \Wr^{\UM}_{\delta \in \Delta} H_\delta \) are not Polish groups, unless they coincide with \( \Wr^{\LF}_{\delta \in \Delta} H_\delta \). 


Since \( \Iso(U) \) is always a Polish group when \( U \) is Polish ultrametric, 
the above discussion shows that none of the variants of classical generalized wreath products from the literature can be used to solve our problem, and that locally finite supports are instead a promising choice. 
Nevertheless, in the case of Hall's finite supports, there is a different topology that can turn \( \Wr^{\Fin}_{\delta \in \Delta} H_\delta \) into a Polish group when \( \langle \Delta ,N \rangle\) is countable and each \( H_\delta \) is closed. 
Indeed, in such a situation the domain \( S^{\Fin} \) is countable and can be given the discrete topology,%
\footnote{This is basically the only situation where this alternative approach can be used, as in general none of \( S^{\Max} \), \( S^{\UM} \), or \( S^{\LF} \) is countable.}
and then we can again equip \( \Wr^{\Fin}_{\delta \in \Delta} H_\delta \subseteq \Sym(S^{\Fin}) \) with the induced pointwise convergence topology, which will be denoted by \( \tau^* \).
When \( \Delta \) is linear, \( \tau^*\) coincides with \( \tau \) if \( \Delta \) has a minimum, but it is usually strictly finer otherwise.
The alternative topology \( \tau^* \) will occasionally be used in Theorems~\ref{thm:homogeneousdiscrete} and ~\ref{thm:Urysohn}, where Hall's generalized wreath products naturally show up.

\subsection{Generalized wreath products over local domains} \label{subsec:wreathproductlocal}

Recall from~\eqref{eq:Sdelta} that for every \( \delta \in \Delta \) we let \( S_\delta = \{ x|_\delta \mid x \in S \} \).
By condition~\ref{Wr1}, every \( g \in \Wr^{S}_{\delta \in \Delta} H_\delta \) acts on \( S \) in a ``local'' way, by sending every \( z \in S_\delta \) to \( g(x)|_\delta \in S_\delta\) for some (equivalently, any) \( x \in S \) such that \( x|_\delta = z \). Therefore, \emph{when \( S\) is closed} the wreath product \( \Wr^{S}_{\delta \in \Delta} H_\delta \) can be presented as the group
of all permutations \( g \in \Sym(\Sloc) \), where \( \Sloc = \bigcup_{\delta \in \Delta} S_\delta \), satisfying the following conditions: \( g(z) \in S_\delta \) for every \( z \in S_\delta \) (and hence \( g(S_\delta) = S_\delta \)), \( g \) preserves restrictions (i.e.\ \( g(z|_\gamma) = g(z)|_\gamma \) for every \( \gamma \geq \delta \) and \( z \in S_\delta \)), and \( g \) satisfies the obvious reformulation of~\ref{Wr2} (i.e.\ condition~\ref{Wr2generalizedlocal} below).%
\footnote{Alternatively, each such \( g \in \Sym(\Sloc) \) can be construed as a family \( (g_\delta)_{\delta \in \Delta} \in \prod_{\delta \in \Delta} \Sym(S_\delta) \) of coherent (with respect to restrictions) ``local'' permutations, each of which satisfies an analogue of~\ref{Wr2}.}
In this setting, the topology \( \tau \) on \( \Wr^{S}_{\delta \in \Delta} H_\delta \) is the pointwise convergence topology on \( \Sym(\Sloc) \), where \( \Sloc \) is given the discrete topology. 
If the skeleton \( \langle \Delta,N \rangle \) and all the sets \( S_{\delta} \) are countable, then \( \Sloc \) is countable, and hence either \( \Sym(\Sloc) \) is a finite symmetric group or \( \Sym(\Sloc) \cong \Sym(\omega) \).
If furthermore all the groups \( H_\delta \) are closed, then \( \Wr^{S}_{\delta \in \Delta} H_\delta \) is, up to isomorphism, a closed subgroup of \( \Sym(\omega) \); in particular, we recover Proposition~\ref{prop:Polishtopologyfrommetric}.

Nothing in the description above depends on the fact that the sets \( S_\delta \) come from a global domain \( S \subseteq \prod_{\delta \in \Delta} N_\delta \). Therefore generalized wreath products can be naturally rephrased as permutation groups over arbitrary families of local domains. Fix a skeleton \( \langle \Delta, N \rangle \). 
Given any \( \Sloc \subseteq \bigcup_{\delta \in \Delta} \left( \prod_{\gamma \geq \delta} N_\gamma \right) \) and \( \delta \in \Delta \), we let \( \Sloc_\delta = \Sloc \cap \prod_{\gamma \geq \delta} N_\gamma \). The sets \( \Sloc_\delta \) clearly form a partition of \( \Sloc \).

\begin{defin} \label{localdomains}
A \markdef{family of local domains} (over the skeleton \( \langle \Delta, N \rangle\)) is a collection \( \Sloc \subseteq \bigcup_{\delta \in \Delta} \left(\prod_{\gamma \geq \delta} N_\gamma \right) \) such that \( \Sloc_\delta \neq \emptyset \) for every \( \delta \in \Delta \), and the following conditions are satisfied:
\begin{enumerate-(a)}
\item \label{localdomains-a}
if \( \gamma \geq \delta \), then \( \Sloc_\gamma = \{ z|_\gamma \mid z \in \Sloc_{\delta} \} \);
\item \label{localdomains-c}
\( \{ z^\delta_i \mid i \in N_\delta \} \subseteq \Sloc_\delta \) for every \( z \in \Sloc_\delta \), where \( z^\delta_i(\delta) = i \) and \( z^\delta_i(\gamma) = z(\gamma) \) if \( \gamma > \delta \).
\end{enumerate-(a)} 
\end{defin}

Notice that the family of local domains \( \Sloc \) is countable exactly when the skeleton \( \langle \Delta,N \rangle \) is countable and \( |\Sloc_\delta| \leq \aleph_0  \) for all \( \delta \in \Delta \).  

\begin{defin} \label{def:generalizedlocal}
Let \( \langle \Delta, N \rangle\) be a skeleton, \( \Sloc \) be a family of local domains, and \( (H_\delta)_{\delta \in \Delta} \) be a family of transitive permutation groups over the sets \( N_\delta \). The group
\[  
\Wr^{\Sloc}_{\delta \in \Delta} H_\delta,
\]
which is still called \markdef{generalized wreath product}, is the group of all permutations \( g \in \Sym(\Sloc)\) satisfying the following conditions,%
\footnote{The letter ``G'' in the enumeration stands for ``generalized''.} 
for all \( \gamma \geq \delta \) and \( z \in \Sloc_\delta \):
\begin{enumerate-(G1)}
\item \label{Wr1generalizedlocal}
\( g(\Sloc_\delta) = \Sloc_\delta \) and \( g(z|_\gamma) = g(z)|_\gamma \);
\item \label{Wr2generalizedlocal}
the map \( i \mapsto g(z^\delta_i)(\delta) \) is a permutation of \( N_\delta \) belonging to \( H_\delta \).
\end{enumerate-(G1)}
\end{defin}

The notion of approximate homogeneity from Definition~\ref{def:approximatehomogeneity} translates to the following:

\begin{defin} \label{def:localhomogeneity}
A family of local domains \( \Sloc = \bigcup_{\delta \in \Delta} \Sloc_\delta \) is \markdef{locally homogeneous} if for every \( \delta \in \Delta \) and \( z,z' \in \Sloc_{\delta } \) there is \( g \in \Wr^{\Sloc}_{\delta \in \Delta} \Sym(N_\delta) \) such that \( g(z) = z' \).
\end{defin}

The group \( \Wr^{\Sloc}_{\delta \in \Delta} H_\delta \) is again equipped with the pointwise convergence topology \( \tau \) (where \( \Sloc \) is given the discrete topology), which turns it into a closed (and hence Polish) subgroup of \( \Sym(\omega) \) whenever \( \Sloc \) is countable and the groups \( H_\delta \) are closed. 

As discussed above, the classical generalized wreath products \( \Wr^S_{\delta \in \Delta} H_\delta \) (for \( S \subseteq \prod_{\delta \in \Delta} N_\delta \) a closed domain) of Section~\ref{subsec:classicalwreathproduct} can be viewed as a particular instance of the generalized wreath products \( \Wr^{\Sloc}_{\delta \in \Delta} H_\delta \) from Definition~\ref{def:generalizedlocal} in which the family of local domains is \( \Sloc = \bigcup_{\delta \in \Delta} S_\delta \).
This process can be reversed.
More precisely, say that the family \( \Sloc \) is \markdef{full} if for every \( \bar\delta \in \Delta \) and \( z \in \Sloc_{\bar\delta}  \) there is \( x \in \prod_{\delta \in \Delta} N_\delta  \) such that \( x|_{\bar \delta} = z \) and \( x|_\gamma \in \Sloc_\gamma \) for every \( \gamma \in \Delta \) 
(equivalently: there is \( (z_\delta)_{\delta \in \Delta} \in \prod_{\delta \in \Delta} \Sloc_\delta \) such that \( z_{\bar \delta} = z \) and \( (z_{\delta})|_{\gamma} = z_{\gamma}\) if \( \delta \leq \gamma \)). 
For example, if the partial order \( \Delta \) is linear and has countable coinitiality, or if \( \Delta \) is an antichain, then every family \( \Sloc \) over the skeleton \( \langle \Delta,N \rangle \) is full. 
On the other hand, it is not hard to produce (even finite) families of local domains over non-linear skeletons which are not full.  \details{Counterexample with four elements:
\begin{itemize}
\item \( \Delta =\{ a,b,c,d\} \) with \( c<a\), \(c<b\), \(d<a\), \(d<b \) and no other \( < \)-relation
\item \( N_a=N_b=N_c=N_d=2 \)
\item \( \Sloc_c=\{\{ (a,0),(b,0),(c,j)\} ,\{ (a,1),(b,1),(c,j)\}\}_{j\in 2} \)
\item \( \Sloc_d=\{\{ (a,0),(b,1),(d,j)\} ,\{ (a,1),(b,0),(d,j)\}\}_{j\in 2} \)
\end{itemize} }
To every full family of local domains \( \Sloc \) we can canonically associate the closed domain
\[ 
S = \left\{ x \in \prod\nolimits_{\delta \in \Delta} N_\delta \mid \forall \gamma \in \Delta \, (x|_\gamma \in \Sloc_\gamma )\right\} , 
\]
which clearly satisfies \( S_\gamma = \Sloc_\gamma \), where as before \( S_\gamma = \{ x|_\gamma \mid x \in S \} \). Vice versa, if \( S \) is a closed domain, then 
\[ 
\Sloc = \bigcup_{\delta \in \Delta} S_\delta 
\] 
is a full family of local domains satisfying \( \Sloc_\gamma = S_\gamma \) for every \( \gamma \in \Delta \). Thus there is a one-to-one correspondence between closed domains and full families of local domains. 
Under such correspondence, \( S \) is separable if and only if \( \Sloc \) is countable, and the weak forms of homogeneity introduced in Definitions~\ref{def:approximatehomogeneity} and~\ref{def:localhomogeneity} are preserved as well. 
Therefore, we have the following equivalence:

\begin{proposition} \label{prop:closedvsfull}
Let \( \langle \Delta, N \rangle\) be a skeleton 
and \( (H_\delta)_{\delta \in \Delta} \) be a family of transitive permutation groups over the sets \( N_\delta \). For every topological group \( G \), the following are equivalent:
\begin{enumerate-(1)}
\item 
\( G \cong \Wr^S_{\delta \in \Delta} H_\delta \) for some closed (separable, approximately homogeneous) domain \( S \);
\item 
\( G \cong \Wr^{\Sloc}_{\delta \in \Delta} H_\delta \) for some full (countable, locally homogeneous) family of local domains \( \Sloc \).
\end{enumerate-(1)}
\end{proposition}

For notational simplicity, when the domain \( S = S^{\AAA} \) is induced by a family of admissible supports \( \AAA \subseteq \pow(\Delta) \), we denote by \( \Sloc^{\AAA} \) the corresponding family of local domains \( \bigcup_{\delta \in \Delta} S_\delta \).
The family 
\[ 
\Sloc^{\Max} = \bigcup\nolimits_{\delta \in \Delta} \left\{ x|_\delta \mid {x \in \prod\nolimits_{\delta \in \Delta }N_\delta} \wedge {\supp(x) \in \Max} \right\} 
\]
induced by Holland's \( S^{\Max} \) will be particularly relevant in our work. 
When there is no danger of confusion (i.e.\ when \( S^{\AAA} \) is a closed domain), we again write \( \Wr^{\AAA}_{\delta \in \Delta} H_\delta \) instead of \( \Wr^{\Sloc^{\AAA}}_{\delta \in \Delta } H_\delta \): this little abuse of notation does not cause problems because, as discussed, in such a situation we have \( \Wr^{S^{\AAA}}_{\delta \in \Delta} H_\delta \cong \Wr^{\Sloc^{\AAA}}_{\delta \in \Delta} H_\delta \).

\subsection{Projective wreath products} \label{subsec:projectivewreathproduct}

By replacing in Definition~\ref{def:generalizedlocal} the restriction operator with arbitrary systems of projections, we obtain a new class of groups, called \emph{projective} wreath products.

\begin{defin} \label{def:projections}
Let \( \langle \Delta, N \rangle \) be a skeleton. 
A \markdef{system of projections} (over the skeleton \( \langle \Delta, N \rangle \)) is a pair \( \langle \Sloc, \bpi \rangle \) satisfying the following conditions:
\begin{itemizenew}
\item 
\( \Sloc \subseteq \bigcup_{\delta \in \Delta} \left(\prod_{\gamma \geq \delta} N_\gamma \right) \) is such that \( \Sloc_\delta \neq \emptyset \) for every \( \delta \in \Delta \), and moreover \( \{ z^\delta_i \mid i \in N_\delta \} \subseteq \Sloc_\delta \) for every \( z \in \Sloc_\delta \);
\item
\( \bpi = (\pi_{\delta\gamma})_{\gamma \geq \delta} \) is a family of surjective maps \( \pi_{\delta \gamma} \colon \Sloc_\delta \to \Sloc_\gamma \) such that for all \( z,z' \in \Sloc_\delta \) and \( \beta \geq \gamma \geq \delta \)
\begin{enumerate-(a)}
\item \label{def:projections-a}
\( \pi_{\delta\gamma}(z) = \pi_{\delta\gamma}(z') \) if and only if \( z|_\gamma = z'|_\gamma \);
\item \label{def:projections-b}
\( \pi_{\gamma\beta} \circ \pi_{\delta\gamma} = \pi_{\delta\beta} \).
\end{enumerate-(a)}
Notice that for every \( \delta \in \Delta \), the projection \( \pi_{\delta\delta} \) is a permutation of \( \Sloc_\delta \) by~\ref{def:projections-a}, and hence the identity on \( \Sloc_\delta \) by~\ref{def:projections-b}.
\end{itemizenew}
\end{defin}

In our results, quite often \( \Sloc \) will be a family of local domains itself: in these cases, the first item of Definition~\ref{def:projections} is automatically satisfied by part~\ref{localdomains-c} of Definition~\ref{localdomains}. 
Notice also that if \( \bpi \) is \markdef{trivial}, i.e.\ \( \pi_{\delta \gamma}(z) = z|_\gamma \) for every \( \gamma \geq \delta \) and \( z \in \Sloc_\delta \), then any \( \Sloc \) as in Definition~\ref{def:projections} is a family of local domains; conversely, every family of local domains can be turned into a system of projections by pairing it with the trivial projections (i.e.\ restrictions).
Finally, we say that a system of projections \( \langle \Sloc,\bpi \rangle \) is \markdef{countable} if so is \( \Sloc \).

\begin{defin} \label{def:projective}
Let \( \langle \Delta, N \rangle\) be a skeleton, and let \( \langle \Sloc, \bpi \rangle \) be a system of projections. Let also \( (H_\delta)_{\delta \in \Delta} \) be a family of transitive permutation groups over the sets \( N_\delta \). The \markdef{projective} (\markdef{generalized}) \markdef{wreath product}
\[  
\Wr^{\Sloc, \bpi}_{\delta \in \Delta} H_\delta
\]
is the group of all permutations \( g \in \Sym(\Sloc)\) satisfying the following conditions,%
\footnote{The letter ``P'' in the enumeration stands for ``projective''.} 
for all \( \gamma \geq \delta \) and \( z \in \Sloc_\delta \):
\begin{enumerate-(P1)}
\item \label{Wr1projective}
\( g(\Sloc_\delta) = \Sloc_\delta \) and \( g(\pi_{\delta \gamma}(z)) = \pi_{\delta \gamma}(g(z)) \);
\item \label{Wr2projective}
the map \( i \mapsto g(z^\delta_i)(\delta) \) is a permutation of \( N_\delta \) belonging to \( H_\delta \).
\end{enumerate-(P1)}
\end{defin}

When \( \Sloc = \bigcup_{\delta \in \Delta} S_\delta \) is the family of local domains canonically induced by a  closed global domain \( S \subseteq \prod_{\delta \in \Delta} N_\delta \), we simply write \( \Wr^{S,\bpi}_{\delta \in \Delta} H_\delta \) instead of \( \Wr^{\Sloc,\bpi}_{\delta \in \Delta} H_\delta \); if moreover \( S = S^{\AAA} \) for some family of admissible supports \( \AAA \), then we further simplify the notation and write \( \Wr^{\AAA,\bpi}_{\delta \in \Delta} H_\delta \). 

Coherently with Definition~\ref{def:localhomogeneity}, we say that a system of projections \( \langle \Sloc,\bpi \rangle \) is \markdef{locally homogeneous} if for every \( \delta \in \Delta \) and \( z,z' \in \Sloc_\delta \) there is \( g \in \Wr^{\Sloc,\bpi}_{\delta \in \Delta} \Sym(N_\delta ) \) such that \( g(z) = z' \).
The notion of fullness can be naturally adapted to the new context as well: a system of projections \( \langle \Sloc,\bpi \rangle \) is \markdef{full} if for every \( \bar \delta \in \Delta \) and \( z \in \Sloc_{\bar \delta}\) there is \( (z_\delta)_{\delta \in \Delta} \in \prod_{\delta \in \Delta} \Sloc_\delta \) such that \( z_{\bar \delta} = z \) and \( (z_{\delta })_{\delta\in\Delta } \) is \markdef{coherent}, that is, \( \pi_{\delta \gamma}(z_\delta) = z_\gamma\) for all \( \delta \leq \gamma \).

The group \( \Wr^{\Sloc,\bpi}_{\delta \in \Delta} H_\delta \) is again equipped with the pointwise convergence topology \( \tau \) (where \( \Sloc \) is discrete), so that it is a closed, and hence Polish, subgroup of \( \Sym(\omega) \) whenever \( \Sloc \) is countable and the groups \( H_\delta \) are closed. 

It is clear that the generalized wreath products \( \Wr^{\Sloc}_{\delta \in \Delta} H_\delta \) from Section~\ref{subsec:wreathproductlocal} are precisely the projective wreath products in which the family of projections \( \bpi \) is trivial. 
Conversely, we are now going to show that under the mild assumption \( \Sloc \subseteq \Sloc^{\Max}\) every projective generalized wreath product can be realized as a generalized wreath product over some family of local domains, \emph{up to isomorphism}. 
We say that a system of projections \( \langle \Sloc,\bpi \rangle \) has \markdef{finite character} if for every \( \delta \in \Delta \), the set \( \{ \gamma \in \Delta \mid \gamma \geq \delta \} \) can be partitioned into finitely many convex sets \( C_0, \dotsc, C_n \) such that for all \( i \leq n \) and all \( \beta \geq \gamma \) with \( \beta, \gamma \in C_i \), the projection \( \pi_{\gamma \beta} \) is trivial (i.e.\ \( \pi_{\gamma\beta} (z) = z|_\beta \) for every \( z \in \Sloc_\gamma \)).

\begin{theorem} \label{thm:proj->unchained}
Let \( \langle \Delta,N \rangle \) be a skeleton, and let \( (H_\delta)_{\delta \in \Delta} \) be a family of transitive permutation groups over the sets \( N_\delta \). Let \( \langle \Sloc, \bpi \rangle \) be a system of projections such that \( \Sloc \subseteq \Sloc^{\Max} \).
Then there is family of local domains \( \Sloc' \) such that
\[  
\Wr^{\Sloc'}_{\delta \in \Delta} H_\delta \,  \cong \, \Wr^{\Sloc,\bpi}_{\delta \in \Delta} H_\delta.
\]
If moreover \( \langle \Sloc,\bpi \rangle \) has finite character, then we can further require \( \Sloc' \subseteq \Sloc^{\Max} \).
\end{theorem}

\begin{proof}
For every \( \delta \in \Delta \) and \( z \in \Sloc_\delta \), let \( \rho(z) \in \prod_{\gamma \geq \delta} N_\delta\) be defined by setting \( \rho(z)(\gamma) = \pi_{\delta \gamma}(z)(\gamma) \), for every \( \gamma \geq \delta \).
The ``rewriting map'' \( \rho \colon \Sloc \to\bigcup_{\delta\in\Delta }\big(\prod_{\gamma\ge\delta }N_{\delta } \big) \) has the following properties.  
\begin{claim} \label{claim:proj->unchained}
For every \( \gamma \geq \delta \), \( z,z' \in \Sloc_\delta \), and \( z'' \in \Sloc_\gamma \), 
\begin{enumerate-(i)}
\item \label{proj->unchained-i}
\( z|_\gamma = z'|_\gamma \iff  \rho(z)|_\gamma = \rho(z')|_\gamma \),
\item \label{proj->unchained-ii}
\( \pi_{\delta \gamma}(z) = z''  \iff \rho(z)|_{\gamma} = \rho(z'') \),
\item \label{proj->unchained-iii}
\( \rho(z)(\delta) = z(\delta) \).
\end{enumerate-(i)}
\end{claim}

\begin{proof}[Proof of the claim]
\ref{proj->unchained-i}
Suppose first that \( z|_\gamma = z'|_\gamma \). Then for every \( \gamma' \geq \gamma \) we have \( z|_{\gamma'} = z'|_{\gamma'} \), and hence \( \pi_{\delta \gamma'}(z) = \pi_{\delta \gamma'}(z') \) by condition~\ref{def:projections-a} of Definition~\ref{def:projections}. 
Therefore \( \rho(z)(\gamma') = \pi_{\delta \gamma'}(z)(\gamma') = \pi_{\delta \gamma'}(z')(\gamma') = \rho(z')(\gamma') \) for every \( \gamma' \geq \gamma \), that is, \( \rho(z)|_\gamma = \rho(z')|_\gamma \). 

Now suppose that \( z|_\gamma \neq z'|_\gamma \). By the maximum condition, we can find a maximal \( \gamma' \geq \gamma \) such that \( z|_{\gamma'} \neq z'|_{\gamma'} \), so that \( \pi_{\delta \gamma'}(z) \neq \pi_{\delta \gamma'}(z') \) by Definition~\ref{def:projections}\ref{def:projections-a}. Fix any \( \gamma'' > \gamma' \). By maximality of \( \gamma' \), \( z|_{\gamma''} = z'|_{\gamma''} \), and hence \( \pi_{\delta \gamma''}(z) = \pi_{\delta \gamma''}(z') \). 
On the other hand, by condition~\ref{def:projections-b} of Definition~\ref{def:projections} we have \( \pi_{\delta \gamma''} (z) = \pi_{\gamma' \gamma''}(\pi_{\delta \gamma'}(z))\) and  \( \pi_{\delta \gamma''} (z') = \pi_{\gamma' \gamma''}(\pi_{\delta \gamma'}(z'))\), hence \( \pi_{\delta \gamma'}(z)|_{\gamma''} = \pi_{\delta \gamma'}(z')|_{\gamma''} \) for every \( \gamma'' > \gamma' \) by Definition~\ref{def:projections}\ref{def:projections-a} again. Since \( \pi_{\delta \gamma'}(z) \neq \pi_{\delta \gamma'}(z') \), necessarily \( \pi_{\delta \gamma'}(z)(\gamma') \neq \pi_{\delta \gamma'}(z')(\gamma') \), hence \( \rho(z)(\gamma') \neq \rho(z')(\gamma') \), and thus \( \rho(z)|_\gamma \neq \rho(z')|_\gamma \).

\ref{proj->unchained-ii}
Using the surjectivity of \( \pi_{\delta \gamma} \), pick \( z^* \in \Sloc_\delta \) such that \( \pi_{\delta \gamma}(z^*) = z'' \).
Then by Definition~\ref{def:projections}\ref{def:projections-a} and~\ref{proj->unchained-i} we have 
\[ 
\pi_{\delta \gamma}(z) = z'' \iff z|_\gamma = z^*|_\gamma \iff \rho(z)|_\gamma = \rho(z^*)|_\gamma , 
\]
hence it is enough to show that \( \rho(z^*)|_\gamma = \rho(z'') \). But this follows from the fact that, by Definition~\ref{def:projections}\ref{def:projections-b} and the choice of \( z^* \), for every \( \gamma' \geq \gamma \)
\[ 
\rho(z^*)(\gamma') = \pi_{\delta \gamma'}(z^*)(\gamma') = \pi_{\gamma \gamma'}(\pi_{\delta \gamma}(z^*))(\gamma') = \pi_{\gamma \gamma'}(z'')(\gamma') = \rho(z'')(\gamma') .
\]

\ref{proj->unchained-iii}
Recall that \( \pi_{\delta \delta}(z) = z \). Thus, \( \rho(z)(\delta)= \pi_{\delta \delta}(z)(\delta) = z(\delta) \).
\end{proof}

Set \( \Sloc'_\delta = \rho(\Sloc_\delta) \) and \( \Sloc' = \bigcup_{\delta \in \Delta} \Sloc'_\delta \). 
By~Claim~\ref{claim:proj->unchained}\ref{proj->unchained-i}, \( \rho \restriction \Sloc_\delta \colon \Sloc_\delta \to \Sloc'_\delta \) is injective, and hence a bijection. 
If moreover \( \langle \Sloc,\bpi \rangle \) has finite character, then \( \Sloc'_\delta \subseteq \Sloc^{\Max} \) for every \( \delta \in \Delta \).
Indeed, suppose towards a contradiction that there are \( z \in \Sloc_\delta \) and an infinite sequence \( \delta < \delta_0 < \delta_1 < \delta_2 < \dotsc \) such that \( \rho(z)(\delta_j) \neq 0 \) for all \( j \in \omega \). Let \( C_0, \dotsc, C_n \) be a finite partition of \( \{ \gamma \in \Delta \mid \gamma \geq \delta \} \) witnessing the finite character of \( \langle \Sloc,\bpi \rangle \) with respect to \( \delta \). By convexity, we can assume that there is \( \bar \imath \leq n \) such that \( \delta_j \in C_{\bar \imath}\) for all \( j \in \omega \). Let \( z' = \pi_{\delta \delta_0 }(z) \in \Sloc_{\delta_0} \). By choice of \( C_{\bar \imath} \),
\[ 
z'(\delta_j) = (z'|_{\delta_j})(\delta_j) =  \pi_{\delta_0 \delta j}(z')(\delta_j) = \pi_{\delta \delta_j}(z)(\delta_j) = \rho(z)(\delta_j) \neq 0 
\]
for all \( j \in \omega \), contradicting \( \Sloc_{\delta_0} \subseteq \Sloc^{\Max}\).

The surjectivity of \( \pi_{\delta \gamma} \) entails that for every \( \gamma \geq \delta \) we have \( \Sloc_\gamma = \{ \pi_{\delta \gamma} (z) \mid z \in \Sloc_\delta \} \), hence by Claim~\ref{claim:proj->unchained}\ref{proj->unchained-ii} we have that \( \Sloc' \) satisfies condition~\ref{localdomains-a} of Definition~\ref{localdomains}, i.e.\ \( \Sloc'_\gamma = \{ z'|_\gamma \mid z' \in \Sloc'_\delta \} \).
%
%
Condition~\ref{localdomains-c} of the same definition easily follows from items~\ref{proj->unchained-i} and \ref{proj->unchained-iii} in Claim~\ref{claim:proj->unchained}, together with the fact that \( \Sloc \) satisfies the first item of Definition~\ref{def:projections}. Thus \( \Sloc' \) is a family of local domains.

Given \( g \in \Sym( \Sloc )\) we let \( g' = \rho \circ g \circ \rho^{-1} \). Then \( g' \in \Sym(\Sloc') \) because \( \rho \colon \Sloc \to \Sloc' \) is a bijection, and the map \( g \mapsto g' \) is an isomorphism between \( \Sym(\Sloc ) \) and \( \Sym(\Sloc' ) \) (where both groups are endowed with the pointwise convergence topology with respect to the discrete topology on \( \Sloc \) and \( \Sloc' \), respectively). It remains to show that its restriction to \( \Wr^{\Sloc,\bpi}_{\delta \in \Delta} H_\delta \) is onto \( \Wr^{\Sloc'}_{\delta \in \Delta} H_\delta \).

By Claim~\ref{claim:proj->unchained}\ref{proj->unchained-ii}, for every \( \gamma \geq \delta \) and \( w \in \Sloc_\delta \) 
\begin{equation} \label{eq:proj->unchained-ii}
\rho(w)|_\gamma = \rho(\pi_{\delta \gamma}(w)).
\end{equation}
Also, if \( w' = \rho(w) \) then by definition of \( g' \)
\begin{equation} \label{eq:g'}
g'(w') = g'(\rho(w)) =  \rho (g(w)).
\end{equation}
As \( \rho \restriction \Sloc_\delta \) is a bijection between \( \Sloc_\delta \) and \( \Sloc'_\delta \), it follows that \( g (\Sloc_\delta) = \Sloc_{\delta } \) if and only if \( g'(\Sloc'_\delta) = \Sloc'_\delta \).
Moreover, by equations~\eqref{eq:proj->unchained-ii} and~\eqref{eq:g'}, we have that for every \( z \in \Sloc_\delta \) and \( z' = \rho(z) \) 
\begin{align*}
g'(z'|_\gamma) & = g'(\rho(z)|_\gamma) = g'(\rho(\pi_{\delta \gamma}(z))) = \rho(g(\pi_{\delta \gamma}(z))) \quad \text{and} \\
g'(z')|_\gamma & = \rho(g(z))|_\gamma = \rho(\pi_{\delta \gamma}(g(z))).
\end{align*}
Since \( \rho \) is injective, it follows that 
\[ 
g'(z'|_\gamma)  = g'(z')|_\gamma \iff \rho(g(\pi_{\delta \gamma}(z))) = \rho(\pi_{\delta \gamma}(g(z))) \iff g(\pi_{\delta \gamma}(z)) = \pi_{\delta \gamma}(g(z)). 
\]
Therefore \( g \) satisfies~\ref{Wr1projective} if and only if \( g' \) satisfies~\ref{Wr1generalizedlocal}.

By items~\ref{proj->unchained-i} and~\ref{proj->unchained-iii} of Claim~\ref{claim:proj->unchained} we further have \( \rho(z^\delta_i) = \rho(z)^\delta_i \) for every \( i \in N_\delta \). Therefore, using also~\eqref{eq:g'}
\[
g'(\rho(z)^\delta_i) = g'(\rho(z^\delta_i)) = \rho(g(z^\delta_i)),
\]
hence using again Claim~\ref{claim:proj->unchained}\ref{proj->unchained-iii} we get 
\[ 
g'(\rho(z)^\delta_i)(\delta) = \rho(g(z^\delta_i))(\delta) = g(z^\delta_i)(\delta) . 
\]
This means that the bijections \( i \mapsto g(z^\delta_i)(\delta) \) and \( i \mapsto g'((z')^\delta_i)(\delta) \) are identical, hence \( g \) satisfies~\ref{Wr2projective} if and only if \( g' \) satisfies~\ref{Wr2generalizedlocal}. This concludes our proof.
\end{proof}


An important instance where Theorem~\ref{thm:proj->unchained} applies consists of projective generalized wreath products of the form \( \Wr^{\LF, \bpi}_{\delta \in \Delta} H_\delta \) because \( \LF \subseteq \Max \) implies \( \Sloc^{\LF} \subseteq \Sloc^{\Max} \).

\begin{remark} \label{rmk:proj->unchained} 
It follows from the proof of Theorem~\ref{thm:proj->unchained} that \( \Sloc'\) is locally homogeneous if and only if so is \( \langle \Sloc,\bpi \rangle \).
\end{remark}

\subsection{Some useful lemmas}

The following lemma implies that condition~\ref{Wr2projective} in Definition~\ref{def:projective} can be dropped if \( H_\delta = \Sym(N_\delta) \) for all \( \delta \in \Delta \).

\begin{lemma} \label{lem:P2isautomatic}
Let \( \langle \Delta,N \rangle \) and \( \langle \Sloc , \bpi \rangle \) be as in Definition~\ref{def:projective}. Then for every \( \delta \in \Delta \), \( z \in \Sloc_\delta \), and \( g \in \Sym(\Sloc) \) satisfying condition~\ref{Wr1projective}, the map \( i \mapsto g(z^\delta_i)(\delta) \) is a permutation of \( N_\delta \).
\end{lemma}

\begin{proof}
For any \( i,j \in N_\delta \) and \( \gamma > \delta \), we have \( z^\delta_i|_\gamma = z^\delta_j|_{\gamma} \), and therefore \( \pi_{\delta \gamma}(z^\delta_i) = \pi_{\delta \gamma} (z^\delta_j) \) by Definition~\ref{def:projections}\ref{def:projections-a}. By~\ref{Wr1projective}, it follows that \( \pi_{\delta \gamma} (g(z^\delta_i)) = \pi_{\delta \gamma}(g(z^\delta_j)) \), hence \( g(z^\delta_i)|_\gamma = g(z^\delta_j)|_\gamma \) by Definition~\ref{def:projections}\ref{def:projections-a} again. On the other hand, by injectivity of \( g \) we have that if \( i \neq j \) then \( g(z^\delta_i) \neq g(z^\delta_j) \), which combined with the discussion above yields \( g(z^\delta_i)(\delta) \neq g(z^\delta_j)(\delta) \). This proves that the map \( i \mapsto g(z^\delta_i)(\delta) \) is injective. 

To prove surjectivity, given any \( j \in N_\delta \) let \( z' = g(z)^\delta_j \) and, using surjectivity of \( g \), set \( z'' = g^{-1}(z') \): we need to show that \( z'' = z^\delta_i \) for some \( i \in N_\delta \). If this was not the case, \( z''|_\gamma \neq z|_\gamma \) for some  \( \gamma > \delta \). 
Arguing as above, we then would get \( z'|_\gamma = g(z'')|_\gamma \neq g(z)|_\gamma \), contradicting the choice of \( z' \).
\end{proof}

It follows that when each \( H_\delta \) is the full permutation group \( \Sym(N_\delta) \), we can conveniently reformulate \( \Wr^{\Sloc,\bpi}_{\delta \in \Delta} H_\delta \) in terms of certain automorphisms of a partial order \( P^{\Sloc,\bpi} \) canonically associated to it. 
More in detail, let \( P^{\Sloc,\bpi} \) be obtained by endowing \( \Sloc \) with the ordering \( \preceq \) defined as follows: if \( z \in \Sloc_\delta \) and \( z' \in \Sloc_\gamma \), then \( z \preceq z' \) if and only if \( \delta \leq \gamma \) and \( \pi_{\delta \gamma}(z) = z' \). 
In particular, if \( \bpi \) is trivial then \( z \preceq z' \iff z \supseteq z' \) for every \( z,z' \in \Sloc \). 
We refer to \( P^{\Sloc,\bpi} = (\Sloc, {\preceq}) \) as the \markdef{canonical partial order} associated to \( \Wr^{\Sloc,\bpi}_{\delta \in \Delta} \Sym(N_\delta) \). 
When \( \Sloc \) and \( \bpi \) are clear from the context, we might drop them from the notation and simply write \( P \) instead of \( P^{\Sloc, \bpi}\), in which case we also conflate \( P \) with its domain and denote its order by \( \leq_P \) instead of \( \preceq \).
We say that an automorphism \( f \in \Aut(P) \) of \( P = P^{\Sloc, \bpi} \) is \markdef{localized} if \( f(\Sloc_{\delta })=\Sloc_{\delta } \) for every \( \delta\in\Delta \).
The subgroup of \( \Aut(P) \) consisting of all localized automorphisms will be denoted by \( \AutL(P) \).
Notice that since \( P=P^{\Sloc,\bpi} \) is a countable partial order when \( \Sloc \) is countable, \( \Aut(P) \) can be construed as a closed subgroup of \( \Sym(\omega) \), and \( \AutL(P) \) becomes a closed subgroup of \( \Aut(P) \subseteq \Sym(\omega) \).
By Lemma~\ref{lem:P2isautomatic}, we obviously get \( \AutL(P) =  \Wr^{\Sloc, \bpi}_{\delta \in \Delta} \Sym(N_\delta) \).

To prove our main results, we need to connect generalized wreath products with automorphism groups of \( L \)-trees.
Let \( \tr = (T,{\leq_T},\lev_T) \) be a countable pruned \( L \)-tree.
Recall that since \( \tr \) can be viewed as a first-order structure with domain \( \omega \), its automorphism group \( \Aut(\tr) \) can be construed as a closed subgroup of \( \Sym(\omega) \), from which it inherits the usual product topology.
Consider the condensed tree \( \Delta(\tr) \) of \( \tr \). 
We define a labeling \( N_\tr \colon \Delta(\tr) \to \omega \cup \{ \omega \} \colon \delta \mapsto N^\tr_\delta \) as follows: for \( \delta \in \Delta(\tr) \), fix any \( t \in \delta \) and set \( N^\tr_\delta = |C_t| \), 
where 
\begin{equation} \label{eq:Ct} 
C_t =\{ t\}\cup\{ t' \sim t \mid \spl{t,t'} = \lev_T(t) \} .
\end{equation}
It is clear that the definition of \( N^\tr_\delta \) does not depend on the choice of \( t \in \delta \).

In the next sections, we are going to show that \( \Aut(\tr) \cong \Wr^{\Sloc,\bpi}_{\delta \in \Delta(\tr)} \Sym(N^\tr_\delta) \), under various assumptions on \( \tr \) and for the appropriate choices of \( \Sloc \) and \( \bpi \). To this aim, we will often
employ the following lemma, which provides a sufficient condition to obtain \( \Aut(\tr) \cong \AutL(P) \) for \( P \) the canonical partial order associated to \( \Wr^{\Sloc, \bpi}_{\delta \in \Delta(\tr)} \Sym(N^\tr_\delta) \).

\begin{lemma} \label{lemma:g} 
Let \( \tr = (T,{\leq_T},\lev_T) \) be a countable pruned \( L \)-tree, and let \( P = P^{\Sloc, \bpi} \) be the canonical partial order associated to the projective wreath product \( \Wr^{\Sloc, \bpi}_{\delta \in \Delta(\tr)} \Sym(N^\tr_\delta) \).
Let
\( g \colon T \to P \) be an order-isomorphism
such that \( g(t) \in \Sloc_{[t]} \) for every \( t \in T \).
Then the map \( \varphi \mapsto f_\varphi \) defined by setting \( f_\varphi = g \circ \varphi \circ g^{-1} \) witnesses \( \Aut(\tr) \cong \AutL(P) \). Therefore, \( \Aut(\tr) \cong \Wr^{\Sloc, \bpi}_{\delta \in \Delta(\tr)} \Sym(N^\tr_\delta) \).
\end{lemma}

\begin{proof}
Let \( \varphi \in \Aut(\tr) \).
It is clear that \( f_\varphi \in \Aut(P) \). 
Given any \( z\in P \), let \( \delta \in \Delta(\tr) \) be such that \( z \in \Sloc_\delta \) and \( t  = g^{-1}(z) \), so that \( [t] = \delta \) by choice of \( g \). 
Since \( t \sim \varphi(t) \), we get that \( [\varphi(t)] = [t] =  \delta \), and hence \( f_\varphi(z) \in \Sloc_\delta \) by choice of \( g \) again. 
Since \( z\in P \) was arbitrary, this yields \( f_\varphi \in \AutL(P) \). 
It easily follows that \( \varphi \mapsto f_\varphi \) is a topological group embedding of \( \Aut(\tr) \) into \( \AutL(P) \): we claim that it is also onto. 

Given \( f \in \AutL(P) \), let \( \varphi_f = g^{-1} \circ f \circ g \), so that \( f_{\varphi_f} = f \): we only need to show that \( \varphi_f \in \Aut(\tr) \). Clearly \( \varphi_f \colon T \to T\) is bijective, and it satisfies \( t \leq_T t' \iff \varphi_f(t) \leq_T \varphi_f(t') \) for all \( t,t' \in T \) because \( g \) is order-preserving.
Fix any \( t \in T \). 
By definition of \( \varphi_f \), we have \( g(\varphi_f(t)) = f(g(t))\).
By \( f \in \AutL(P) \) and choice of \( g \), we must have \( [t] = [\varphi_f(t)]\), so that \( t \sim \varphi_f(t) \) and, in particular, \( \lev_T(t) = \lev_T(\varphi_f(t))\). 
This concludes the proof that \( \varphi_f \in \Aut(\tr) \).
\end{proof}

Conversely, the following lemmas provide a sufficient condition for turning a projective wreath product into the automorphism group of some \( L \)-tree.
A skeleton \( \langle \Delta,N \rangle \) is said to be \markdef{treeable} if there is a linear order \( L \) and a map \( \lev_\Delta \colon \Delta \to L \) such that \( (\Delta,{\leq_\Delta},\lev_\Delta) \) is a \emph{pruned} \( L \)-tree; when we want to single out a specific linear order \( L \) as above, we say that \( \langle \Delta,N \rangle \) is \markdef{\( L \)-treeable}.

\begin{lemma} \label{lem:generalcase}
Let \( P = P^{\Sloc,\bpi} \) be the canonical partial order associated to a projective wreath product of the form \( \Wr^{\Sloc,\bpi}_{\delta \in \Delta} \Sym(N_\delta)\), for some countable skeleton \( \langle \Delta,N \rangle \) and some system of projections \( \langle \Sloc,\bpi \rangle \) with \( \Sloc \subseteq \Sloc^{\Max} \). 
Suppose that \( \langle \Delta,N \rangle \) is \( L\)-treeable for some linear order \( L \) without a maximum.
Then there is a map \( \lambda_P \colon P \to L \) such that \( (P,{\leq_P},\lev_P) \) is a pruned \( L \)-tree and \( \lambda_P \) is constant on each \( \Sloc_\delta \).
\end{lemma}

\begin{proof}
Let \( \lev_\Delta \colon \Delta \to L \) be such that \( (\Delta,{\leq_\Delta},\lev_\Delta) \) is a pruned \( L \)-tree, and let \( \lev_P(z) = \lev_{\Delta}(\delta) \) for every \( z \in \Sloc_\delta \), so that by construction \( \lev_P(z) = \lev_P(z') \) whenever \( z \) and \( z' \) belong to the same \( \Sloc_\delta \). Items~\ref{def:L-tree-1} and~\ref{def:L-tree-2} of Definition~\ref{def:L-tree} are obviously satisfied, because they are satisfied by the \( L \)-tree \( (\Delta,{\leq_\Delta},\lev_\Delta) \); in particular, for \( z \in \Sloc_\delta \) and \( \ell \geq \lev_P(z) = \lev_\Delta(\delta) \), the unique \( z'' \in P \) such that \( z'' \geq_P z \) and \( \lev_P(z'') = \ell \) is \( \pi_{\delta\gamma}(z) \), for \( \gamma = \delta|_\ell \). 

To check that \ref{def:L-tree-3} is satisfied, pick any \( z,z' \in P = \Sloc \), and let \( \delta,\delta' \in \Delta \) be such that \( z \in \Sloc_\delta \) and \( z' \in \Sloc_{\delta'} \). 
Since \( (\Delta,{\leq_\Delta},\lev_{\Delta}) \) is an \( L \)-tree, there is \( \gamma \in \Delta \) such that \( \gamma \geq_{\Delta} \delta \) and \( \gamma \geq_{\Delta} \delta' \). 
Replacing \( z \) and \( z' \) with \( \pi_{\delta\gamma}(z) \) and \( \pi_{\delta'\gamma}(z') \), respectively, we can assume without loss of generality that \( \delta = \delta' \). Since \( L \) has no maximum and \( \supp(z),\supp(z') \in \Max \), there is \( \gamma \geq_{\Delta} \delta \) such that \( z(\gamma') = z'(\gamma') = 0 \) for all \( \gamma' \geq_{\Delta} \gamma \). This means that \( z|_\gamma = z'|_\gamma \), and therefore \( \pi_{\delta \gamma}(z) = \pi_{\delta \gamma}(z') \) by condition~\ref{def:projections-a} in the second item of Definition~\ref{def:projections}. 

Next we check that also item~\ref{def:L-tree-4} of Definition~\ref{def:L-tree} is satisfied. Consider two \( \leq_P \)-incomparable \( z \in \Sloc_{\delta} \) and \( z' \in \Sloc_{\delta'} \), for some \( \delta, \delta' \in \Delta \). 

Suppose first that \( \delta \) and \( \delta' \) are \( \leq_{\Delta} \)-comparable, say \( \delta \leq_{\Delta} \delta' \) for the sake of definiteness. Replacing \( z \) with \( \pi_{\delta \delta'}(z) \), we can assume \( \delta = \delta' \). Let \( \bar \gamma = \max \{ \gamma \geq_{\Delta} \delta \mid z(\gamma) \neq z'(\gamma) \} \): such a maximum exists because by hypothesis \( \supp(z), \supp(z') \in \Max \).  
Then for every \( \gamma' \geq_{\Delta} \delta \), we have \( z|_{\gamma'} = z'|_{\gamma'} \) if and only if \( \gamma' >_{\Delta} \bar \gamma \).
Using again condition~\ref{def:projections-a} in the second item of Definition~\ref{def:projections}, this means
that \( \pi_{\delta \gamma'}(z) = \pi_{\delta \gamma'}(z') \) if and only if \( \gamma' >_{\Delta} \bar \gamma \), therefore \( \spl{z,z'}\) is well-defined and equals \( \lev_{\Delta}(\bar \gamma)\).

Assume now that \( \delta \) and \( \delta' \) are \( \leq_{\Delta} \)-incomparable, and let
\( \ell = \spl{\delta,\delta'} \).
If \( \pi_{\delta \gamma}(z) = \pi_{\delta' \gamma}(z') \) for every \( \ell' > \ell \) and \( \gamma = \delta|_{\ell'} = \delta'|_{\ell'} \), then clearly \( \ell = \spl{z,z'} \). Otherwise we can replace \( z \) and \( z' \) with \( \pi_{\delta \gamma}(z) \) and \( \pi_{\delta'\gamma}(z') \), respectively, where \( \gamma = \delta|_{\ell'} = \delta'|_{\ell'} \) for some \( \ell' > \ell \) such that \( \pi_{\delta \gamma}(z) \neq \pi_{\delta'\gamma}(z') \). This shows that without loss of generality we can again assume \( \delta = \delta' \), and thus we can argue as in the previous paragraph. 

Finally, the fact that the \( L \)-tree \( (P,{\leq_P},\lev_P) \) is pruned follows from the fact that \( (\Delta,{\leq_\Delta},\lev_\Delta) \) is pruned, together with the surjectivity of the maps \( \pi_{\delta \gamma}\) and the fact that \( L \) is countable (and hence has countable coinitiality, if it does not have a minimum).
\end{proof}

\begin{lemma} \label{lemma:fromwreathtotrees}
Let \( P = P^{\Sloc,\bpi} \) be the canonical partial order associated to the projective wreath product \( \Wr^{\Sloc,\bpi}_{\delta \in \Delta} \Sym(N_\delta) \), where \( \langle \Sloc,\bpi \rangle \) is a countable system of projections. 
Suppose that there are a linear order \( L \) without minimum and a map \( \lev_P \colon P \to L \) such that \( (P,{\leq_P},\lev_P) \) is a pruned \( L \)-tree and \( \lev_P \) is constant on \( \Sloc_\delta \), for every \( \delta \in \Delta \).
Then there is a countable pruned \( L' \)-tree \( \tr_P = (T,{\leq_T},\lev_T) \), where \( L' = \omega^* \cdot L \), such that \( \AutL(P) \cong \Aut(\tr_P) \), and therefore \( \Wr^{\Sloc,\bpi}_{\delta \in \Delta} \Sym(N_\delta) \cong \Aut(\tr_P) \).

\details{
{\color{blue}Moreover, let \( \lev_\Delta \colon \Delta \to L \) be defined by \( \lev_\Delta(\delta) = \lev_P(z) \) for some/any \( z \in \Sloc_\delta \), and assume that \( \langle \Sloc,\bpi \rangle\) is locally homogeneous. 
Then \( (\Delta,{\leq_\Delta},\lev_\Delta) \) is a weak \( L \)-tree, and if it is an \( L \)-tree then \( \tr_P\) can be required to be special.} 
}
\end{lemma}

\begin{proof}
Consider the product \( T_0 =  \omega^* \times P = \omega^* \times \Sloc\) equipped with the antilexicographic ordering.
For each \( (n,z) \in T_0 \), let \( \lev_{T_0}(n,z) = (n,\lev_P(z)) \in L'\). 
Then \( \tr_0 = (T_0,{\leq_{T_0}}, \lev_{T_0})  \) is a countable pruned \( L' \)-tree because \( (P,{\leq_P},\lev_P) \) is a countable pruned \( L \)-tree.
Now fix a bijection \( \# \colon \Delta \to \omega \setminus \{ 0 \} \), and for every \( z \in P = \Sloc \) let \( \# (z) = \# (\delta) \) for the unique \( \delta \in \Delta \) such that \( z \in \Sloc_\delta \). 
Let \( T \) be obtained from \( T_0 \) by adding for each \( z \in P \) and \( \ell \leq_{L'} (\#(z),\lev_P(z))\) a new element \( z_\ell \) with \( \lev_T(z_\ell) = \ell \). 
Let \( z_\ell \leq_T y_{\ell'} \iff z = y \wedge \ell \leq_{L'} \ell' \), and for every \( (n,y) \in T_0 \), let \( z_\ell \leq_T (n,y) \iff z <_P y \vee ( z = y \wedge n < \#(z)) \) (see Figure~\ref{fig:T_P}).
We claim that \( \tr_P = (T,{\leq_T},\lev_T) \) is as required.

\begin{figure} \centering
\hspace{0.5cm}
\begin{tikzpicture}
\draw [fill] (0,0) circle [radius=0.05];
\node at (0,0.4) {\( (0,z) \)};
\draw [fill] (0,-1) circle [radius=0.05];
\node [right] at (0.2,-1) {\( (1,z) \)};
\draw [thick] (0,0) to (0,-1);
\node at (0,-1.4) {\( \vdots \)};
\draw [fill] (0,-2) circle [radius=0.05];
\draw [fill] (0,-3) circle [radius=0.05];
\draw [fill] (0,-4) circle [radius=0.05];
\node [right] at (0.2,-4) {\( t_z = (\#(z),z)  \)};
\draw [fill] (0,-5) circle [radius=0.05];
\draw [thick] (0,-2) to (0,-5);
\draw [fill] (-1,-4) circle [radius=0.05];
\node [left] at (-1.2,-4) {\( z_{(\#(z),\lev_P(z))} \)};
\draw [fill] (-1,-5) circle [radius=0.05];
\draw [thick] (0,-3) to (-1,-4) to (-1,-5);
\node at (0,-5.4) {\( \vdots \)};
\node at (-1,-5.4) {\( \vdots \)};
\draw [fill] (-1,-6) circle [radius=0.05];
\node [left] at (-1.2,-6) {\( z_{(n,\lev_P(z))} \)};
\draw [fill] (-1,-7) circle [radius=0.05];
\node [left] at (-1.2,-7) {\( z_{(n+1,\lev_P(z))} \)};
\draw [fill] (0,-6) circle [radius=0.05];
\node [right] at (0.2,-6) {\( (n,z) \)};
\draw [fill] (0,-7) circle [radius=0.05];
\node [right] at (0.2,-7) {\( (n+1,z) \)};
\draw [thick] (0,-6) to (0,-7);
\draw [thick] (-1,-6) to (-1,-7);
\node at (0,-7.4) {\( \vdots \)};
\node at (-1,-7.4) {\( \vdots \)};
\node at (-1,-8.4) {\( \vdots \)};
\draw [fill] (-1,-9) circle [radius=0.05];
\node [left] at (-1.2,-9) {\( z_\ell \)};
\node at (-1,-9.4) {\( \vdots \)};
\draw [<->] (3.2,0) to (3.2,-7.9);
\draw [dotted] (-0.2,0) to (3.5,0);
\draw [dotted] (-0.2,-7.9) to (3.5,-7.9);
\node [right] at (3.3,-3.9) {\(  \omega^* \times \{ \lev_P(z) \} \)};
\draw [very thick,dashed,->] (-4.5,-10) to (-4.5,0.9);
\node [left] at (-4.7,-0.4) {\(  L' = \omega^* \cdot L \)};
\end{tikzpicture}
\caption{The tree \( \tr_P \) is obtained by replacing each \( z \in P \) with the tree depicted in the figure (with the order described in the text), where \( (0,z) \) is identified with the original \( z \).
The nodes \( z_\ell \) on the left branch are \( \leq_T \)-incomparable with every node which is not strictly \( \leq_T \)-above \( t_z \).}
\label{fig:T_P}
\end{figure}

Clearly, \( \tr_P \) is a countable pruned \( L' \)-tree. 
We need to show that \( \AutL(P) \cong \Aut(\tr_P) \). 
Every \( f \in \AutL(P) \) can be canonically lifted to some \( f^\uparrow \in \Aut(\tr_P) \) by setting \( f^\uparrow(n,z) = (n,f(z)) \) for all \( (n,z) \in T_0 \), and \( f^\uparrow(z_\ell) = f(z)_\ell \) for all \( z_\ell \in T \setminus T_0 \): the latter makes sense because \( f \in \AutL(P) \) entails \( \#(z) = \#(f(z)) \) and \( \lev_P(z) = \lev_P(f(z))\) for all \( z \in P \).
The map \( f \mapsto f^\uparrow \) is a topological group embedding, so we only need to show that it is onto.
Fix any \( h \in \Aut(\tr_P) \). For every \( z  \in P \), consider the node \( t_z = (\#(z),z) \in T_0 \subseteq T \).

\begin{claim} \label{claim:t_p1}
There is \( y \in P \) such that \( \lev_P(y) = \lev_P(z) \), \( \#(y) = \#(z) \), and \( h(t_z) = t_y\).
\end{claim}

\begin{proof}[Proof of the claim]
By construction, there are \( \leq_T \)-incomparable \( t',t'' \leq_T t_z \) with \( \spl{t',t''} <_L \lev_T(t_z) \): indeed, using the fact that \( P \) is pruned and \( L \) has no minimum, we can pick any \( z' <_P z \) and let \( t' = t_{z'} \) and \( t'' = z'_{(\#(z'),\lev_P(z'))} \), so that \( \spl{t',t''} = (\#(z'),\lev_P(z')) <_{L'} (\#(z),\lev_P(z)) = \lev_T(t_z) \). 
In contrast, every \( y_\ell \in T \setminus T_0 \) is such that all \( t',t'' \leq_T y_\ell \) are \( \leq_T \)-comparable. 
Since \( h \) is also level-preserving, it follows that \( h(t_z) = (\#(z),y) \) for some \( y \in P \) with \( \lev_P(y) = \lev_P(z) \).
Since there is \( t \in T \) such that  \( \spl{t_z,t} = (\#(z),\lev_P(z)) \), as witnessed by \( t = z_{(\#(z),\lev_P(z))} \), the same must hold for \( (\#(z),y) \), and thus \( \#(y)= \#(z) \) because \( \#(z) \neq 0 \).
Therefore \( (\#(z),y) = (\#(y),y) = t_y \), and we are done. 
\end{proof}

Let \( f \colon P \to P \) be defined by setting \( f(z) = y \) if and only if \( h(t_z) = t_y \). It is easy to verify that \( f \in \Aut(P) \), and in fact \( f \in \AutL(P) \) because \( \#(z) = \#(f(z)) \) by Claim~\ref{claim:t_p1}.
Finally, the fact that \( h(t_z) = (\#(z),f(z)) \) easily implies that \( h(z_\ell) = f(z)_\ell \) for all relevant \( \ell \in L' \). It follows that \( h = f^\uparrow \), as desired.

\details{\color{blue}
We now come to the additional part.
The fact that \( (\Delta, {\leq_\Delta}, \lev_\Delta) \) is a weak \( L \)-tree easily follows from the fact that \( (P,{\leq_P},\lev_P) \) enjoys the same property. 
Suppose now that \( (\Delta, {\leq_\Delta}, \lev_\Delta) \) is an \( L \)-tree: we need to check that \( \tr_P \) is special. 
Notice that for each \( z \in \Sloc_\delta \subseteq P \), its orbit \( [z] = \{ f(z) \mid f \in \AutL(P) \} \) under the natural action of \( \AutL(P) \) coincides with \( \Sloc_\delta \); in other words, the condensed tree of \( (P,{\leq_P},\lev_P) \) is (canonically isomorphic to) \( (\Delta,{\leq_\Delta},\lev_\Delta) \). 
For \( t \in T \), let \( [t] = \{ h(t) \mid h \in \Aut(\tr_P) \} \) be the orbit of \( t \) under the natural action of \( \Aut(\tr_P) \).
Since \( f \mapsto f^\uparrow\) is an isomorphism between \( \AutL(P) \) and \( \Aut(\tr_P) \), for every \( n \in \omega^* \) and \( \ell \leq_{L'} (\#(z),\lev_P(z))\) we have
\[  
[(n,z)] = \{ n \} \times \Sloc_\delta = \{ (n,z') \mid z' \in \Sloc_\delta \} \quad \text{and} \quad [z_\ell] = \{ z'_\ell \mid z' \in \Sloc_\delta \}.
\]
This shows that, up to isomorphism, \( \Delta(\tr_P) \) 
can be construed as the partial order obtained by applying the above procedure to \( (\Delta,{\leq_\Delta},\lev_\Delta) \) rather than \( (P,{\leq_P},\lev_P) \): first every \( \delta \in \Delta \) is replaced by a decreasing chain of nodes of the form \( (n,\delta) \), for \( n \in \omega^* \), and then we add the nodes of the form \( \delta_\ell \) for every \( \ell \leq_{L'} (\#(\delta),\lev_{\Delta}(\delta)) \). 
It easily follows that \( \Delta(\tr_P) \) is an \( L' \)-tree because so is \( (\Delta,{\leq_\Delta},\lev_\Delta) \).}
\end{proof}

In order to use Lemmas~\ref{lem:generalcase} and~\ref{lemma:fromwreathtotrees} together, we need to show that we can always assume, without loss of generality, that the underlying treeable skeleton of a projective wreath product \( \Wr^{\Sloc,\bpi}_{\delta \in \Delta} \Sym(N_\delta) \) is actually \( L \)-treeable for some linear order \( L \) without minimum and maximum.

\begin{lemma} \label{lem:nominandmax}
Suppose that \( \Wr^{\Sloc,\bpi}_{\delta \in \Delta} \Sym(N_\delta) \) is a projective wreath product over an \( L \)-treeable skeleton \( \langle \Delta,N \rangle \), for some linear order \( L \). Then there is a linear order \( L' \supseteq L \) with neither maximum nor minimum, an \( L' \)-treeable skeleton \( \langle \Delta',N' \rangle \), and a system of projections \( \langle \Sloc',\bpi' \rangle \) such that
\[  
\Wr^{\Sloc,\bpi}_{\delta \in \Delta} \Sym(N_\delta) \cong \Wr^{\Sloc',\bpi'}_{\delta \in \Delta'} \Sym(N'_\delta).
\]
Moreover:
\begin{itemizenew}
\item 
if \( \langle \Delta,N \rangle \) is countable, so is \( \langle \Delta',N' \rangle \);
\item 
if \( \Delta \) is linear, so is \( \Delta' \);
\item 
if \( \Sloc \) is countable, so is \( \Sloc'\), and if \( \Sloc \subseteq \Sloc^{\Max}\), then \( \Sloc' \subseteq \Sloc^{\Max}\) too;
\item 
if \( \bpi \) is trivial then \( \bpi' \) is trivial, and the family of local domains \( \Sloc' \) is full whenever \( \Sloc \) is full;
\item
if \( \langle \Sloc,\bpi \rangle \) is locally homogeneous, then so is \( \langle \Sloc',\bpi' \rangle \).
\end{itemizenew}
\end{lemma}

\begin{proof}
Let \( \lev_\Delta \colon \Delta \to L \) witness that \( \langle \Delta,N \rangle \) is treeable.
If \( L \) has a maximum, we can canonically extend the pruned \( L \)-tree \( ( \Delta, {\leq_\Delta}, \lev_{\Delta} ) \) to a pruned \( L'\)-tree \( (\Delta', {\leq_{\Delta'}}, \lev_{\Delta'} ) \), where \( L' = L+\omega \), \( \Lev_{\ell'}(\Delta') = \{ \delta_{\ell'} \} \) is a singleton for every \( \ell' \in L' \setminus L \), and \( \delta_{\ell'} \geq_{\Delta'} \delta' \) for every \( \delta ' \in \Delta' \) with \( \lev_{\Delta'}(\delta') \leq \ell' \). Further setting \( N'_\delta = N_\delta \) for \( \delta \in \Delta \) and  \( N'_{\delta_{\ell'}} = 1 \) for all \( \ell' \in L' \setminus L \), we get that the skeleton \( \langle \Delta',N' \rangle \) is \( L' \)-treeable, as witnessed by \( \lev_{\Delta'} \). 
Then for every \( \delta \in \Delta \) we define \( \iota_\delta \colon \Sloc_\delta \to \prod_{\gamma \geq_{\Delta'} \delta } N'_\delta \)  by setting \( \iota_\delta(z)(\gamma) = z(\gamma) \) if \( \gamma \in \Delta \) and \( \iota_\delta(z)(\gamma) = 0  \) otherwise. If we let \( \Sloc'_\delta \) be the range of such map, then \( \iota_\delta  \) is a bijection between \( \Sloc_\delta \) and \( \Sloc'_\delta \). For \( \ell' \in L' \setminus L \), instead, we have no other choice than letting \( \Sloc'_{\delta_{\ell'}} \) be the singleton containing the only sequence with empty support.
Finally, we can extend \( \bpi = (\pi_{\delta \gamma})_{\gamma \geq_\Delta \delta} \) to \( \bpi' = (\pi'_{\delta \gamma})_{\gamma \geq_{\Delta'} \delta}\) in the obvious way: since \( \Sloc'_{\delta_{\ell'}}\) contains only one element for every \( \ell' \in L' \setminus L \), then \( \pi'_{\delta \delta_{\ell'}} \) needs to be a constant map for every \( \delta \le_{\Delta'} \delta_{\ell'} \), and indeed \( \pi'_{\delta \delta_{\ell'}}(z) = z|_{\ell'} \) for every \( z \in \Sloc'_\delta \); while for \( \delta,\gamma \in \Delta \) we let \( \pi'_{\delta \gamma } = \iota_\gamma \circ \pi_{\delta \gamma} \circ \iota_\delta^{-1} \). 
It is clear that, by construction, \( \Wr^{\Sloc,\bpi}_{\delta \in \Delta} \Sym(N_\delta) \cong \Wr^{\Sloc',\bpi'}_{\delta \in \Delta'} \Sym(N'_\delta )\), as witnessed by the isomorphism \( g \mapsto g' \) defined by setting \( g'(z) = \iota_\delta(g(\iota_\delta^{-1}(z)))\) for every \( \delta \in \Delta \) and \( z \in \Sloc'_\delta \), and \( g'(z) = z \) for the unique \( z \in \Sloc'_{\delta_{\ell'}} \) if \( \ell' \in L' \setminus L \).

The previous paragraph shows that, without loss of generality, we can assume that \( L \) has no maximum.
Suppose now that \( L \) has a minimum \( \ell = \min L \) (otherwise we are done). Then we can replace \( L \) with \( L' = \omega^* + L  \), \( \Delta \) with \( \Delta' = \Delta \cup \{ (\delta ,n) \mid \delta \in \Lev_\ell(\Delta), n \in \omega^* \} \), ordered so that \( \{ (\delta,n) \mid n \in \omega^* \} \) forms a decreasing chain below every \( \delta \in \Lev_\ell(\Delta) \):
\( (\delta,n) \leq_{\Delta'} \delta' \) if and only if either \( \delta' = (\delta,m) \) for some \( m \leq n \) or else \( \delta' \in \Delta \) and \( \delta \leq_{\Delta} \delta' \).
Let also \( \lev_{\Delta'}\) be the extension of \( \lev_{\Delta} \) obtained by setting \( \lev_{\Delta'}(\delta') = n \in \omega^*\) for every \( \delta' = (\delta,n) \in \Delta' \setminus \Delta \).
The map \( N \) on \( \Delta \) is then extended to a map \( N'\) on \( \Delta' \) by letting \( N'_{\delta'} = 1 \) for every \( \delta' \in \Delta' \setminus \Delta \). In this way, \( \langle \Delta', N' \rangle \) is an \( L' \)-treeable skeleton, as witnessed by \( \lev_{\Delta'} \). 
Set \( \Sloc'_\delta = \Sloc_\delta \) for every \( \delta \in \Delta \), and for \( \delta' = (\delta,n) \in \Delta' \setminus \Delta \) let \( \Sloc'_{\delta'} = \{ z \in \prod_{\gamma \geq_{\Delta'} \delta'} N'_\gamma \mid z|_{\delta} \in \Sloc_\delta \} \). 
Finally, extend \( \bpi = (\pi_{\delta \gamma})_{\gamma \geq_{\Delta} \delta}\) to \( \bpi' = (\pi'_{\delta \gamma})_{\gamma \geq_{\Delta'} \delta }\) by letting \( \pi'_{\delta' \gamma } \) be trivial for every \( \delta' = (\delta,n) \in \Delta' \setminus \Delta \) and \( \delta' \leq_{\Delta'} \gamma \leq_{\Delta'} \delta \).
By construction, \( \Wr^{\Sloc,\bpi}_{\delta \in \Delta} \Sym(N_\delta) \cong \Wr^{\Sloc',\bpi'}_{\delta \in \Delta'} \Sym(N'_\delta)\) via the isomorphism \( g \mapsto g' \) defined by letting \( g' \) be the extension of \( g \) such that for every \( \delta' = (\delta,n) \in \Delta' \setminus \Delta \) and \( z \in \Sloc'_{\delta'} \), \( g'(z) \) is the unique \( z' \in \Sloc'_{\delta'} \) such that \( z'|_\delta = g(z|_\delta) \). 

The ``moreover'' part holds by construction, and thus the proof is complete.
\end{proof}

\section{Main results} \label{sec:main}

To turn the automorphism group of a countable pruned \( L \)-tree \( \tr \) into a generalized wreath product, we need to employ a specific labeling procedure which turns nodes of \( \tr \) into corresponding generalized sequences with finite support. To this aim, we will repeatedly use Lemma~\ref{lem:labelsfinitesupport} below.

Let \( \tr = (T, {\leq_T}, \lev_T) \) be a weak \( L \)-tree. An \markdef{upward closed chain} \( C \) is a subset of \( T \) all of whose elements are \( \leq_T \)-comparable and such that if \( t \in C \) and \( t' \geq_T t \), then \( t' \in C \). If \( C \subseteq T \) is an upward closed chain and \( t \in T \), we write \( t \leq_T C \) (respectively, \( t <_T C \)) if \( t \leq_T t'\) (respectively, \( t <_T t' \)) for all \( t' \in C \). We say that \( C \) is \markdef{proper} if there is \( t \in T \) such that \( t <_T C \). Notice that if \( C \subseteq T\) is a (proper) upward closed chain in \( \tr \), then \( [C] = \{ [t] \mid t \in C \} \subseteq \Delta(T) \) is a (proper) upward closed chain in \( \Delta(\tr) \). For every \( \delta \in \Delta(T) \) with \( \delta <_{\Delta(T)} [C] \), we let
\[  
\tr^\delta_C = \{ t \in T \mid {\delta \leq_{\Delta(T)} [t]} \, \wedge \, {t <_T C} \}.
\]

\begin{lemma} \label{lem:labelsfinitesupport}
Let \( \tr = (T, {\leq_T}, \lev_T) \) be a countable pruned \( L \)-tree, and consider the associated domain \( S^{\LF} \subseteq \prod_{\delta \in \Delta(\tr)} N^\tr_\delta \). 
Let \( C \subseteq T \) be a (possibly empty) proper upward closed chain, and \( \bar\delta \in \Delta(T) \) be such that \( \bar\delta <_{\Delta(T)} [C] \). 
Let \( \bar z \in \prod_{\gamma \in [C]} N^{\tr}_\gamma\) be such that \( \supp( \bar z ) \) is finite, and for every \( \gamma \) such that \( \bar\delta \leq_{\Delta(T)} \gamma <_{\Delta(T)} [C] \) let \( \Sloc^{\LF}_{\gamma,\bar z} = \{ z \in \Sloc^{\LF}_\gamma \mid z \supseteq \bar z \}  \). 
Then there is a bijection \( t \mapsto z_t \) between \( \tr^{\bar \delta}_C \) and \(\bigcup \{ \Sloc^{\LF}_{\gamma,\bar z} \mid \bar \delta \leq_{\Delta(T)} \gamma <_{\Delta(T)} [C] \} \) such that \( z_t \in \Sloc^{\LF}_{[t]} \) and \( t \leq_T t' \iff z_t \supseteq z_{t'}\) for every \( t ,t' \in \tr^{\bar\delta}_C \).
\end{lemma}

\begin{proof} 
Let \( \bar \ell = \lev_{\Delta(T)}(\bar \delta) \), and let \( (t_i)_{i < I} \) be an enumeration without repetitions of \( \tr_C^{\bar \delta } \cap \bar \delta = \{ t \in \bar \delta \mid t <_T C \} \), for the appropriate \( I \leq \omega \). 
For \( t \in \tr^{\bar \delta}_C \), let 
\[
i_{t} = \min \{ i < I \mid t \geq_T t_i \} 
\]
and
\[
\Lambda(t) = |\{ t' \in C_t \mid i_{t'} < i_t \}|,
\]
where \( C_t \) is as in~\eqref{eq:Ct}.
Since \( C_t \subseteq \tr^{\bar \delta}_C \) and \( |C_t| = N^{\tr}_{[t]} \), the function \( \Lambda \) maps bijectively each \( C_t \)
onto the corresponding \( N^{\tr}_{[t]}\).

For \( t \in \tr^{\bar \delta}_C \), let \( z_t \in \prod\nolimits_{\gamma \geq_{\Delta(T)} [t]} N^{\tr}_{\gamma} \) be defined by \( z_t(\gamma) = \bar z (\gamma)\) if \( \gamma \in [C] \), and \( z_t(\gamma) = \Lambda(t|_{\lev_{\Delta(T)}(\gamma)}) \) if \( [t] \leq_{\Delta(T)} \gamma <_{\Delta(T)} [C] \).
By construction, if \( t \leq_T t'\) then \( z_t \supseteq z_{t'} \).

First we show that \( z_t \in \Sloc^{\LF}_{[t]}\), which amounts to showing that \( \supp(z_t) \) is finite.
To this aim, it is enough to consider the case where \( t = t_i \in \tr_C^{\bar \delta } \cap \bar \delta \) and argue by induction on \( i < I \). 
The case \( i = 0 \) is easy: \( \Lambda(t_0|_{\lev_{\Delta(T)}(\gamma)}) = 0 \) for all \( \bar \delta \leq_{\Delta(T)} \gamma <_{\Delta(T)} [C] \), hence \( \supp(z_{t_0}) = \supp(\bar z) \) is finite. 
For \( i > 0 \), let \( \ell_i = \min_L \{ \spl{t_i,t_j} \mid j < i \} \), which exists because there are only finitely many \( j \) to be considered. 
By definition of \( \tr^{\bar \delta}_C \), we have \( \bar \ell \leq_L \ell_i < \lev_T(t') \) for every \( t' \in C \).
Then \( \Lambda(t_i|_{\ell}) = 0 \) for \( \bar \ell \leq_L \ell <_L \ell_i \), while there is \( \bar\jmath < i \) such that \( t_i|_\ell = t_{\bar\jmath}|_\ell \), and hence \( \Lambda(t_i|_\ell) = \Lambda(t_{\bar\jmath}|_\ell) \), for all \( \ell >_L \ell_i \).
It follows that \( |\supp(z_{t_i})| \leq |\supp(z_{t_{\bar\jmath}})| + 1 \), and since \(\supp(z_{t_{\bar\jmath}}) \) is finite by inductive hypothesis, \( \supp(z_{t_i}) \) is finite as well.

We already observed that, by construction, \( t \leq_T t' \) implies \( z_t \supseteq z_{t'} \), for all \( t ,t' \in \tr^{\bar \delta}_C \);
we now show that \( t \nleq_T t' \) implies \( z_t \not\supseteq z_{t'} \). If \( t' <_T t \), then the result easily follows from 
\( [t'] <_{\Delta(T)} [t] \).
Suppose now that \( t \) and \( t' \) are \( \leq_T \)-incomparable, and let \( \ell=\spl{t,t'} \). Since
\( t,t' \in \tr^{\bar \delta}_C \), we have \( \bar \delta \leq_{\Delta(T)} [t|_\ell] , [t'|_\ell] <_{\Delta(T)} [C] \), hence in particular \( t|_\ell \sim t'|_{\ell} \) because \( \Delta(\tr) \) is a weak \( L \)-tree.
It follows that \( t'|_\ell \in C_{t|_\ell} \), and thus by construction \( \Lambda(t|_\ell) \neq \Lambda(t'|_\ell)\). Therefore, \( z_t(\gamma) \neq z_{t'}(\gamma) \) for \( \gamma = [t|_\ell] = [t'|_\ell] \), and in particular \( z_t \not\supseteq z_{t'} \).

It remains to prove that \( t \mapsto z_t \) is bijective. 
Injectivity follows from the equivalence \( t \leq_T t' \iff z_t \supseteq z_{t'} \).
As for surjectivity, it is enough to show that for every \( z \in \Sloc^{\LF}_{\bar \delta,\bar z} \) there is \( i < I \) such that \( z = z_{t_i} \).
Let \( k = | \supp(z) \setminus \supp(\bar z) | \). We argue by induction on \( k \in \omega \). 
If \( k = 0 \), then by construction \( z= z_{t_0} \). 
Assume now that \( k > 0 \), and let \( \ell \) be \( L \)-smallest such that \(  \bar \delta|_{\ell} \in  \supp(z) \setminus \supp(\bar z) \). 
By inductive hypothesis, there is \( i < I \)  such that \( z' = z_{t_{i}}\), where \( z' \in \Sloc^{\LF}_{\bar\delta,\bar z} \) is such that \( z'(\bar \delta|_\ell) = 0 \) and \( z'(\bar \delta|_{\ell'}) = z(\bar \delta|_{\ell'})  \) for all \( \ell'\geq_L \bar \ell \) with \( \ell' \neq \ell \). 
Let \( j = i_{t} \) for the unique \( t \in C_{t_i|_\ell}\) such that \( \Lambda(t) = z(\bar \delta|_\ell) \): then \( z_{t_j} = z \), as desired.
\end{proof}


\subsection{Homogeneous spaces}\label{sec:homogeneous}

Informally speaking, a mathematical structure is (one-point) homogeneous if for all points \( x , y \) in the structure, if the map sending \( x \) to \( y \) is a partial automorphism, then there is an automorphism of the whole structure which sends \( x \) to \( y \).
In the metric context, this translates to the following: 
a Polish ultrametric space \( U \) is \markdef{homogeneous} if \( \Iso(U) \) acts transitively on \( U \), that is, for every \( x,y \in U \) there is \( \psi \in \Iso(U) \) such that \( \psi(x) = y \).
Moving to the realm of weak \( L \)-trees \( \tr \), instead, we have that \( \tr \) is \markdef{homogeneous} if \( \Delta(\tr) \) is linear, i.e., for every \( t ,t' \in T \) with \( \lev_T(t) = \lev_T(t') \) there is \( f \in \Aut(\tr) \) such that \( f(t) = t' \). 
When \( L \) is countable, this is equivalent to requiring a homogeneity property of the body of \( \tr \), as shown in the next lemma.

\begin{lemma} \label{lem:homogeneoustree}
Suppose that \(\coi(L) \leq \aleph_0 \), and let \( \tr \) be a weak \( L \)-tree.
Then the following are equivalent:
\begin{enumerate-(1)}
\item \label{lem:homogeneoustree-1}
\( \tr \) is homogeneous;
\item \label{lem:homogeneoustree-2}
\( \tr \) is pruned, and for every \( b,b' \in [\tr] \) there is \( f \in \Aut(\tr) \) which takes \( b \) to \( b' \), i.e.\ \( f(b(\ell)) = b'(\ell) \) for all \( \ell \in L \).
\end{enumerate-(1)}
\end{lemma}

\begin{proof} 
\ref{lem:homogeneoustree-1} \( \Rightarrow \) \ref{lem:homogeneoustree-2}.
The fact that \( \tr \) is pruned follows from Lemma~\ref{lem:inhabitated} applied to the unique branch of \( \Delta(\tr) \). The second part is obvious if \( L \) has a minimum, so suppose that \( L \) has a strictly \( \leq_L \)-decreasing coinitial sequence \( (\ell_n)_{n \in \omega} \), and let \( b,b' \in [\tr] \). 
We recursively construct automorphisms \( f_n \in \Aut(\tr)  \) such that \( f_n(b(\ell_n)) = b'(\ell_n) \) and \( f_n \restriction \tr|_{\ell_{n-1}} = f_{n-1} \restriction \tr|_{\ell_{n-1}} \) if \( n > 0 \).
Let \( f_0 \in \Aut(\tr) \) be such that \( f(b(\ell_0)) = b'(\ell_0) \). Suppose now that \( f_n \in \Aut(\tr) \) as above has been defined. Let \( f'_{n+1} \in \Aut(\tr) \) be such that \( f'_{n+1}(b(\ell_{n+1})) = b'(\ell_{n+1}) \), so that also \( f'_{n+1}(b(\ell_n)) = b'(\ell_n) \). Define \( f_{n+1} \in \Aut(\tr) \) by letting for every \( t \in T \)
\[  
f_{n+1}(t) = 
\begin{cases}
f'_{n+1}(t) & \text{if } t \leq_T b(\ell_{n}) \\
f_n(t) & \text{otherwise}.
\end{cases}
\]
It is easy to check that \( f_{n+1} \) is as required.
The sequence of automorphisms \( (f_n)_{n \in \omega} \) converges to \( f = \bigcup_{n \in \omega} f_n \restriction \tr|_{\ell_n} \in \Aut(\tr) \), which clearly takes \( b \) to \( b' \).

\ref{lem:homogeneoustree-2} \( \Rightarrow \) \ref{lem:homogeneoustree-1}.
Obvious.
\end{proof}

\begin{theorem} \label{thm:homogeneous}
For every topological group \( G \), the following are equivalent:
\begin{enumerate-(1)}
\item \label{thm:homogeneous-1}
\( G \cong \Iso(U) \) for some homogeneous Polish ultrametric space \( U \);
\item \label{thm:homogeneous-2}
\( G \cong \Aut(\tr) \) for some countable homogeneous pruned \( L \)-tree \( \tr \);
\item \label{thm:homogeneous-3} 
\( G \cong \Wr^{\LF}_{\delta \in \Delta} \Sym(N_\delta) \) for some countable linear skeleton \( \langle \Delta, N \rangle \);
\item \label{thm:homogeneous-4}
\( G \cong \Wr^{S}_{\delta \in \Delta} \Sym(N_\delta) \) for some countable linear skeleton \( \langle \Delta, N \rangle \) and some approximately homogeneous closed separable domain \( S \subseteq S^{\Max} \);
\item \label{thm:homogeneous-5}
\( G \cong \Wr^{\Sloc}_{\delta \in \Delta} \Sym(N_\delta) \) for some countable linear skeleton \( \langle \Delta, N \rangle \) and some countable locally homogeneous family of local domains \( \Sloc \subseteq \Sloc^{\Max} \);
\item \label{thm:homogeneous-6}
\( G \cong \Wr^{\Sloc,\bpi}_{\delta \in \Delta} \Sym(N_\delta) \) for some countable linear skeleton \( \langle \Delta, N \rangle \) and some countable locally homogeneous system of projections \( \langle \Sloc, \bpi \rangle \) with \( \Sloc \subseteq \Sloc^{\Max} \).
\end{enumerate-(1)}
\end{theorem}

\begin{proof}
\ref{thm:homogeneous-1} \( \Rightarrow \) \ref{thm:homogeneous-2}.
Let \( D \) be the distance set of \( U \), and pick any \( L \) satisfying Condition~\ref{condition2} from Section~\ref{subsec:F}: then \( \tr = \func{F}(U) \) is a countable pruned \( L \)-tree satisfing \( \Aut(\tr) \cong \Iso(U) \) (Theorem~\ref{thm:spacestotrees} and Corollary~\ref{cor:spacestotrees}), so we only need to check that \( \tr \) is homogeneous. 
This amounts to showing that \( t \sim t' \) for every \( t,t' \in T \) with \( \lev_T(t) = \lev_T(t') = \ell \).
By construction, there are \( x,y \in U \) such that \( t = (B_d(x,\ell),\ell)\) and \( t' = (B_d(y,\ell),\ell) \). Since \( U \) is homogeneous, there is \( \psi \in \Iso(U) \) such that \( \psi(x) = y \). 
Then by definition of \( \func{F} \) the automorphism \( \func{F}(\psi) \in \Aut(\tr) \) is such that \( \func{F}(\psi)(t) = (B_d(\psi(x),\ell),\ell) = t' \).

\ref{thm:homogeneous-2} \( \Rightarrow \) \ref{thm:homogeneous-3}.
Let \( \tr \) be as in~\ref{thm:homogeneous-2}: we claim that \( \Aut(\tr) \cong \Wr^{\LF}_{\delta \in \Delta(\tr)} \Sym(N^\tr_\delta) \), which is enough because the countable skeleton \( \langle \Delta(\tr), N_{\tr} \rangle \) is linear by homogeneity of \( \tr \). 
Let \( P \) be the canonical partial order associated to the wreath product \( \Wr^{\LF}_{\delta \in \Delta(\tr)} \Sym(N^\tr_\delta) \), that is, \( P = \Sloc = \bigcup_{\delta \in \Delta(T)} \Sloc^{\LF}_\delta \) and, for every \( z,z' \in P \), \( z \leq_P z' \iff z \supseteq z' \).
By Lemma~\ref{lemma:g}, it is enough to show that there is an order-isomorphism \( g \colon T \to P \) such that \( g(t) \in \Sloc^{\LF}_{[t]} \), for every \( t \in T \). 
 
Let \( \ell_0\) be the minimum of \( L \) if it exists, or the first element of a strictly \( L \)-decreasing sequence \( (\ell_n)_{n \in \omega} \) coinitial in \( L \) if \( L \) has no minimum. 
Let \( \delta_0 \) be the unique member of \( \Lev_{\ell_0}(\Delta(\tr))  \).
Apply Lemma~\ref{lem:labelsfinitesupport} with \( C = \emptyset\) and \( \bar \delta = \delta_0 \).
The resulting map 
\[ 
g_0 \colon \tr|_{\ell_0} \to P|_{\delta_0}, \qquad t \mapsto z_t
\]
is such that 
\begin{equation} \label{eq:homogeneousorderisomorphism} 
t \leq_T t' \iff z_t \supseteq z_{t'} \qquad \text{and} \qquad  z_t \in \Sloc^{\LF}_{[t]} 
\end{equation}
for every \( t,t' \in T \) with \( \lev_T(t), \lev_T(t') \geq \ell_0 \), where \( P|_{\delta_0} = \bigcup_{\delta \geq_{\Delta(T)} \delta_0 } \Sloc^{\LF}_{\delta} \).
In particular, if \( \ell_0 = \min L \), then \( g = g_0 \) is already as desired.

If \( L \) has no minimum, instead, we repeat the process. Given any \( n > 0 \), let \( \delta_n \) be the unique element of \( \Lev_{\ell_n}(\Delta(\tr))\) and \( P|_{\delta_n} = \bigcup_{\delta \geq_{\Delta(T)} \delta_n} \Sloc^{\LF}_{\delta} \). 
Suppose that we already defined the function \( g_{n-1} \colon \tr|_{\ell_{n-1}} \to P|_{\delta_{n-1}} \) so that~\eqref{eq:homogeneousorderisomorphism} is satisfied with \( z_t = g_{n-1}(t) \): we want to extend \( g_{n-1} \) to \( g_n \colon \tr|_{\ell_n} \to P|_{\delta_n} \) by defining \( t \mapsto z_t \) on the set \( T_n = T|_{\ell_n} \setminus T|_{\ell_{n-1}} \)
so that~\eqref{eq:homogeneousorderisomorphism} is still satisfied for all \( t,t' \in \tr|_{\ell_n} \).
To this aim, for any \( \bar t \in \Lev_{\ell_{n-1}}(\tr) \) apply Lemma~\ref{lem:labelsfinitesupport} with \( C = \{ t \in T \mid \bar t \leq_T t \} \), \( \bar \delta = \delta_n \), and \( \bar z = z_{\bar t } = g_{n-1}(\bar t) \), and denote by \( g_n^{\bar t} \) the resulting function \( t \mapsto z_t \) on \( \tr^{\bar \delta}_C \). 
Then \( g_n = g_{n-1} \cup \bigcup \{ g_n^{\bar t} \mid \bar t \in \Lev_{\ell_{n-1}}(\tr) \}  \) is as required.

Since \( \Delta(\tr) \) is linear and \( (\ell_n)_{n \in \omega} \) is coinitial in \( L \), we get \( P = \bigcup_{n \in \omega} P|_{\delta_n}\), and thus \( g = \bigcup_{n \in \omega} g_n \) is as required.

\ref{thm:homogeneous-3} \( \Rightarrow \) \ref{thm:homogeneous-1}.
Without loss of generality, we may assume that \( \Delta \subseteq \RR^+ \), and that \( \inf \Delta = 0 \) if \( \Delta \) has no minimum.
Set \( U = S^{\LF} \subseteq \prod_{\delta \in \Delta} N_\delta \), and for distinct \( x,y \in U \) let \( d(x,y) = \max \{ \delta \in \Delta \mid x(\delta) \neq y(\delta) \} \), which exists because \( x,y \in S^{\LF} \). 
It is easy to verify that \( d \) is an ultrametric, and that \( (U,d) \) is separable by countability of \( \Delta \): a countable dense set is given by \( S^{\Fin} \subseteq U \).
Moreover, \( d \) is also complete. This is obvious if \( \Delta \) has a minimum, because in this case \( (U,d) \) is uniformly discrete. 
If instead \( \Delta \) has no minimum, fix any sequence \( (\delta_k)_{k \in \omega} \) coinitial in \( \Delta \). 
If \( (x_n)_{n \in \omega} \) is \( d \)-Cauchy, then for every \( k \in \omega \) there is \( M_k \) such that \( x_n(\delta) = x_m(\delta) \) for all \( n,m \geq M_k \) and \( \delta \geq_\Delta \delta_k \): 
for \( \delta \in \Delta \), let \( x \in \prod_{\delta \in \Delta} N_\delta \) be defined by \( x(\delta) = x_{M_k}(\delta) \) for some (equivalently, any) \( k \) such that \( \delta_k \leq_{\Delta} \delta \).
Then \( \supp(x) \in \LF \) because for every \( \delta \in \Delta \) we have \( x|_\delta = x_n|_\delta \) for all sufficiently large \( n \in \omega\), 
and for the same reason \( (x_n)_{n \in \omega} \) converges to \( x \).

We next show that \( (U,d) \) is homogeneous. To this aim, it is enough to show that for each \( x \in U \) there is \( \psi \in \Iso(U) \) such that \( \psi(x) = x_0 \), where \( x_0 \in U \) is such that \( \supp(x_0) = \emptyset \). For each \( \delta \), let \( \psi_\delta \in \Sym(N_\delta) \) be the identity if \( \delta \notin \supp(x) \), or the permutation swapping \( 0 \) with \( x(\delta) \) and fixing the rest of \( N_\delta \) if \( \delta \in \supp(x) \). It is easy to verify that the map \( \psi \colon U \to U \) defined by setting \( \psi(y)(\delta) = \psi_\delta(y(\delta)) \) for all \( y \in U \) and \( \delta \in \Delta \) is as desired.

Finally, for every \( \psi \in \Sym(U) = \Sym(S^{\LF}) \) one can easily check that \( \psi \in \Iso(U) \) if and only if \( \psi \in \Wr^{\LF}_{\delta \in \Delta} \Sym(N_\delta) \). Indeed, condition~\ref{Wr1} from Definition~\ref{def:wreathproduct} precisely amounts to requiring that \( \psi \) preserve distances with respect to \( d \) because \( x|_\delta = y|_\delta \iff d(x,y) < \delta \) for every \( x,y \in U \) and \( \delta \in \Delta \), while~\ref{Wr2} follows automatically from~\ref{Wr1} by Remark~\ref{rmk:permutationisautomatic}.
This shows that \( \Wr^{\LF}_{\delta \in \Delta} \Sym(N_\delta) = \Iso( U ) \), and we are done. 




\ref{thm:homogeneous-3} \( \Rightarrow \) \ref{thm:homogeneous-4}, \ref{thm:homogeneous-4} \( \Rightarrow \) \ref{thm:homogeneous-5}, and \ref{thm:homogeneous-5} \( \Rightarrow \) \ref{thm:homogeneous-6} are obvious.

\ref{thm:homogeneous-6} \( \Rightarrow \) \ref{thm:homogeneous-2}.
Let \( P = P^{\Sloc,\bpi}\) be the canonical partial order associated to \( \Wr^{\Sloc,\bpi}_{\delta \in \Delta} \Sym(N_\delta) \). Set \( L = (\Delta,\leq_\Delta) \).
Then the skeleton \( \langle \Delta,N \rangle \) is trivially \( L \)-treeable, hence by Lemma~\ref{lem:nominandmax} we can assume that \( L \) has no maximum.
Consider the function \( \lev_P \colon P \to \Delta \) defined by \( \lev_P(z) = \delta \) for every \( z \in \Sloc_\delta \). Then \( \tr = (P,\leq_P, \lev_P) \) is a countable pruned \( L \)-tree. 
(Here we use that \( \Sloc \subseteq \Sloc^{\Max} \) and that \( \Delta \) has no maximum.)
Since \( \leq_\Delta \) is linear, \( \Aut(\tr) = \AutL(P) \cong \Wr^{\Sloc,\bpi}_{\delta \in \Delta} \Sym(N_\delta) \). Together with the fact that \( \langle \Sloc,\bpi \rangle \) is locally homogeneous, this entails that \( \tr \) is homogeneous, as desired.
\end{proof}

\begin{remark} \label{rmk:homogeneous}
Since in~\ref{thm:homogeneous-3} of Theorem~\ref{thm:homogeneous} the skeleton \( \langle \Delta,N \rangle \) is linear, we also have \( \Wr^{\LF}_{\delta \in \Delta} \Sym(N_\delta) = \Wr_{\delta\in\Delta }^{ \UM } \Sym(N_{\delta }) \). Moreover, since \( \Delta \) is linear and has coinitiality at most \( \aleph_0 \), every family of local domains as in~\ref{thm:homogeneous-5} is automatically full.
\end{remark}


An interesting subclass of the ultrametric spaces studied above consists of the homogeneous Polish ultrametric spaces \( U \) which are discrete; by homogeneity, the latter is equivalent to the fact that \( U \) is uniformly discrete. 

\begin{theorem} \label{thm:homogeneousdiscrete}
For every topological group \( G \), the following are equivalent:
\begin{enumerate-(1)}
\item \label{thm:homogeneousdiscrete-1}
\( G \cong \Iso(U) \) for some (uniformly) discrete homogeneous Polish ultrametric space \( U \);
\item \label{thm:homogeneousdiscrete-2}
\( G \cong \Aut(\tr) \) for some countable homogeneous pruned \( L \)-tree \( \tr \), where \( L \) is a linear order with minimum; 
\item \label{thm:homogeneousdiscrete-3}
\( G \cong \Wr^{\LF}_{\delta \in \Delta} \Sym(N_\delta) \) for some countable linear skeleton \( \langle \Delta, N \rangle \) such that \( \Delta \) has a minimum; 
\item \label{thm:homogeneousdiscrete-4}
\( G \cong \Wr^{S}_{\delta \in \Delta} \Sym(N_\delta) \) for some countable linear skeleton \( \langle \Delta, N \rangle \) such that \( \Delta \) has a minimum and some approximately homogeneous closed separable domain \( S \subseteq S^{\Max} \);
\item \label{thm:homogeneousdiscrete-5}
\( G \cong \Wr^{\Sloc}_{\delta \in \Delta} \Sym(N_\delta) \) for some countable linear skeleton \( \langle \Delta, N \rangle \) such that \( \Delta \) has a minimum and some countable locally homogeneous family of local domains \( \Sloc \subseteq \Sloc^{\Max} \);
\item \label{thm:homogeneousdiscrete-6}
\( G \cong \Wr^{\Sloc,\bpi}_{\delta \in \Delta} \Sym(N_\delta) \) for some countable linear skeleton \( \langle \Delta, N \rangle \) such that \( \Delta \) has a minimum and some countable locally homogeneous system of projections \( \langle \Sloc, \bpi \rangle \) with \( \Sloc \subseteq \Sloc^{\Max} \);
\item \label{thm:homogeneousdiscrete-7}
\( G \cong \Wr^{\Fin}_{\delta \in \Delta} \Sym(N_\delta) \) for some countable linear skeleton \( \langle \Delta, N \rangle \), where \( \Wr^{\Fin}_{\delta \in \Delta} \Sym(N_\delta) \) is equipped with the topology \( \tau^* \).
\end{enumerate-(1)}
\end{theorem}

\begin{proof}
The equivalence of items from~\ref{thm:homogeneousdiscrete-1} to~\ref{thm:homogeneousdiscrete-6} can be proved as in Theorem~\ref{thm:homogeneous}, taking into account the following observations:
\begin{itemizenew}
\item 
In the implication~\ref{thm:homogeneousdiscrete-1} \( \Rightarrow \) \ref{thm:homogeneousdiscrete-2}, we can choose \( L \) with a minimum because the distance set \( D \) of \( U \) is bounded away from \( 0 \).
\item 
The proof of~\ref{thm:homogeneousdiscrete-2} \( \Rightarrow \) \ref{thm:homogeneousdiscrete-3} is actually simpler because \( L \) has a minimum.
\item 
Since \( \leq_\Delta \) has a minimum, in the implication~\ref{thm:homogeneousdiscrete-3} \( \Rightarrow \) \ref{thm:homogeneousdiscrete-1} we can assume that \( \Delta \subseteq \RR^+ \) is bounded away from \( 0 \), so that the resulting Polish ultrametric space \( (S^{\LF},d) \) is uniformly discrete.
\item
When employing Lemma~\ref{lem:nominandmax} in the proof of~\ref{thm:homogeneousdiscrete-6} \( \Rightarrow \) \ref{thm:homogeneousdiscrete-2}, we can use only the first part of its proof to ensure that \( L = \Delta \) has no maximum but maintains its minimum.
\end{itemizenew}

The implication~\ref{thm:homogeneousdiscrete-3} \( \Rightarrow \) \ref{thm:homogeneousdiscrete-7} follows from the fact that when \( \Delta \) is linear and \( \leq_\Delta \) has a minimum, then \( \Wr^{\LF}_{\delta \in \Delta} \Sym(N_\delta) = \Wr^{\Fin}_{\delta \in \Delta} \Sym(N_\delta) \) and, moreover, the topology \( \tau^*\) on \( \Wr^{\Fin}_{\delta \in \Delta} \Sym(N_\delta) \) coincides with the usual topology \( \tau \) of pointwise convergence with respect to \( \tau_\Delta \).
For the reverse implication~\ref{thm:homogeneousdiscrete-7} \( \Rightarrow \) \ref{thm:homogeneousdiscrete-3}, consider an arbitrary countable linear skeleton \( \langle \Delta,N \rangle \). Extend it to \( \langle \Delta',N' \rangle \) by adding a bottom element \( \bar \delta \) to \( \Delta \) and setting \( N'_{\bar \delta} = 1 \). Then \( \Wr^{\Fin}_{\delta \in \Delta} \Sym(N_\delta) \), when equipped with the topology \( \tau^* \), is isomorphic to \( \Wr^{\LF}_{\delta \in \Delta'} \Sym(N_\delta) \), hence we are done. 
\end{proof}

\begin{remark} \label{rmk:homogeneousdiscrete}
As in Remark~\ref{rmk:homogeneous}, in part~\ref{thm:homogeneousdiscrete-3} of Theorem~\ref{thm:homogeneousdiscrete} we again have that \( \Wr^{\LF}_{\delta \in \Delta} \Sym(N_\delta) = \Wr_{\delta\in\Delta }^{ \UM } \Sym(N_{\delta }) \), but since \( \leq_\Delta \) has minimum, this time we also have \( \Wr^{\LF}_{\delta \in \Delta} \Sym(N_\delta) = \Wr^{\Fin}_{\delta \in \Delta} \Sym(N_\delta) \). As before, every family of local domains as in~\ref{thm:homogeneousdiscrete-5} is automatically full.
\end{remark}

Interestingly, the equivalence between~\ref{thm:homogeneousdiscrete-1} and~\ref{thm:homogeneousdiscrete-7} in Theorem~\ref{thm:homogeneousdiscrete} shows that the class of all isometry groups of discrete homogeneous Polish ultrametric spaces coincides with the generalized wreath products defined by Hall in~\cite{hall1962} once we require that the underlying linear order is countable and the permutation groups are full.

Another class of interest in the context of homogeneous spaces is given by Polish ultrametric Urysohn spaces.
See~\cite{GaoShao2011} for more details and information.
\details{Since the distance set of a Polish ultrametric space is always countable, there is no universal Polish ultrametric space. However, if we fix a countable distance set \( D \subseteq \RR^+ \), then there is a (unique, up to isometry) Polish ultrametric space which is ultrahomogeneous,\todo{Vogliamo mettere una definizione esplicita in una footnote? (\`E standard e non lo usiamo mai.)} has distance set \( D \), and is universal for Polish ultrametric spaces with distance set contained in \( D \). Such a space is called Polish \( D \)-ultrametric Urysohn space, but since the set \( D \) can be canonically retrieved from the space as its distance set, we can safely drop the reference to it ans speak of \markdef{Polish ultrametric Urysohn spaces}. We refer the reader to~\cite{GaoShao2011} for more details and information.}
In order to characterize their isometry groups 
using generalized wreath products, we need to introduce some more terminology. A countable linear skeleton \( \langle \Delta,N \rangle \) is \markdef{maximal} if \( N \) always attains the maximal value, i.e.\ \( N_\delta = \omega \) for every \( \delta \in \Delta \); it is \markdef{quasi-maximal} if each \( N_\delta \) is either maximal or trivial, that is, \( N_\delta \in \{ 1, \omega \} \) for every \( \delta \in \Delta \). 
We provide two characterizations in terms of generalized wreath products: to our taste, the most appealing is obtained by considering only maximal countable linear skeletons, and then taking Hall's generalized wreath products together with generalized wreath products induced by locally finite supports over such skeletons. Along the way, we also consider a related class of spaces. A metric space is \markdef{\( r \)-discrete}, where \( r \in \RR^+ \), if all its points are at distance \( r \) from each other. A Polish ultrametric space \( U \) is \markdef{wide} if for every \( r \) in its distance set, \( U \) contains a copy of the countably infinite \( r \)-discrete space.

\begin{theorem} \label{thm:Urysohn}
For every topological group \( G \), the following are equivalent:
\begin{enumerate-(1)}
\item \label{thm:Urysohn-1}
\( G \cong \Iso(U) \) for some Polish ultrametric Urysohn space \( U \);
\item \label{thm:Urysohn-2}
\( G \cong \Iso(U) \) for some homogeneous wide Polish ultrametric space \( U \);
\item \label{thm:Urysohn-3}
\( G \cong \Wr^{\LF}_{\delta \in \Delta} \Sym(N_\delta) \) for some quasi-maximal countable linear skeleton \( \langle \Delta,N \rangle \);
\item \label{thm:Urysohn-4}
\( G \cong \Wr^{\LF}_{\delta \in \Delta} \Sym(N_\delta) \) or \( G \cong \Wr^{\Fin}_{\delta \in \Delta} \Sym(N_\delta) \), where the latter is equipped with the topology \( \tau^* \), for some maximal countable linear skeleton \( \langle \Delta,N \rangle \)
\end{enumerate-(1)}
\end{theorem}

\begin{proof}
Since Polish ultrametric Urysohn spaces are, by definition, ultrahomogeneous and universal for Polish ultrametric spaces with the same distance set (or a subset thereof), they are both homogeneous and wide: this shows that~\ref{thm:Urysohn-1} \( \Rightarrow \) \ref{thm:Urysohn-2}. 

For \ref{thm:Urysohn-2} \( \Rightarrow \) \ref{thm:Urysohn-3}, we employ the same argument as in the proof of the implications~\ref{thm:homogeneous-1} \( \Rightarrow \) \ref{thm:homogeneous-2} and \ref{thm:homogeneous-2} \( \Rightarrow \) \ref{thm:homogeneous-3} in Theorem~\ref{thm:homogeneous}. 
Given a homogeneous wide Polish ultrametric space \( U \), we consider its distance set \( D \), and then pick any linear order \( D \subseteq L \subseteq \RR^+ \) satisfying Condition~\ref{condition2} from Section~\ref{subsec:F}.
By construction, the countable pruned \( L \)-tree \( \tr = \func{F}(U) = (T,{\leq_T}, \lev_T) \), which is such that \( \Aut(\tr) \cong \Iso(U) \), is homogeneous and such that for every \( t \in T \):
\begin{itemizenew}
\item 
if \( \lev_T(t) \in D \), then \( |C_t| = \omega \) because \( U \) is wide;
\item 
if \( \lev_T(t) \in L \setminus D \), then \( |C_t| = 1 \) by definition of \( \tr = \func{F}(U) \): indeed, if \( \lev_T(t') = \lev_T(t) \) for some \( t' \neq t \), then \( \spl{t,t'} > \lev_T(t) \) because \( \lev_T(t) \notin D \).
\end{itemizenew}
This shows that the countable linear skeleton \( \langle \Delta(\tr), N_{\tr} \rangle \) is quasi-minimal, and since \( \Aut(\tr) \cong \Wr^{\LF}_{\delta \in \Delta(\tr)} \Sym(N^{\tr}_\delta) \) we are done.

We now prove~\ref{thm:Urysohn-3} \( \Rightarrow \) \ref{thm:Urysohn-4}. Let \( \langle \Delta,N \rangle \) be a quasi-maximal countable linear skeleton. We distinguish two cases. 
If there is \( \bar\delta \in \Delta \) such that \( N_{\delta} = 1 \) for every \( \delta \leq \bar\delta \), we let \( \Delta' = \{ \bar \delta \} \cup \{ \delta \in \Delta \mid N_\delta = \omega \} \); otherwise, we let \( \Delta' = \{ \delta \in \Delta \mid N_\delta = \omega \} \). It is easy to verify that if we let \( N' \) be the restriction of \( N \) to \( \Delta' \), then \( \Wr^{\LF}_{\delta \in \Delta} \Sym(N_\delta) \cong \Wr^{\LF}_{\delta \in \Delta'} \Sym(N'_\delta) \). 
If we are in the second case, then \( \langle \Delta',N' \rangle \) is already maximal, and we are done. 
If instead we are in the first case, then we consider the maximal countable linear skeleton \( \langle \Delta'',N'' \rangle \) obtained by dropping \( \bar \delta \) from \( \Delta ' \), and observe that \( \Wr^{\LF}_{\delta \in \Delta'} \Sym(N'_\delta) = \Wr^{\Fin}_{\delta \in \Delta'} \Sym(N'_\delta) \cong \Wr^{\Fin}_{\delta \in \Delta''} \Sym(N''_\delta)\) because \( \bar \delta = \min \Delta'\) and in the middle group \( \tau \) and \( \tau^* \) coincide.

Finally, in order to prove \ref{thm:Urysohn-4} \( \Rightarrow \) \ref{thm:Urysohn-1} fix a maximal countable linear skeleton \( \langle \Delta,N \rangle \). We first consider the case where \( G \cong \Wr^{\LF}_{\delta \in \Delta} \Sym(N_\delta) \). Without loss of generality we can assume that \( \Delta \subseteq \RR^+ \) and \( \inf \Delta = 0 \) if \( \Delta \) has no minimum. Let \( U \) be the (unique, up to isometry) Polish ultrametric Urysohn space with distance set \( D = \Delta \), which in particular is homogeneous and wide. If we follow the constructions in the proof of \ref{thm:Urysohn-2}~\( \Rightarrow \)~\ref{thm:Urysohn-3} and \ref{thm:Urysohn-3}~\( \Rightarrow \)~\ref{thm:Urysohn-4} one after the other, 
we easily obtain that \( \Iso(U) \cong \Wr^{\LF}_{\delta \in \Delta} \Sym(N_\delta) \). 
Suppose now that \( G \cong \Wr^{\Fin}_{\delta \in \Delta} \Sym(N_\delta) \), where the latter is equipped with \( \tau^* \).
We can again assume that \( \Delta \subseteq \RR^+\), but this time we also ensure that \( \inf \Delta > 0 \). Let again \( U \) be the Polish ultrametric Urysohn space with distance set \( D = \Delta \). If we argue as before but adding some \( 0 < r < \inf D \) to the linear order \( L \) in the proof of \ref{thm:Urysohn-2}~\( \Rightarrow \)~\ref{thm:Urysohn-3}, then when passing through the construction in the proof of~\ref{thm:Urysohn-3} \( \Rightarrow \) \ref{thm:Urysohn-4} we will be in the first case (i.e.\ the one where \( \bar \delta \) exists), and hence \( \Iso(U) \cong \Wr^{\Fin}_{\delta \in \Delta} \Sym(N_\delta) \) by construction.
\end{proof}

\begin{remark} \label{rmk:Urysohn}
For the sake of brevity, in Theorem~\ref{thm:Urysohn} we skipped the analogues of items~\ref{thm:homogeneous-2}, \ref{thm:homogeneous-4}, \ref{thm:homogeneous-5}, and~\ref{thm:homogeneous-6} from both Theorem~\ref{thm:homogeneous} and Theorem~\ref{thm:homogeneousdiscrete}; the interested reader can guess them by inspecting the proof above. Also, there are several variants of item~\ref{thm:Urysohn-3} in Theorem~\ref{thm:Urysohn} that can be obtained by further restricting the condition of quasi-minimality on the skeleton \( \langle \Delta,N \rangle\). For example, one could ask that \( N_\delta = \omega \) for every \( \delta \in \Delta \) which is not the minimum of \( \Delta \), and \( N_{\min \Delta} = 1 \) (alternatively, but still equivalently: \( N_{\min \Delta} \in \{ 1,\omega \}\)) if such a minimum exists. 
\end{remark}

We conclude the section by noticing that the characterization provided in Theorem~\ref{thm:Urysohn} and the constructions used in its proof easily allow to reprove various known results concerning \( \Iso(U) \) for \( U \) a Polish ultrametric Urysohn space, including Theorems 5.12 and 7.3 from~\cite{GaoShao2011}.

\subsection{Exact spaces}\label{sec:exactdistances}

If a Polish ultrametric space \( U \) is not homogeneous, it makes sense to consider its homogeneous components. 
These are the equivalence classes induced by the natural action of \( \Iso(U) \) on \( U \): the \markdef{homogeneous component} \( [x] \) of \( x \in U \) is the set \( \{ y \in U \mid \varphi(x) = y \text{ for some } \varphi \in \Iso(U) \} \).

\begin{defin}
Two homogeneous components \( [x] \) and \( [y] \) of a Polish ultrametric space have \markdef{exact distance} if their distance is realized by points, i.e., if the infimum in \( d([x],[y]) = \inf \{ d(x',y') \mid x' \in [x] , y' \in [y] \} \) is attained.
A Polish ultrametric space is \markdef{exact} if every two homogeneous components of the space have exact distance.
\end{defin}

Not all Polish ultrametric spaces are exact.
\details{
\begin{ex} \label{ex:notexact}
Recall the tree \( T \) from Example~\ref{ex:counterexampleDelta} and equip
\begin{multline*}
U = \Lev_T(0) 
= \{ t_\alpha^0 \mid \alpha \in \Can \text{ is eventually constant}\} \\ \cup \{ t_\alpha^\diamond  \mid \alpha \in \Can \text{ is eventually } 1\}
\end{multline*} 
with the metric given by
\begin{equation} \label{eq:exactdistancecounterex}
d(t^0_\alpha,t^0_{\alpha'}) = d(t^0_\alpha,t^\diamond_{\alpha'}) = d(t^\diamond_\alpha,t^\diamond_{\alpha'}) = 1+2^{-n}
\end{equation}
if \( \alpha \neq \alpha'\) and \( n \in \omega \) is smallest such that \( \alpha(n) \neq \alpha'(n) \), while
\[
d(t^0_\alpha,t^\diamond_\alpha) = 1 
\]
if \( \alpha \) is eventually constant with value \( 1 \). 

Then $(U,d)$ is a uniformly discrete Polish ultrametric space and, as in Example~\ref{ex:counterexampleDelta}, it has two homogeneous components, namely, \( [t^0_{0^{\NN}}] \) and \( [t^0_{1^{\NN}}] = [t^\diamond_{1^{\NN}}] \), which have distance \(1\).
However, points coming from different components have distance larger than \( 1 \) by~\eqref{eq:exactdistancecounterex}.
\end{ex}

\begin{lemma} \label{lem:exactdistances}
For every metric space \( U \)
the following are equivalent:
\begin{enumerate-(1)}
\item \label{lem:exactdistances-1}
\( U \) is exact;
\item \label{lem:exactdistances-2}
for every \( x, x_0 \in U \) there exists \( \psi \in \Iso(U) \) such that 
\[
d(\psi(x), x_0) = \inf  \{ d(\varphi(x),  x_0) \mid \varphi \in \Iso(U) \}.
\]
\end{enumerate-(1)}
\end{lemma}

\begin{proof}
Notice that \( \inf  \{ d(\varphi(x),  x_0) \mid \varphi \in \Iso(U) \} = d([x],x_0) \).

\ref{lem:exactdistances-1} \( \Rightarrow \) \ref{lem:exactdistances-2}.
Let \( r = \inf  \{ d(\varphi(x),  x_0) \mid \varphi \in \Iso(U) \} = d([x],x_0) \). Then it is enough to show that \( d(\psi(x),x_0) \leq r \) for some \( \psi \in \Iso(U) \). 
By \ref{lem:exactdistances-1}, let \( \bar x \in [x] \) and \( \bar x_0 \in [x_0] \) be such that \( d(\bar x,\bar x_0) = d([x],[x_0]) \). Let \( \varphi \in \Iso(U) \) be such that \( \varphi(\bar x_0)  = x_0 \). 
Since \( \varphi(\bar x) \in [x] \), there is \( \psi \in \Iso(U)\) such that \( \psi(x) = \varphi(\bar x) \).
Then \( d(\psi(x),x_0) = d(\varphi(\bar x),\varphi(\bar x_0)) = d(\bar x,\bar x_0) = d([x],[x_0]) \leq d([x],x_0) = r \).

\ref{lem:exactdistances-2} \( \Rightarrow \) \ref{lem:exactdistances-1}.
Given \( x,y \in U \), fix any \( x_0 \in [y] \).
By~\ref{lem:exactdistances-2}, there is \( \psi \in \Iso(U) \) such that \( d(\psi(x),x_0) = d([x],x_0) \). We claim that \( d([x],x_0) = d([x],[x_0]) \), which implies that \( \psi(x) \) and \( x_0 \) realize the distance between \( [x] \) and \( [y] \).
Clearly, \( d([x],[x_0]) \leq d([x],x_0) \). Suppose towards a contradiction that such inequality is strict, and let \( x' \in [x] \) and \( x'_0 \in [x_0] \) be such that \( d(x',x'_0) < d([x],x_0) \). Let \( \varphi \in \Iso(U) \) be such that \( \varphi(x'_0) = x_0 \). Then \( d(\varphi(x'),x_0) < d([x],x_0) \), which is a contradiction because \( \varphi(x') \in [x] \).
\end{proof}

Condition~\ref{lem:exactdistances-2} can be reformulated as: 

Indeed, it is not difficult to show that for every metric space \( U \), \( U \) is exact if and only if the following condition holds:
\begin{quote}
For every \( r \in \RR^+ \) and every \( x ,x_0 \in U \), if for all \( r'> r\) there is \( \varphi \in \Iso(U) \) such that \( d(\varphi(x),  x_0) \leq r' \), then there is \( \psi \in \Iso(U) \) such that \( d(\psi(x), x_0) \leq r \).
\end{quote}
}
As Lemma \ref{lem:exactdistancesprelim} shows, the counterpart of exactness for an \( L \)-tree \( \tr = (T, {\leq_T} , \lev_T) \) is the following:

\begin{enumerate}[label={(\( \star\))}, leftmargin=2pc]
\item \label{propertystar}
For every proper upward closed chain \( C \subseteq T \) and every \( t \in T  \) such that \( [t] \leq_{\Delta(T)} [C] \) there is \( \bar t \in [t] \) such that \( \bar t \leq_T C \). 
\end{enumerate}

\begin{remark} \label{rmk:propertystar}
Property~\ref{propertystar} needs to be checked only when \( C \) has no minimum; indeed if \( C \) has a minimum \( t' \) then, since \( [t] \leq_{\Delta(T)} [t'] \) by hypothesis, there is \( f \in \Aut(T) \) such that \( \bar t= f(t) \leq_T t' \leq_T C \).
\end{remark}

\begin{lemma} \label{lem:exactdistancesprelim}
Let \( \func{U}\), \( \func{F} \), and \( \func{G} \) be the functors from Section~\ref{sec:functors}. Let \( U \in \cat{U}_D \) be an exact Polish ultrametric space, and let \( \tr \in \cat{T}^{\mathrm{pr}}_L \) be a countable pruned \( L \)-tree satisfying property~\ref{propertystar}. Then:
\begin{enumerate-(i)}
\item \label{lem:exactdistancesprelim-i}
if \( \inf D > 0 \), so that \( \func{U} \) can be defined, then \( \func{U}(U) \) is exact;
\item \label{lem:exactdistancesprelim-ii}
\( \func{F}(U) \) satisfies property~\ref{propertystar};
\item \label{lem:exactdistancesprelim-iii}
if \( L \) has no minimum, then \( \func{G}(\tr) \) is exact.
\end{enumerate-(i)}
\end{lemma}

\begin{proof}
\ref{lem:exactdistancesprelim-i}
Suppose that \( \inf D > 0 \), where \( D \) is the distance set of \( U \). 
Recall from Section~\ref{subsec:U} that \( \func{U}(U) = (U \times \bar U, \hat d ) \), where \( (\bar U, \bar d) \) is a strongly rigid perfect compact Polish ultrametric space whose distance set is contained in \( (0;\inf D) \). 
Recall also that, by the proof of Theorem~\ref{thm:toperfectlocallycompact}, all isometries in \( \Iso(\func{U}(U)) \) are of the form \( (x,y) \mapsto (\psi(x),y) \), for some \( \psi \in \Iso(U) \).
Therefore, the homogeneous component of an arbitrary \( (x,y) \in \func{U}(U) \) is \( [(x,y)] = [x] \times \{ y \} \), where \( [x] \) is the homogeneous component of \( x \) in \( U \). 
Suppose that \( (x,y),(x',y') \in \func{U}(U) \) are such that \( [(x,y)] \neq [(x',y')] \). If \( [x] = [x'] \), necessarily \( y \neq y' \), and thus \( \hat d([(x,y)],[(x',y')]) = \hat d ((x,y),(x,y')) \). If instead \( [x] \neq [x'] \), then \( \hat{d}([(x,y)],[(x',y')]) = d([x],[x'])  \) because \( [x] \cap [x'] = \emptyset \).
By the hypothesis on \( U \), there are \( z \in [x] \) and \( z' \in [x'] \) such that \( d(z,z') = d([x],[x']) \). Then \( z \neq z'\) because \( d([x],[x']) \geq \inf D > 0 \), and hence \( \hat d ((z,y),(z',y')) = d(z,z')  = \hat{d}([(x,y)],[(x',y')]) \). Since \( (z,y) \in [(x,y)] \) and \( (z',y') \in [(x',y')] \), we are done.

\ref{lem:exactdistancesprelim-ii}
Let \( \tr = (T,{\leq_T},\lev_T) \) be the \( L \)-tree \( \func{F}(U) \), for the appropriate \( D \subseteq L \subseteq \RR^+ \), and recall 
the definition of \( \func{F}(\psi) \) from Section~\ref{subsec:F}.
Let \( \hat t = (\hat B,\hat \ell) \) witness that the upward closed chain \( C \subseteq T \) is proper, and let \( t = (B,\ell) \) be such that \( [t] \leq_{\Delta(T)} [C] \): we want to show that there is \( \bar t \in [t] \) such that \( \bar t \leq_T C \). 
Let \( \hat x \in \hat B \) and \( x \in B \).
By the choice of \( \hat t \) and \( \hat x \), every \( t' \in C \) is of the form \( t' = (B_d(\hat x,\ell'),\ell') \), where \( \ell' = \lev_T(t') \).
Let \( L' = \{ \lev_T(t') \mid t' \in C \} \), and notice that \( \hat \ell < \ell' \) and \( \ell \leq \ell' \) for all \( \ell' \in L' \).

By the equivalence~\eqref{eq:simforultra}, the fact that \( [t] \leq_{\Delta(T)} [C] \) translates to the fact that for every \( \ell' \in L' \) there is \( \varphi \in \Iso(U) \) such that \( d(\varphi(x), \hat x) < \ell' \).
We claim that there is \( \psi \in \Iso(U) \) such that \( d(\psi(x), \hat x) < \ell' \) for all \( \ell' \in L' \). This is obvious if \( L' \) has a minimum.
If this is not the case, then we let \( r = \inf L' \), so that \( r < \ell' \) for every \( \ell' \in L' \), and we notice that by the observation above, \( d([x],[\hat x]) \leq r \). Since by hypothesis \( [x] \) and \( [\hat x] \) have exact distance, there are \( y \in [x] \) and \( \hat y \in [\hat x] \) such that \( d(y, \hat y) \leq r \). Let \( \psi_0,\psi_1 \in \Iso(U) \) be such that \( \psi_0( x) = y \) and \( \psi_1(\hat y) = \hat x \), and let \( \psi = \psi_1 \circ \psi_0 \). Then 
\[ 
d(\psi(x),\hat x) = d(\psi_0(x),\psi_1^{-1}(\hat x)) = d(y,\hat y) \leq r 
\]
and we are done.

Finally, we observe that for \( \psi \in \Iso(U) \) as above, \( \bar t = \func{F}(\psi)(t) \) is such that \( \bar t \leq_T C \). Indeed, for every \( t' \in C \) with \( \lev_T(t') = \ell' \in L' \) we have
\[ 
\bar t = (B_d(\psi(x),\ell),\ell) \leq_T (B_d(\psi(x),\ell'),\ell') = (B_d(\hat x,\ell'),\ell') = t'.
\]

\ref{lem:exactdistancesprelim-iii}
Let \( U = \func{G}(\tr) \), and recall from Section~\ref{subsec:G} that
\( U \) is of the form \( (T,d_T) \) and \( \Iso(U) = \Aut(T) \) (Corollary~\ref{cor:treestospaces}).
Therefore, 
the homogeneous components of \( U \) coincide with the classes \( [t] \in \Delta(T) \). 
Let \( [t],[t'] \in \Delta(T) \) be distinct.
We borrow from Section~\ref{subsec:G} the notation \( \ell^+ \) and \( \ell^- \), for \( \ell \in L \). 
If \( [t] <_{\Delta(T)} [t'] \), then \( d_T([t],[t']) \geq \ell^- \), where \( \ell = \lev_{\Delta(T)}([t']) \). 
On the other hand, \( t|_{\ell} \in [t'] \) because \( [t] \leq_{\Delta(T)} [t'] \), and by construction \( d_T(t,t|_{\ell}) = \ell^- \). 
This implies that  \( d_T([t],[t']) = \ell^-\), and we already know that such distance is realized by the points \( t \in [t] \) and \( t|_{\ell} \in [t'] \). 
The case where \( [t'] <_{\Delta(T)} [t] \) is symmetric, so let us assume that \( [t] \) and \( [t'] \) are \( \leq_{\Delta(T)} \)-incomparable.
Let \( L' = \{ \ell \in L \mid t'|_\ell \sim t|_\ell \} \) and \( C = \{ t'|_\ell \mid \ell \in L' \} \). Then \( C \subseteq T \) is an upward closed chain, and by case assumption \( t' \) witnesses that \( C \) is proper. Moreover, \( [t] \leq_{\Delta(T)} [C] \) by choice of \( L' \). By property~\ref{propertystar}, there is \( \varphi \in \Aut(T) = \Iso(U) \) such that \( \bar t = \varphi(t) \leq_T C \).
By case assumption, \( t' \) and \( \bar t \) are \( \leq_T \)-incomparable; let \( \ell = \spl{t',\bar t} \), so that \( d(t',\bar t) = \ell^+ \). By choice of \( L' \) and \( \bar t = \varphi(t) \leq_T C \), we have that \( \ell = \max (L\setminus L')  \), and there are no \( y \in [t] \) and \( y' \in [t'] \) such that \( y|_\ell = y'|_\ell \). Therefore \( d([t],[t']) = \ell^+ \), and such distance is realized by the points \( \bar t \) and \( t' \).
\end{proof}




\begin{theorem} \label{thm:exactdistancesnew} 
For every topological group \( G \), the following are equivalent:
\begin{enumerate-(1)}
\item \label{thm:exactdistancesnew-1}
\( G \cong \Iso(U) \) for some exact Polish ultrametric space \( U \);
\item \label{thm:exactdistancesnew-1'}
\( G \cong \Iso(U) \) for some exact perfect locally compact Polish ultrametric space \( U \);
\item \label{thm:exactdistancesnew-1''}
\( G \cong \Iso(U) \) for some exact uniformly discrete Polish ultrametric space \( U \);
\item \label{thm:exactdistancesnew-2}
\( G \cong \Aut(\tr) \) for some countable pruned \( L \)-tree \( \tr \) with property~\ref{propertystar} and some linear order \( L \);
\item \label{thm:exactdistancesnew-2'}
\( G \cong \Aut(\tr) \) for some countable pruned special \( L \)-tree \( \tr \) with property~\ref{propertystar} and some linear order \( L \) with a minimum;
\item \label{thm:exactdistancesnew-2''}
\( G \cong \Aut(\tr) \) for some countable pruned special \( \QQ \)-tree \( \tr \) with property~\ref{propertystar}; 
\item \label{thm:exactdistancesnew-3}
\( G \cong \Wr^{\LF}_{\delta \in \Delta} \Sym(N_\delta) \) for some countable treeable skeleton \( \langle \Delta,N \rangle \);
\item \label{thm:exactdistancesnew-3'}
\( G \cong \Wr^{S}_{\delta \in \Delta} \Sym(N_\delta) \) for some countable treeable skeleton \( \langle \Delta,N \rangle \) and some approximately homogeneous closed separable domain \( S \subseteq \prod_{\delta \in \Delta} N_\delta \);
\item \label{thm:exactdistancesnew-3''}
\( G \cong \Wr^{\Sloc}_{\delta \in \Delta} \Sym(N_\delta) \) for some countable treeable skeleton \( \langle \Delta,N \rangle \) and some countable locally homogeneous full family of local domains \( \Sloc \subseteq \Sloc^{\Max} \);
\item \label{thm:exactdistancesnew-3'''}
\( G \cong \Wr^{\Sloc,\bpi}_{\delta \in \Delta} \Sym(N_\delta) \) for some countable treeable skeleton \( \langle \Delta,N \rangle \) and some countable locally homogeneous full system of projections \( \langle \Sloc,\bpi \rangle\) such that \( \Sloc \subseteq \Sloc^{\Max} \).
\end{enumerate-(1)}
\end{theorem}

\begin{proof} 
The implication \ref{thm:exactdistancesnew-1} \( \Rightarrow \) \ref{thm:exactdistancesnew-2''}
follows from Corollary~\ref{cor:spacestotrees2} and Lemma~\ref{lem:exactdistancesprelim}\ref{lem:exactdistancesprelim-ii}, and \ref{thm:exactdistancesnew-2''} \( \Rightarrow \) \ref{thm:exactdistancesnew-2} is obvious. 

To prove \ref{thm:exactdistancesnew-2} \( \Rightarrow \) \ref{thm:exactdistancesnew-1''}, 
let \( \tr \) be a countable pruned special \( L \)-tree with property~\ref{propertystar}. 
Without loss of generality, we can assume that \( L \) has no minimum. 
Indeed, if this is not the case we consider the linear order \( L' = \omega^*+L\). 
Since \( \tr \) is countable, \( \Lev_{\min L} (\tr) \) is countable as well. 
Let \( \tr' = (T',{\leq_{T'}},\lev_{t'}) \) be the \( L' \)-tree obtained by adding to \( T \) a new node \( (t,n) \) for every \( t \in \Lev_{\min L}(\tr) \) and \( n \in \omega^* \), and then stipulating that \( \lev_{T'}(t,n) = n \in \omega^*\), \( (t,n) \leq_{T'} t' \) if and only if \( t\le_Tt' \), and \( (t,n) \leq_{T'} (t',m) \) if and only if \( t = t' \) and \( n \geq m \), while \( \tr' |_{\min L} = \tr \).
Then \( \tr' \) is countable and pruned, \( \Aut(\tr') \cong \Aut(\tr) \), and \( \tr' \) still satisfies~\ref{propertystar}: in fact, this is a consequence of  \( \tr'|_{\min L} = \tr \) if the upward closed chain \( C \subseteq T' \) is proper also in \( \tr \) (i.e.\ \( t <_{T'} C \) for some \( t \in T \)), while if this is not the case then \( C \) has a minimum and Remark~\ref{rmk:propertystar} applies.

Therefore, if \( D \subseteq \RR^+ \) is the range of an embedding of \( 2 \cdot L \) into \( \RR^+\) satisfying \( \inf D > 0 \), then using Theorem~\ref{thm:treestospaces}, Corollary~\ref{cor:treestospaces}, and Lemma~\ref{lem:exactdistancesprelim}\ref{lem:exactdistancesprelim-iii} we get the desired implication. 

To prove \ref{thm:exactdistancesnew-1''} \( \Rightarrow \) \ref{thm:exactdistancesnew-1'} we use instead Theorem~\ref{thm:toperfectlocallycompact} and Corollary~\ref{cor:toperfectlocallycompact} together with Lemma~\ref{lem:exactdistancesprelim}\ref{lem:exactdistancesprelim-i},  picking as \( D' \subseteq \RR^+ \) the set \( D \cup \{ 2^{-n} \mid n \in \omega \} \), where \( D \) is the distance set of the uniformly discrete Polish ultrametric space given as input (\( \inf D > 0 \) is ensured by uniform discreteness).
Since \ref{thm:exactdistancesnew-1'} \( \Rightarrow \) \ref{thm:exactdistancesnew-1} is obvious, we have shown that  \ref{thm:exactdistancesnew-1}--\ref{thm:exactdistancesnew-2} and \ref{thm:exactdistancesnew-2''} are equivalent to each other.

To prove \ref{thm:exactdistancesnew-1''} \( \Rightarrow \) \ref{thm:exactdistancesnew-2'} we use again
the functor \( \func{F} \) from Section~\ref{subsec:F} (and in particular Theorem~\ref{thm:spacestotrees}, Corollary~\ref{cor:spacestotrees}, and Lemma~\ref{lem:spacestotreesspecial}, together with 
Lemma~\ref{lem:exactdistancesprelim}\ref{lem:exactdistancesprelim-ii}), exploiting the fact that \( \inf D > 0 \) when \( D \) is the distance set of a uniformly discrete Polish ultrametric space, so that we can 
choose \( L \) as in the first item of Lemma~\ref{lem:spacestotreesspecial}. The implication \ref{thm:exactdistancesnew-2'} \( \Rightarrow\) \ref{thm:exactdistancesnew-2} is again trivial, thus also \ref{thm:exactdistancesnew-2'} is equivalent to all other conditions up to \ref{thm:exactdistancesnew-2''}.

The implications~\ref{thm:exactdistancesnew-3} \( \Rightarrow \) \ref{thm:exactdistancesnew-3'} and~\ref{thm:exactdistancesnew-3''} \( \Rightarrow \) \ref{thm:exactdistancesnew-3'''} are obvious, while \ref{thm:exactdistancesnew-3'} \( \Rightarrow \) \ref{thm:exactdistancesnew-3''} follows from Proposition~\ref{prop:closedvsfull}. 

We now prove \ref{thm:exactdistancesnew-3'''} \( \Rightarrow \) \ref{thm:exactdistancesnew-2}. 
Fix any projective wreath product \( \Wr^{\Sloc,\bpi}_{\delta \in \Delta} \Sym(N_\delta) \) as in~\ref{thm:exactdistancesnew-3'''}. 
By Lemma~\ref{lem:nominandmax}, we can assume that \( \langle \Delta, N \rangle \) is \( L \)-treeable for some linear order \( L \) without maximum and minimum. 
Therefore we can apply Lemmas~\ref{lem:generalcase} and~\ref{lemma:fromwreathtotrees} to get an \( L' \)-tree \( \tr_P \), where \( P = P^{\Sloc,\bpi} \) is the canonical partial order associated to \( \Wr^{\Sloc,\bpi}_{\delta \in \Delta} \Sym(N_\delta) \) and \( L' = \omega^* \cdot L \), such that \( \Wr^{\Sloc,\bpi}_{\delta \in \Delta} \Sym(N_\delta) \cong \Aut(\tr_P) \). 
Therefore we only need to show that when \( \tr_P = (T,{\leq_T},\lev_T) \) is constructed as in the proof of Lemma~\ref{lemma:fromwreathtotrees}, it satisfies property~\ref{propertystar}.  
Let \( C \subseteq T \) be a proper upward closed chain, and \( t \in T \) be such that \( [t] \leq_{\Delta(T)} [C] \). 
If \( C \) contains a node of the form \( z_\ell \), then we can take any \( f \in \Aut(\tr_P)\) such that \( f(t) \leq_T z_\ell \) and get that \( \bar t = f(t) \leq_T C \) because \( \leq_T \) is linear on \( \{ t' \in T \mid t' \leq_T z_\ell \} \). 
Therefore we can assume that \( C \subseteq T_0 = \omega^* \times P \). If \( C \) has a minimum we are done by Remark~\ref{rmk:propertystar}, thus we can assume that \( C  = \omega^* \times C' \) for some proper upward closed chain \( C' \) in the \( L \)-tree \( (P,{\leq_P},\lev_P) \). 
It follows that, by construction of \( \tr_P \), we can without loss of generality assume that \( t \) is of the form \( t = (0,z) \) for some \( z \in P \). 
The condition \( [t] \leq_T [C] \) entails that for every \( z' \in C' \) there is \( f \in \AutL(P) \) such that \( f(z) \leq_P z' \), i.e.\ \( \pi_{\gamma \gamma'}(f(z)) = z' \) where \( \gamma \leq_\Delta \gamma' \) are such that \( z \in \Sloc_{\gamma} \) and \( z' \in \Sloc_{\gamma'} \).
Let \( \hat z \in P \) be such that \( \hat z <_P C' \).
Since \( \langle \Sloc,\bpi \rangle \) is full, there is some coherent \( (z_\delta)_{\delta \in \Delta} \in \prod_{\delta \in \Delta} \Sloc_\delta \) such that \( \hat z = z_\delta \) for the appropriate \( \delta \in \Delta \), which in particular entails that \( z' = z_{\gamma'} = \pi_{\delta \gamma'}(\hat z) \) for every \( z' \in C' \) as above. 
Let \( \bar z = z_\gamma \): then by coherence \( \pi_{\gamma \gamma'}(\bar z) = z' \), i.e.\ \( \bar z \leq_P z' \), for every \( z' \in C \), that is, \( \bar z \leq_P C' \).
Since \( z,\bar z \in \Sloc_\gamma \), by local homogeneity there is \( f \in \AutL(P) \) such that \( f(z) = \bar z \). Then setting \( \bar t = (0,\bar z) \) we have that \( \bar t \in [t] \), as witnessed by \( f^\uparrow \in \Aut(\tr_P) \), and \( \bar t \leq_T C \) by construction (using \( \bar z \leq_P C' \)).

Finally, we show that \ref{thm:exactdistancesnew-2'} \( \Rightarrow \) \ref{thm:exactdistancesnew-3}.
Let \( \ell_0 = \min L \), and
let \( (\delta_j)_{j < J }\) be an enumeration without repetitions of \( \Lev_{\ell_0}(\Delta(T)) \), for the appropriate \( J \leq \omega \). 
We want to define a bijection \( t \mapsto z_t \) between \( T \) and \( \Sloc^{\LF} \) such that \( z_t \in \Sloc^{\LF}_{[t]} \) and \( t \leq_T t' \iff z_t \supseteq z_{t'} \) for every \( t,t' \in T \). The map \( z \mapsto z_t \) is then an order isomorphism \( g \colon \tr \to P \), where \( P \) is the canonical partial order associated to \( \Wr^{\LF}_{\delta \in \Delta(\tr)} \Sym(N^\tr_\delta)\), as in the hypothesis of Lemma~\ref{lemma:g}, yielding \( \Aut(\tr) \cong \Wr^{\LF}_{\delta \in \Delta(\tr)} \Sym(N^\tr_\delta) \). 

For each \( j < J \), let \( T_j = \{ t \in T \mid [t] \geq_{\Delta(T)} \delta_j \} \).
We first define \( t \mapsto z_t \) on \( T_0 \) by applying Lemma~\ref{lem:labelsfinitesupport} with \( C = \emptyset \) and \( \bar \delta = \delta_0 \). Then we recursively extend the map to \( T_j \), supposing that it has already been defined on \( \bigcup_{j' < j} T_{j'} \). 
Using the hypothesis that \( \tr \) is special and that there are only finitely many \( j' < j \), we can let \( \bar\ell = \min \{ \spl{\delta_{j'},\delta_j} \mid j' < j \} \). 
Then the map \( t \mapsto z_t \) is already defined on all \( t \in T_j \) with \( \lev_T(t)>_L \bar \ell \) and we only need to define it on all \( t \in T_j \) such that  \( \ell_0 \leq_L \lev_T(t) \leq_L \bar \ell \). 
Fix \( j' < j \) such that \( \spl{\delta_{j'},\delta_j} = \bar \ell \).
For \( t' \in \delta_{j'}\), let \( C^{(t')}  = \{ t'|_{\ell'} \mid \ell' >_L \bar \ell \} \). Then \( C^{(t')} \) is an upward closed chain in \( \tr \), and \( t' \) witnesses that it is proper. By construction, we also have \( \delta_j <_{\Delta(T)} [C^{(t')}]\).
Thus we can apply Lemma~\ref{lem:labelsfinitesupport} with \( C  = C^{(t')} \), \( \bar \delta = \delta_j \), and \( \bar z = \bigcup \{ z_t \mid t \in C^{(t')} \} \) and define the map \( t \mapsto z_t \) on the nodes of \( \tr^{\delta_j}_{C^{(t')}}\). 

In order to show that this procedure defines the map \( t \mapsto z_t \) on the whole \( T_{j} \), we show that
%
for every \( t \in \delta_j \) there is \( t' \in \delta_{j'} \) such that \( t \leq_T C^{(t')}\). 
%
To this aim, fix any \( t'' \in \delta_{j'} \), and consider the proper upward closed chain \( C^{(t'')} \). By choice of \( j' \) and definition of \( C^{(t'')}\), we have \( [t] \leq_{\Delta(T)} [C^{(t'')}] \), hence by property~\ref{propertystar} there is \( f \in \Aut(\tr) \) such that \( f(t) \leq_T C^{(t'')}\). Then \( t \leq_T C^{(t')} \) for \( t' = f^{-1}(t'') \).

It is easy to check that, by construction, the map \( t \mapsto z_t \) just defined has all the desired properties,
\details{Let \( \hat t \in T_j \). 
If \( \lev_T(\hat t) >_L \bar \ell \), then \( \hat t \in T_{j'} \) for some \( j'<j \) and \( z_{\hat t} \) was already defined in such a way that \( z_{\hat t} \in \Sloc^{\LF}_{[\hat t]} \). 
If \( \lev_T(\hat t) \leq_T \bar \ell \), then pick \( t \in \delta_j \) such that \( t \leq_T \hat t \) (which exists by Lemma~\ref{lem:inhabitated} applied to the branch \( b_j \) uniquely determined by \( \delta_j \)) and let \( t' \in \delta_{j'} \) be as above. 
Then \( \hat t \in \tr^{\delta_j}_{C^{(t')}}\) and \( z_{\hat t} \) has been defined in the previous step, so that in particular \( z_{\hat t} \in \Sloc^{\LF}_{[\hat t]} \). 

Finally, we check that for every \( \hat t' \), \( \hat t \leq_T \hat t' \iff z_{\hat t} \supseteq z_{\hat t'} \). Let \( \hat t' \) be the element used to define \( z_{\hat t} \). 
If \( \lev_T(\hat t) >_L \bar \ell \), then \( \hat t ,\hat t' \in T_{j'} \) and we can apply the inductive hypothesis. Hence we can assume that \( \lev_T(\hat t) \leq_L \bar \ell \).

If \( \hat t \leq_T \hat t' \), then \( \hat t' \in T_j \) as well. If \( \lev_T(\hat t') >_L \bar \ell \), then \( \hat t' \in C^{(t')}\) and thus \( z_{\hat t} \supseteq z_{\hat t'} \) by construction. 
If instead \( \lev_T(\hat t' ) \leq_L \bar \ell \), then \( \hat t' \in C^{(t')} \) as well, and hence \( z_{\hat t} \supseteq z_{\hat t'} \) by Lemma~\ref{lem:labelsfinitesupport}. 

Conversely, assume that \( z_{\hat t} \supseteq z_{\hat t' } \). Since \( z_{\hat t'} \in \Sloc^{\LF}_{[\hat t']} \) and \( z_{\hat t} \in \Sloc^{\LF}_{[\hat t]}\), this implies \( [\hat t] \leq_{\Delta(T)} [\hat t'] \), and hence \( \hat t' \in T_j \). 
Let \( \bar z \)\todo[color=green!40]{da qui in poi la notazione è sbagliata} be as in the definition of \( z_t \), so that \( z_t \supseteq \bar z \). 
If \( \lev_T(\hat t') >_L \bar \ell \) then \( z_{\hat'} \subseteq \bar z\), hence \(\hat  t' \in C^{(t')} \) and \( \hat t \leq_T \hat t' \); otherwise, \( z_{\hat t'} \supsetneq \bar z \), which can happen only when \( \hat t' \in \tr^{\delta_j}_{C^{(t')}} \), and then \( \hat t \leq_T \hat t' \) by Lemma~\ref{lem:labelsfinitesupport}.
}
and since \( \bigcup_{j<J} T_j = T \), this concludes our proof.
\end{proof}

\begin{remark} \label{rmk:exactdistancesnew}
Since in~\ref{thm:exactdistancesnew-3} of Theorem~\ref{thm:exactdistancesnew} the skeleton \( \langle \Delta,N \rangle \) is treeable, we again have \( \Wr^{\LF}_{\delta \in \Delta} \Sym(N_\delta) = \Wr_{\delta\in\Delta }^{ \UM } \Sym(N_{\delta }) \). 
\end{remark}

\begin{remark} \label{rmk:rigidity}
It can be shown that when the wreath products appearing in conditions~\ref{thm:exactdistancesnew-3}--\ref{thm:exactdistancesnew-3'''} of  Theorem~\ref{thm:exactdistancesnew} are obtained from a countable special pruned \( L \)-tree \( \tr \) satisfying~\ref{propertystar}, their skeleton \( \langle \Delta(\tr),N^\tr \rangle \) is \markdef{rigid}, that is, there is no nontrivial automorphism \( g \) of the \( L \)-tree \( \Delta(\tr) \) satisfying
\( N^{\tr}_{g(\delta)} = N^\tr_\delta \) for all \(  \delta \in \Delta(T) \).
\end{remark}
\details{\todo{Per ora la faccenda della rigidit\`a \`e stata trasformata in un remark, visto che non gioca pi\`u alcun ruolo. Si pu\`o ripristinarla (adattando tutto il testo alla nuova notazione/presentazione), oppure spostarla nella sezione dei ``risultati aggiuntivi'' al fondo.}
We conclude with some structural information on the wreath products appearing in conditions~\ref{thm:exactdistancesnew-3}--\ref{thm:exactdistancesnew-3'''} of  Theorem~\ref{thm:exactdistancesnew}: 
we can show that when they are obtained from a countable special pruned \( L \)-tree \( \tr \) satisfying~\ref{propertystar},
their skeleton exhibits a natural form of rigidity.

\begin{lemma} \label{lem:ridigskeleton}
Let \( \tr \) be a countable pruned special \( L \)-tree satisfying property~\ref{propertystar}, and consider the 
\( L \)-tree-skeleton \( \langle \Delta(\tr), N^\tr \rangle \). 
Then there is no nontrivial automorphism \( g \) of the \( L \)-tree \( \Delta(\tr) \) satisfying
\( N^\tr_{g(\delta)} = N^\tr_\delta \) for all \(  \delta \in \Delta(\tr) \).
\end{lemma}

\begin{proof} 
Let \( g \in \Aut(\Delta(\tr)) \) be such that \( N^\tr_{g(\delta)} = N^\tr_\delta \) for all \(  \delta \in \Delta(\tr) \): we must show that \( g \) is the identity on \( \Delta(\tr) \). 
To this aim, we construct an automorphism \( f \in \Aut(\tr) \) such that \( f(t) \in g([t]) \) for all \( t \in T \), so that \( g([t]) = [f(t)] =  [t] \) because \( f \) itself witnesses \( f(t) \sim t \).

Let \( \ell_0 \) be the minimum of \( L \) if it exists, or the first element of a strictly \( L \)-decreasing sequence \( (\ell_n)_{n  \in \omega} \) coinitial in \( L \) if \( L \) has no minimum. 
We build \( \varphi^{(0)} \in \Sym(\Lev_{\ell_0}(\tr)) \) such that \( \varphi^{(0)}(t) \in g([t]) \) and \( t|_\ell = t'|_\ell \iff \varphi^{(0)}(t)|_\ell = \varphi^{(0)}(t')|_\ell \), for every \( t,t' \in \Lev_{\ell_0}(\tr) \) and \( \ell \geq_L \ell_0 \). 
This is done in stages, via a back-and-forth argument.

Let \( (t_i)_{i < I} \) be an enumeration without repetition of \( \Lev_{\ell_0}(\tr) \), for the appropriate \( I \leq \omega \). We begin by letting \( \varphi^{(0)}_0(t_0) \) be any element of \( g([t_0]) \). 
Suppose now that \( k > 0 \) and that a finite function \( \varphi^{(0)}_{k-1} \) has been defined. 
If \( k \) is even, we let \( i < I \) be least such that \( t_i \notin \dom(\varphi^{(0)}_{k-1})\) (if there is any, otherwise we set \( \varphi^{(0)} = \varphi^{(0)}_{k-1} \) and we are done), and extend \( \varphi^{(0)}_{k-1}\) to \( \varphi^{(0)}_k \) by deciding the value of \( \varphi^{(0)}_k(t_i) \), so that \( \dom(\varphi^{(0)}_k) = \dom(\varphi^{(0)}_{k-1}) \cup \{ t_i \} \). 
Let \( \bar \ell = \min_L \{ \spl{t_i,t} \mid t \in \dom(\varphi^{(0)}_{k-1}) \} \), and let \( \bar t_i = t_i|_{\bar \ell} \). 
We distinguish two cases. 
If there is \( \hat t \in \dom (\varphi^{(0)}_{k-1})\) such that \( \hat t|_{\bar \ell}  \in C_{\bar t_i} \)\todo[color=green!40]{ricordare \eqref{eq:Ct}, che qui la notazione si confonde con la \(C\) delle catene}, then since \( N^\tr_{[\bar t_i]} = N^\tr_{g[\bar t_i]}\) we can find \( s_i \in C_{\varphi^{(0)}_{k-1}(\hat t)|_{\bar \ell}} \subseteq g([\bar t_i]) \) such that \( s_i \neq \varphi^{(0)}_{k-1}(t)|_{\bar \ell} \) for all \( t \in \dom(\varphi^{(0)}_{k-1}) \) with \( t |_{\bar \ell} \in C_{\bar t_i}\). 
Using Lemma~\ref{lem:inhabitated} applied in the obvious way to the tree \( \tr|_{\ell_0} \), we can then find \( t'_i \in g([t_i])\) such that \( t'_i \leq_T s_i \) and set \( \varphi^{(0)}_k(t_i) = t'_i \). 
Suppose now that \( t|_{\bar \ell} \not\sim \bar t_i \) for every \( t \in \dom(\varphi^{(0)}_{k-1}) \) such that \( \spl{t_i,t} = \bar \ell \). 
Fix any such \( t \), and let \( s_i \) be any element of \( g([\bar t_i]) \). 
Then \( s_i \) and \( \varphi^{0}_{k-1}(t)|_{\bar \ell} \) satisfy the hypothesis of property~\ref{propertystar}, that is, for every \( \ell >_L \bar \ell \) there is \( s'_i \sim s_i \) such that \( s'_i|_\ell = (\varphi^{(0)}_{k-1}(t)|_{\bar \ell})|_\ell = \varphi^{(0)}_{k-1}(t)|_\ell \). 
Therefore by~\ref{propertystar} we can find \( s''_i \sim s_i \) such that \( s''_i |_\ell = \varphi^{(0)}_{k-1}(t)|_\ell \) for all \( \ell >_L \bar \ell \).
Notice that by case assumption, for every \( t \in \dom(\varphi^{(0)}_{k-1}) \) such that \( \spl{t_i,t} = \bar \ell  \) we have \( s''_i \neq \varphi^{(0)}_{k-1}(t)|_{\bar \ell} \), as otherwise \( g([t|_{\bar \ell}]) = [\varphi^{(0)}_{k-1}(t)|_{\bar \ell}] = [s''_i] = [s_i] =  g([\bar t_i]) \) and hence \( [\bar t_i] = [t|_{\bar \ell}] \), i.e.\ \(  t|_{\bar \ell} \sim \bar t_i \). 
Therefore we can set \( \varphi^{(0)}_{k}(t_i) = t'_i \) for any \( t'_i \in g([t_i]) \) with \( t'_i \leq_T s''_i \), which exists by Lemma~\ref{lem:inhabitated} again.
If \( k \) is odd, instead, we let \( i < I \) be least such that \( t_i \notin \ran(\varphi^{(0)}_{k-1}) \), and apply the above construction to \( g^{-1} \) and \( (\varphi^{(0)}_{k-1})^{-1} \) to ensure \( t_i \in \ran(\varphi^{(0)}_k) \). The construction stops after \( K \)-many stages, for the appropriate \( K \leq \omega \), i.e.\ when \( \varphi^{(0)} = \bigcup_{k < K} \varphi^{(0)}_k\) is a permutation of the whole \( \Lev_{\ell_0}(\tr) \): it is easy to check that such \( \varphi^{(0)} \) is as desired.

If \( \ell_0 \) was the minimum of \( L \), then we can conclude the proof by setting \( f(t) = \varphi^{(0)}(t')|_{\lev_T(t)} \) for some (equivalently, any) \( t' \in \Lev_{\ell_0}(\tr) \) such that \( t' \leq_T t \).

If instead \( L \) had no minimum, by recursion on \( n \in \omega \) we build permutations \( \varphi^{(n)} \in \Sym(\Lev_{\ell_n}(\tr))\) such that for all \( t,t' \in \Lev_{\ell_n}(\tr) \) and \( \ell \geq_L \ell_n \) the following conditions hold: \( \varphi^{(n)}(t) \in g([t])\), \( t|_\ell = t'|_\ell \iff \varphi^{(n)}(t)|_\ell = \varphi^{(n)}(t')|_\ell \), and moreover 
\( \varphi^{(n)}(t)|_{\ell_{n-1}} = \varphi^{(n-1)}(t|_{\ell_{n-1}}) \) if \( n > 0 \). (The latter is to ensure that \( \varphi^{(n)} \) is compatible with \( \varphi^{(n-1)} \).) This is done via a back-and-forth argument as in the case \( n = 0 \), with the following variations. Fix \( (t_i)_{i<I} \), an enumeration without repetitions of \( \Lev_{\ell_n}(\tr) \). We again let \( \varphi^{(n)}_0(t_0) \) be any element  \( t'_0 \in g([t_0]) \), but since \( n > 0 \) we also need to ensure that \( t'_0 |_{\ell_{n-1}} = \varphi^{(n-1)}(t_0|_{\ell_{n-1}})\): this can be done thanks to Lemma~\ref{lem:inhabitated} applied to \( \tr|_{\ell_n} \). 
At even stages \( k > 0 \), when we have to add \( t_i \) to \( \dom(\varphi^{(n)}_{k}) \), we again let \( \bar \ell  = \min_L \{ \spl{t_i,t}  \mid  t \in \dom(\varphi^{(n)}_{k-1}) \}  \), and compare it with \( \ell_{n-1} \).
If \( \bar \ell \geq_L \ell_{n-1} \), then we argue as in the case \( k = 0 \) and let \( \varphi^{(n)}_k(t_i) \)  be any node in \( g([t_i]) \) that is \( \leq_T \)-below \( \varphi^{(n-1)}(t_i|_{\ell_{n-1}}) \). If instead \( \bar \ell <_L \ell_{n-1} \), then we argue exactly as in the even nonzero stages of the case \( n = 0 \). For odd \( k > 0 \) we work similarly, but in the backward direction.

Once all the permutations \( \varphi^{(n)} \) have been defined, it is enough to define \( f \) by setting \( f(t) = \varphi^{(n)}(t')|_{\lev_T(t)} \) for any \( \ell_n \leq_L \lev_T(t) \) and \( t' \in \Lev_{\ell_n}(\tr) \) with \( t' \leq_T t\).
\end{proof}
}

\subsection{The general case}\label{sec:generalcase}

We finally consider the case of isometry groups of arbitrary Polish ultrametric spaces. 
Notice that the equivalence between~\ref{thm:generalcase-1} and~\ref{thm:generalcase-2} below is analogous to the equivalence between isometry groups of Polish ultrametric spaces and automorphism groups of countable \( R \)-trees implicitly obtained by Gao and Shao in~\cite[Lemma 6.8, Proposition 6.11, and Proposition 6.14]{GaoShao2011}.

\begin{theorem} \label{thm:generalcase}
For every topological group \( G \), the following are equivalent:
\begin{enumerate-(1)}
\item \label{thm:generalcase-1}
\( G \cong \Iso(U) \) for some Polish ultrametric space \( U \);
\item \label{thm:generalcase-1'}
\( G \cong \Iso(U) \) for some perfect locally compact Polish ultrametric space \( U \);
\item \label{thm:generalcase-1''}
\( G \cong \Iso(U) \) for some uniformly discrete Polish ultrametric space \( U \);
\item \label{thm:generalcase-2}
\( G \cong \Aut(\tr) \) for some countable pruned \( L \)-tree \( \tr \) and some linear order \( L \);
\item \label{thm:generalcase-2'}
\( G \cong \Aut(\tr) \) for some countable pruned special \( L \)-tree \( \tr \) and some linear order \( L \) with a minimum;
\item \label{thm:generalcase-2''}
\( G \cong \Aut(\tr) \) for some countable pruned special \( \QQ \)-tree \( \tr \); 
\item \label{thm:generalcase-3}
\( G \cong \Wr^{\LF,\bpi}_{\delta \in \Delta} \Sym(N_\delta) \) for some countable treeable skeleton \( \langle \Delta, N \rangle \) and some 
locally homogeneous system of projections of the form \( \langle \Sloc^{\LF}, \bpi \rangle \) with finite character; 
\item \label{thm:generalcase-3'}
\( G \cong \Wr^{\Sloc}_{\delta \in \Delta} \Sym(N_\delta) \) for some countable treeable skeleton \( \langle \Delta, N \rangle \) and some countable family of local domains \( \Sloc \subseteq \Sloc^{\Max} \); 
\item \label{thm:generalcase-3''}
\( G \cong \Wr^{\Sloc,\bpi}_{\delta \in \Delta} \Sym(N_\delta) \) for some countable treeable skeleton \( \langle \Delta, N \rangle \) and some countable system of projections \( \langle \Sloc,\bpi \rangle \) such that \( \Sloc \subseteq \Sloc^{\Max} \). 
\end{enumerate-(1)}
\end{theorem}

\begin{proof}
The equivalence of items~\ref{thm:generalcase-1}--\ref{thm:generalcase-2''} can be proved as in Theorem~\ref{thm:exactdistancesnew} (forgetting about Lemma~\ref{lem:exactdistancesprelim}).
\details{
The implication \ref{thm:generalcase-1} \( \Rightarrow \) \ref{thm:generalcase-2''}
follows from Corollary~\ref{cor:spacestotrees2} and the observation following it, and \ref{thm:generalcase-2''} \( \Rightarrow \) \ref{thm:generalcase-2} is obvious. 

To prove \ref{thm:generalcase-2} \( \Rightarrow \) \ref{thm:generalcase-1''}, first observe that, as in the proof of \ref{thm:exactdistancesnew-2} \( \Rightarrow \) \ref{thm:exactdistancesnew-1'} of Theorem~\ref{thm:exactdistancesnew}, we can assume that the countable linear order \( L \) has no minimum. 
Therefore, if \( D \subseteq \RR^+ \) is the range of an embedding of \( 2 \cdot L \) into \( \RR^+\) satisfying \( \inf D > 0 \), then using Theorem~\ref{thm:treestospaces} and Corollary~\ref{cor:treestospaces} we get the desired implication (recall that the condition \( \inf D > 0 \) ensures that the Polish ultrametric space we obtained is uniformly discrete).

To prove \ref{thm:generalcase-1''} \( \Rightarrow \) \ref{thm:generalcase-1'} we use instead Theorem~\ref{thm:toperfectlocallycompact} and Corollary~\ref{cor:toperfectlocallycompact}, picking as \( D' \subseteq \RR^+ \) the set \( D \cup \{ 2^{-n} \mid n \in \omega \} \), where \( D \) is the distance set of the uniformly discrete Polish ultrametric space given as input. Condition \( \inf D > 0 \) is ensured by uniform discreteness.
Since \ref{thm:generalcase-1'} \( \Rightarrow \) \ref{thm:generalcase-1} is obvious, we have shown that  \ref{thm:generalcase-1}--\ref{thm:generalcase-2} and \ref{thm:generalcase-2''} are equivalent to each other.

To prove \ref{thm:generalcase-1''} \( \Rightarrow \) \ref{thm:generalcase-2'} we can again use
the functor \( \func{F} \) from Section~\ref{subsec:F} (and in particular Theorem~\ref{thm:spacestotrees} and Corollary~\ref{cor:spacestotrees}), exploiting the fact that \( \inf D > 0 \) when \( D \) is the distance set of a uniformly discrete Polish ultrametric space, and thus we can choose \( L \) as in the first item of Lemma~\ref{lem:spacestotreesspecial}. The implication \ref{thm:generalcase-2'} \( \Rightarrow\) \ref{thm:generalcase-2} is again trivial, thus also \ref{thm:generalcase-2'} is equivalent to all other conditions up to \ref{thm:generalcase-2''}.
}
The implication~\ref{thm:generalcase-3} \( \Rightarrow \) \ref{thm:generalcase-3'} follows from Theorem~\ref{thm:proj->unchained}, and \ref{thm:generalcase-3'} \( \Rightarrow \) \ref{thm:generalcase-3''} is obvious. 

We now prove \ref{thm:generalcase-3''} \( \Rightarrow \) \ref{thm:generalcase-2}. Fix any projective wreath product \( \Wr^{\Sloc,\bpi}_{\delta \in \Delta} \Sym(N_\delta) \) such that \( \langle \Delta, N \rangle \) is countable and treeable, and \( \Sloc \subseteq \Sloc^{\Max} \) is countable. 
By Lemma~\ref{lem:nominandmax}, we can assume that \( \langle \Delta, N \rangle \) is \( L \)-treeable for some linear order \( L \) without maximum and minimum. 
Therefore we can apply Lemmas~\ref{lem:generalcase} and~\ref{lemma:fromwreathtotrees} to get the desired result.  

To conclude the proof, we show that \ref{thm:generalcase-2'} \( \Rightarrow \) \ref{thm:generalcase-3}.
Let \( \ell_0 \) be the minimum of \( L \), and let \( \tr = (T,{\leq_T},\lev_T) \) be a countable pruned special \( L \)-tree.
Let \( (\delta_j)_{j < J}\) be an enumeration without repetitions of \( \Lev_{\ell_0}(\Delta(\tr)) \), for the appropriate \( J \leq \omega \). 
For every \( j < J \), apply Lemma~\ref{lem:labelsfinitesupport} with \( C = \emptyset\) and \( \bar \delta = \delta_j \)
to obtain a bijection \( t \mapsto z^{(j)}_t\) between \( \tr_j = \tr^{\delta_j}_{\emptyset} \) and \( \bigcup_{\ell \in L} \Sloc^{\LF}_{\delta_j|_\ell} \) satisfying \( z^{(j)}_t \in \Sloc^{\LF}_{[t]} \) and, for all \( t,t' \in \tr_j \), 
\begin{equation} \label{eq:z^j_t} 
t \leq_T t' \iff z^{(j)}_t \supseteq z^{(j)}_{t'} ,
\end{equation}
so that in particular for every \( \ell \geq_L \lev_T(t)\) and \( \gamma = [t|_\ell] \) we have
\begin{equation} \label{eq:z^j_t2} 
z^{(j)}_{t|_\ell} = \big( z^{(j)}_t \big)|_\gamma.
\end{equation}

For each \( \beta \in \Delta(\tr) \), let \( j_\beta < J \) be least such that \( \delta_{j_\beta} \leq_{\Delta(\tr)} \beta \), and for \( \delta \leq_{\Delta(\tr)} \gamma \) define \( \pi_{\delta\gamma} \colon \Sloc^{\LF}_\delta \to \Sloc^{\LF}_\gamma \) as follows:
given \( z \in \Sloc^{\LF}_\delta \), let \( t \in \delta \) be the unique element of \( \tr_{j_\delta} \) such that \( z = z^{(j_\delta)}_t \), and then set \( \pi_{\delta\gamma}(z) = z^{(j_\gamma)}_{t|_\ell}\), where \( \ell = \lev_{\Delta(\tr)} (\gamma) \). In particular, by~\eqref{eq:z^j_t2} we get that 
\begin{equation} \label{eq:finitecharacter} 
j_\gamma = j_\delta \ \  \Longrightarrow \ \  \pi_{\delta \gamma}(z) = z|_\gamma .
\end{equation}
We claim that pairing \( \Sloc^{\LF}\) with the family \( \bpi = (\pi_{\delta\gamma})_{\gamma \geq_{\Delta(\tr)} \delta}\) defined above, we get a system of projections \( \langle \Sloc^{\LF},\bpi \rangle \) over the \( L \)-treeable (because \( \tr \) is special) skeleton \( \langle \Delta(\tr),N_\tr \rangle \) which has finite character. 

The family \( \Sloc^{\LF}\) obviously satisfies the first item of Definition~\ref{def:projections}. Fix any \( \gamma, \delta \in \Delta(\tr) \) such that \( \gamma \geq_{\Delta(\tr)} \delta\), and let \( \ell = \lev_{\Delta(\tr)}(\gamma) \) and \( \ell' = \lev_{\Delta(\tr)}(\delta) \). 
We first show that \( \pi_{\delta \gamma } \colon \Sloc^{\LF}_\delta \to \Sloc^{\LF}_\gamma \) is surjective. Let \( b \in [\Delta(\tr)] \) be the branch determined by \( \delta_{j_\delta} \) (i.e.\ \( b \) is such that \( b(\ell_0) = \delta_{j_\delta}\)). 
Given \( z' \in \Sloc^{\LF}_\gamma \), let \( t' \in \gamma = b(\ell) \) be such that \( z' = z^{(j_\gamma)}_{t'}\). By Lemma~\ref{lem:inhabitated}, there is \( t \in \delta = b(\ell') \) such that \( t|_\ell = t' \): set \( z = z_t^{(j_\delta)}\). Then \( \pi_{\delta \gamma}(z) = z'\), as desired. 

Pick now \( z,z' \in \Sloc^{\LF}_\delta  \), and let \( t,t' \in \delta \) be the node of \( \tr \) used to define \( \pi_{\delta \gamma}(z) \) and \( \pi_{\delta \gamma}(z')\), respectively.
Then by injectivity of \( s \mapsto z^{(j_\gamma)}_s \) and~\eqref{eq:z^j_t2} applied with \(j = j_\delta\)
we have
\begin{multline*}
\pi_{\delta \gamma}(z) = \pi_{\delta \gamma}(z') \iff z^{(j_\gamma)}_{t|_\ell} = z^{(j_\gamma)}_{t'|_\ell} \iff t|_\ell = t'|_{\ell} \\ \iff  z^{(j_\delta)}_{t|_\ell} = z^{(j_\delta)}_{t'|_\ell} \iff \big(z^{(j_\delta)}_t\big)|_\gamma = \big(z^{(j_\delta)}_{t'}\big)|_\gamma \iff z|_\gamma = z'|_\gamma.
\end{multline*}
This shows that~\ref{def:projections-a} in the second item of Definition~\ref{def:projections} is satisfied. To prove that also~\ref{def:projections-b} is satisfied, pick \( \beta \geq_{\Delta(\tr)} \gamma \) and let \( \ell' = \lev_{\Delta(\tr)}(\beta) \). Given \( z \in \Sloc^{\LF}_\delta\), let \( t \in \delta \) be the node of \( \tr \) used to define both \( \pi_{\delta \gamma }(z)\) and \( \pi_{\delta \beta }(z)\). Since \( \pi_{\delta \gamma}(z) = z^{(j_\gamma)}_{t|_\ell}\), it follows that \( t|_\ell \in \gamma \) is precisely the node of \( \tr \) used to define \( \pi_{\gamma \beta}(\pi_{\delta \gamma}(z))\). Therefore
\[  
\pi_{\gamma \beta}(\pi_{\delta \gamma}(z)) = z^{(j_\beta)}_{(t|_\ell)|_{\ell'}} = z^{(j_\beta)}_{t|_{\ell'}} = \pi_{\delta \beta}(z).
\]

This concludes the proof of the fact that \( \langle \Sloc^{\LF}, \bpi \rangle \) is a system of projections. The fact that it has finite character follows from~\eqref{eq:finitecharacter} and the fact that if \( \gamma \geq_{\Delta(\tr)} \delta \) then \( j_\gamma \leq j_\delta \), so that it is enough to let \( C_0, \dotsc, C_n \) be an enumeration of the nonempty sets of the form
\( \{ \gamma \geq_{\Delta(\tr)} \delta \mid j_\gamma = j \} \),
for \( j \leq j_\delta \).

We next show that \( \Aut(\tr) \cong \Wr^{\LF,\bpi}_{\delta \in \Delta(\tr)} \Sym(N^\tr_\delta) \).
To this aim let \( P = P^{\Sloc^{\LF},\bpi}\) be the canonical partial order associated to \( \Wr^{\LF,\bpi}_{\delta \in \Delta(\tr)} \Sym(N^\tr_\delta) \). By Lemma~\ref{lemma:g} it is enough to provide an order-isomorphism \( g \colon T \to P \) such that \( g(t) \in \Sloc^{\LF}_\delta \) for every \( \delta \in \Delta(\tr)\) and \( t \in \delta \): we claim that this can be realized by setting \( g(t) = z_t^{(j_\delta )}\). 
By construction, \( g \) is a bijection and \( g(t) \in \Sloc^{\LF}_\delta \). Moreover, for any \( t \in \tr \), \( \delta = [t] \in \Delta(\tr) \), and \( \ell \geq \lev_T(t) \), we have
\[  
g(t|_\ell) = z^{(j_\gamma)}_{t|_\ell} = \pi_{\delta \gamma}(z^{(j_\delta)}_t) = \pi_{\delta \gamma}(g(t)),
\]
where \( \gamma = [t|_\ell] \). 
This implies that \( t \leq_T t' \iff g(t) \leq_P g(t') \) for every \( t,t' \in T \).

Finally, to show that \( \langle \Sloc^{\LF}, \bpi \rangle \) is locally homogeneous, we observe that, by construction, for every \( \delta \in \Delta(\tr) \) and \( z,z' \in \Sloc^{\LF}_\delta \) there is \( \varphi \in \Aut(\tr) \) such that \( \varphi(g^{-1}(z)) = g^{-1}(z') \). By the way \( f_\varphi \in \AutL(P) = \Wr^{\LF,\bpi}_{\delta \in \Delta} \Sym(N^\tr_\delta) \) is defined in the statement of Lemma~\ref{lemma:g}, this means that \( f_\varphi(z) = z' \), as desired.
\end{proof}

\begin{remark} \label{rmk:generalcase}
As a by-product of Theorem~\ref{thm:generalcase}, every collection of systems of projections between those employed in items~\ref{thm:generalcase-3} and~\ref{thm:generalcase-3''} gives rise to the same class of Polish groups, and thus characterizes isometry groups of Polish ultrametric spaces.
For example, \ref{thm:generalcase-3} could be weakened by dropping the local homogeneity requirement or the request of having finite character (or both). On the other hand, \ref{thm:generalcase-3''} could be strengthened by restricting to projective wreath products over local homogeneous skeletons. The local homogeneity condition can be added to~\ref{thm:generalcase-3'} as well, thanks to Remark~\ref{rmk:proj->unchained}.
\end{remark}

\section{Additional results and open problems} \label{sec:openproblems}

The functors \( \func{F} \) and \( \func{G} \) from Sections~\ref{subsec:F} and \ref{subsec:G}, respectively, can be construed as Borel functions when restricted to the standard Borel spaces of (codes for) Polish ultrametric spaces in \( \cat{U}_D \) or countable pruned \( L \)-trees in \( \cat{T}_L \).
Moreover, categorical full embeddings are also reductions between the corresponding isomorphism and embeddability relations; for example, if \( \tr,\tr' \) are countable pruned \( L \)-trees,  then \( \tr \) embeds into \( \tr' \) if and only if there is an isometric embedding of \( \func{G}(\tr) \) into \( \func{G}(\tr')\), and \( \tr \) is isomorphic to \( \tr' \) if and only if \( \func{G}(\tr) \) and \( \func{G}(\tr') \) are isometric. 
Finally, categorical full embeddings preserve the automorphism groups of the objects (algebraically); for example, \( \Aut(\tr) \simeq \Iso(\func{G}(\tr))\). 
It then follows that Theorems~\ref{thm:spacestotrees} and~\ref{thm:treestospaces} can be used to reprove the main results of Sections 5.1 and 6.1 of~\cite{Camerlo:2015ar}. 
(Here we use that rooted combinatorial trees without terminal nodes can be identified with pruned \( \omega^* \)-trees in our sense, and that \( 2 \cdot \omega^* \cong \omega^* \) embeds into any ill-founded set of distances \( D \subseteq \RR^+ \).)

The same comments apply to the functor \( \func{U} \) from Section~\ref{subsec:U}. This yields the following Borel-reducibility results, which escaped the analysis in~\cite{Camerlo:2015ar} --- we refer the reader to that paper for more details on the concepts and the background involved.

\begin{theorem}
\begin{enumerate-(1)}
\item 
The relation of isometry on the class of perfect locally compact Polish ultrametric spaces is Borel bi-reducible with graph isomorphism, and thus is \( S_\infty \)-complete.
\item 
The relation of isometric embeddability on the class of perfect locally compact Polish ultrametric spaces is invariantly universal, and hence complete for analytic quasi-orders.
\end{enumerate-(1)}
\end{theorem}

The characterizations presented in this paper might provide a new tool to understand  the complexity under Borel reducibility of the isomorphism relation on classes of non-Archimedean Polish groups (that is, closed subgroups of $ \Sym (\omega)$), a program that started by Kechris, Nies and Tent in~\cite{KNT2018}. 
In particular, one problem considered in this line of research is establishing the complexity of the isomorphism relation on isometry groups of Polish ultrametric spaces.
As a test case, one could first restrict the attention to Polish ultrametric Urysohn spaces. 
A related problem was already considered
in \cite{EGLMM2021,EGL}, where the authors studied the isomorphism types of isometry groups of $\Delta $-metric Urysohn spaces, and showed that they are completely determined by their distance set \( \Delta \) (up to a natural equivalence notion). It is natural to ask if a similar phenomenon arises also in the context of isometry groups of Polish ultrametric Urysohn spaces. 
Our structural results, and in particular Theorem~\ref{thm:Urysohn}, may help to show that there are many non-isomorphic such groups, as a step to prove that the isomorphism relation on them is very complicated.

Our results might also be helpful in tackling some natural open problems in the context of topological groups. For example, Gao and Kechris showed that if the Polish ultrametric space \( X \) is Heine-Borel, then \( \Iso(X) \) is the closure of an increasing union of compact subgroups, and hence it is amenable (\cite[Theorem 8.9]{GaoKec}). This raises the question of which other isometry groups of Polish ultrametric spaces are amenable. 
Similar questions might be asked for interesting group properties other than amenability, such as being oligomorphic, or having uncountable strong cofinality. Concerning the latter, for example, one might attempt to generalize Malicki's characterization~\cite[Theorem 5.3]{malick2014}, which crucially uses generalized wreath products, from \( W \)-spaces to arbitrary Polish ultrametric spaces.

In a somewhat different direction, a natural prolongation of the present work is to seek for characterizations of isometry groups of Polish ultrametric spaces belonging to some natural subclasses, such as compact spaces, proper spaces, Heine-Borel spaces, and so on.

On the group-theoretic side, it would be desirable to elaborate more on the various classes of generalized wreath products employed in our characterizations. For example, it is unclear whether the four classes identified by the columns of the table in Appendix~\ref{sec:table} are distinct from each other; by Proposition~\ref{prop:closedvsfull} and Theorem~\ref{thm:proj->unchained}, any generalized wreath product used to distinguish the various situations need to have an infinite skeleton. One might also wonder whether generalized wreath products over families of local domains \( \Wr^{\Sloc}_{\delta \in \Delta} H_\delta \) and projective wreath products \( \Wr^{\Sloc,\bpi}_{\delta \in \Delta} H_\delta \) enjoy natural universality properties analogous to those uncovered by Holland and Malicki for the groups of the form \( \Wr^{\Max}_{\delta \in \Delta} H_\delta \) and \( \Wr^{\UM}_{\delta \in \Delta} H_\delta \), respectively. We expect a positive answer for both problems, but even negative answers would have interesting consequences concerning the study of Polish ultrametric spaces.

\details{
\section{Further work}

\begin{enumerate}
\item
Find analogous characterizations for other classes of Polish ultrametric spaces: compact, proper, Heine-Borel, and so on. 
\item
Study (universality) properties of generalized wreath products over arbitrary families of local domains, and of projective wreath products. 
\item
Thorough comparison among the various kinds of wreath products. 
\item
Find applications.  For example, we expect that our analysis will enable us to reprove and generalize some existing structural results: 
\begin{itemize}
\item 
For which Polish ultrametric spaces \( X \) is the conjugacy equivalence relation on \( \Iso(X) \) Borel bi-reducible with graph isomorphism? (See Gao-Kechris)
\item 
Extend to all Polish ultrametric spaces, or at least to the exact ones, Malicki's characterization of the class of \( W \)-spaces whose isometry group has uncountable strong cofinality. 
\item
Determine when the isometry group of a Polish ultrametric space is amenable (see Gao-Kechris), or when it is oligomorphic, etc...
\end{itemize}
\end{enumerate}}

\newpage

\appendix
\section{Summarizing table} \label{sec:table}

In the following table, 
\( \langle \Delta,N \rangle \) is an arbitrary countable treeable skeleton, each global domain \( S \subseteq \prod_{\delta \in \Delta} N_\delta \) is assumed to be closed and separable, while all families of local domains \( \Sloc \) and all systems of projections \( \langle \Sloc,\bpi \rangle \) are assumed to be countable.

\begin{center}
\resizebox{\textwidth}{!}{%
\begin{tabular}{|c||c|c|c|c|} 
\hline
\multirow{5}{2.3cm}{\centering 
 \\ \( \Iso(U) \) for \\
\( U \) Polish ultrametric space} &
\multirow{5}{2.5cm}{\centering \( U \) homogeneous \\ and \\ unif.\ discrete} &
\multirow{5}{2.5cm}{\centering \( U \) homogeneous} &
\multirow{1}{3.3cm}{\centering \( U \) exact} &
\multirow{1}{2.5cm}{\centering all \( U \)} \bigstrut \\
\cline{4-5}
&
&
& 
\multirow{2}{3.3cm}{\centering \( U \) exact and \\ unif.\ discrete}&
\multirow{2}{2.5cm}{\centering \( U \) unif.\ discrete} \bigstrut \\
&
&
&
&
\\
\cline{4-5}
&
&
&
\multirow{2}{3.3cm}{\centering \( U \) exact and \\ perfect loc.\ compact} &
\multirow{2}{2.5cm}{\centering \( U \) perfect loc.\ compact} \bigstrut \\
&
&
&
&
\\
\hline \hline
\multirow{5}{2.3cm}{\centering 
  \( \Aut(\tr) \) for \( \tr \) countable pruned \( L \)-tree} &
\multirow{5}{2.5cm}{\centering \( \tr \) homogeneous \\ and \\ \( L \) has \( \min \)} &
\multirow{5}{2.5cm}{\centering \( \tr \) homogeneous} &
\multirow{1}{3.3cm}{\centering \( \tr \) \( L \)-tree with~\ref{propertystar}} &
\multirow{1}{2.5cm}{\centering all \( \tr \)} \bigstrut \\
\cline{4-5}
&
&
& 
\multirow{2}{3.3cm}{\centering \( \tr \) special \( L \)-tree with \ref{propertystar}  and \( L \) has \( \min \)}&
\multirow{2}{2.5cm}{\centering \( \tr \) special \( L \)-tree and \( L \) has \( \min \)} \bigstrut \\
&
&
&
&
\\
\cline{4-5}
&
&
&
\multirow{2}{3.3cm}{\centering \( \tr \) special \( \QQ \)-tree \\ with~\ref{propertystar}} &
\multirow{2}{2.5cm}{\centering \( \tr \) special \\ \( \QQ \)-tree} \bigstrut \\
&
&
&
&
\\
\hline \hline
\multirow{4}{2.3cm}{\centering Locally finite \\ supports: \\ \vspace{0.3cm}
\( \Wr^{\LF, -}_{\delta \in \Delta} \)} &
\multirow{4}{2.8cm}{\centering \(\Wr^{\LF}_{\delta \in \Delta} \Sym(N_\delta) \) \\ \vspace{0.2cm}\( \Delta \) linear with \( \min \)} &
\multirow{4}{2.5cm}{\centering \(\Wr^{\LF}_{\delta \in \Delta} \Sym(N_\delta) \) \\ \vspace{0.2cm} \( \Delta \) linear} &
\multirow{4}{3.3cm}{\centering \vspace{-0.4cm} \\ \(\Wr^{\LF}_{\delta \in \Delta} \Sym(N_\delta) \) \\ \vspace{0.2cm}\( {}^{(\dagger)}\) } &
\multirow{4}{2.5cm}{\centering \vspace{-0.4cm} \\  \(\Wr^{\LF,\bpi}_{\delta \in \Delta} \Sym(N_\delta) \) \\ \vspace{0.2cm}\( {}^{(\ddagger)} \)} \bigstrut \\
&
&
&
&
\\
&
&
&
&
\\
&
&
&
&
\\
\hline
\multirow{7}{2.5cm}{\centering ``Classical'' \\ generalized \\ wreath \\ products: \\ \vspace{0.3cm}
\( \Wr^{\Fin/\UM}_{\delta \in \Delta} \)} &
\multirow{2}{2.8cm}{\centering Hall's groups \\ with topology \( \tau^* \)} &
\multirow{7}{2.5cm}{\centering Malicki's groups \\ with \( \Delta \) linear} &
\multirow{7}{3.3cm}{\centering \vspace{0.2cm} \\  Malicki's \\ groups \\ \vspace{0.2cm}\({}^{(\dagger)}\)} &
\multirow{7}{2.5cm}{\centering ---} \bigstrut \\
&
&
&
&
\\
\cline{2-2}
&
\multirow{2}{2.5cm}{\centering Hall's groups \\ \( \Delta \) with \( \min \)} &
&
& \bigstrut
\\
&
&
&
&
\\
\cline{2-2}
&
\multirow{3}{2.8cm}{\centering Malicki's groups \\ \( \Delta \) linear with \( \min \)} &
&
& \bigstrut
\\
&
&
&
&
\\
&
&
&
&
\\
\hline
\multirow{5}{2.3cm}{\centering \( S \) global \\ domain: \\ \vspace{0.2cm} \( \Wr^S_{\delta \in \Delta} \)  \\ \vspace{0.2cm} \( \big[S \subseteq S^{\Max} \big] \)} &
\multirow{5}{2.8cm}{\centering \( \Wr^S_{\delta \in \Delta} \Sym(N_\delta) \) \\ \vspace{0.2cm} \( \Delta \) linear with \( \min \) \\ \( S \) appr.\ homog.} &
\multirow{5}{2.5cm}{\centering \( \Wr^S_{\delta \in \Delta} \Sym(N_\delta) \) \\ \vspace{0.2cm} \( \Delta \) linear \\ \( S \) appr.\ homog.} &
\multirow{5}{3.3cm}{\centering \vspace{-0.1cm} \\ \( \Wr^S_{\delta \in \Delta} \Sym(N_\delta) \) \\ \vspace{0.2cm} \( S \) appr.\ homog. \\ \vspace{0.2cm}\({}^{(\dagger)}\)} &
\multirow{5}{2.5cm}{\centering ---} \bigstrut \\
&
&
&
&
\\
&
&
&
&
\\
&
&
&
&
\\
&
&
&
&
\\
\hline
\multirow{5}{2.3cm}{\centering \( \Sloc \)  family of \\ local domains: \\ \vspace{0.2cm} \( \Wr^{\Sloc}_{\delta \in \Delta} \) \\ \vspace{0.2cm} \( \big[\Sloc \subseteq \Sloc^{\Max} \big] \)} &
\multirow{5}{2.8cm}{\centering  \vspace{-0.35cm} \\ \( \Wr^{\Sloc}_{\delta \in \Delta} \Sym(N_\delta) \) \\ \vspace{0.2cm} \( \Delta \) linear with \( \min \) \\ \( \Sloc \) local.\ homog. \\ \vspace{0.2cm}\({}^{(*)}\) } &
\multirow{5}{2.8cm}{\centering \vspace{-0.35cm} \\ \( \Wr^{\Sloc}_{\delta \in \Delta} \Sym(N_\delta) \) \\ \vspace{0.2cm} \( \Delta \) linear \\ \( \Sloc \) local.\ homog. \\ \vspace{0.2cm}\({}^{(*)}\)} &
\multirow{5}{3.3cm}{\centering \ \\ \vspace{-0.35cm}  \( \Wr^{\Sloc}_{\delta \in \Delta} \Sym(N_\delta) \) \\ \vspace{0.2cm} \( \Sloc \) full and \\ locally homogeneous \\ \vspace{0.2cm}\({}^{(\dagger)}\)} &
\multirow{5}{2.5cm}{\centering \vspace{0.3cm} \\ \( \Wr^{\Sloc}_{\delta \in \Delta} \Sym(N_\delta) \) 
\\ \vspace{0.57cm}\( {}^{(\sharp)} \)} \bigstrut \\
&
&
&
&
\\
&
&
&
&
\\
&
&
&
&
\\
&
&
&
&
\\
\hline
\multirow{6}{2.3cm}{\centering Projective \\ wreath \\ products: \\ \vspace{0.2cm} \( \Wr^{\Sloc, \bpi}_{\delta \in \Delta} \) \\ \vspace{0.2cm} \( \big[ \Sloc \subseteq \Sloc^{\Max} \big] \) } &
\multirow{6}{2.8cm}{\centering \vspace{-0.2cm} \\ \( \Wr^{\Sloc,\bpi}_{\delta \in \Delta} \Sym(N_\delta) \) \\ \vspace{0.2cm} \( \Delta \) linear with \( \min \) \\ \( \langle \Sloc,\bpi \rangle \) loc.\ homog. \\ \vspace{0.2cm}\({}^{(*)}\)} &
\multirow{6}{2.8cm}{\centering \vspace{-0.2cm} \\ \( \Wr^{\Sloc,\bpi}_{\delta \in \Delta} \Sym(N_\delta) \) \\ \vspace{0.2cm} \( \Delta \) linear \\ \( \langle \Sloc,\bpi \rangle \) loc.\ homog. \\ \vspace{0.2cm}\({}^{(*)}\)} &
\multirow{6}{3.3cm}{\centering \vspace{-0.2cm} \\ \( \Wr^{\Sloc,\bpi}_{\delta \in \Delta} \Sym(N_\delta) \) \\ \vspace{0.2cm} \( \langle \Sloc, \bpi \rangle \) full and \\ locally homogeneous \\ \vspace{0.2cm}\({}^{(\dagger)}\)} &
\multirow{6}{2.5cm}{\centering \vspace{0.4cm} \\  \( \Wr^{\Sloc,\bpi}_{\delta \in \Delta} \Sym(N_\delta) \) 
\\ \vspace{0.6cm}\( {}^{(\sharp)} \) 
} \bigstrut \\
&
&
&
&
\\
&
&
&
&
\\
&
&
&
&
\\
&
&
&
&
\\
&
&
&
&
\\
\hline
\end{tabular}%
}
\end{center}

\begin{footnotesize}
\begin{enumerate}
\item [\({}^{(*)}\)] 
When \( \Delta \) is linear, \( \Sloc \) and \( \langle \Sloc,\bpi \rangle \) are automatically full.
\item [\( {}^{(\dagger)}\)] 
One can add the requirement that the skeleton \( \langle \Delta,N \rangle \) be rigid.
\item [\( {}^{(\ddagger)}\)] 
We can require that the system 
of projections \( \langle \Sloc^{\LF}, \bpi \rangle \) 
has finite character, and also that it is locally homogeneous.
\item [\( {}^{(\sharp)}\)]
We can further require local homogeneity.
\end{enumerate}
\end{footnotesize}


Classes of groups in the same column of the above table are equivalent up to (topological) isomorphism.
For example, the first three lines of the second column assert that the isometry groups of homogeneous Polish ultrametric spaces coincide with the automorphism groups of homogeneous countable pruned \( L \)-trees, and they are exactly the wreath products with locally finite supports  over countable linear skeletons of full permutation groups. More in detail: 
the first column comes from Theorem~\ref{thm:homogeneousdiscrete} and Remark~\ref{rmk:homogeneousdiscrete}; 
the second column corresponds to Theorem~\ref{thm:homogeneous} and Remark~\ref{rmk:homogeneous}; 
the third column summarizes Theorem~\ref{thm:exactdistancesnew} and Remark~\ref{rmk:exactdistancesnew}; 
finally, the fourth column conveys the content of Theorem~\ref{thm:generalcase} and Remark~\ref{rmk:generalcase}.

\end{document}